\documentclass[12pt]{article}
\usepackage[english]{babel}
\usepackage{amsthm,amsfonts, amsbsy, amssymb,amsmath,graphicx}
\usepackage{graphics}

\def\thtext#1{
\catcode`@=11 \gdef\@thmcountersep{. #1}
\catcode`@=12}

\renewcommand{\baselinestretch}{1.2}

\newtheorem{theorem}{Theorem}[section]
\newtheorem{lemma}{Lemma}[section]
\newtheorem{crl}{Corollary}[section]

\newcounter{imn1}[section]
\renewcommand{\thtext}
{\thesection.\arabic{imn1}}

\newcounter{imn2}[section]
\renewcommand{\thtext}
{\thesection.\arabic{imn2}}

\newcounter{imn3}[section]
\renewcommand{\thtext}
{\thesection.\arabic{imn3}}

 \newenvironment{definition}{\trivlist \item[\hskip\labelsep{\bf Definition}]
 \refstepcounter{imn1}{\bf\thesection.\arabic{imn1}.}}%
 {\endtrivlist}

 \newenvironment{rk}{\trivlist \item[\hskip\labelsep{\bf Remark}]
 \refstepcounter{imn2}{\bf\thesection.\arabic{imn2}.}}%
 {\endtrivlist}

 \newenvironment{exa}{\trivlist \item[\hskip\labelsep{\bf Example}]
 \refstepcounter{imn3}{\bf\thesection.\arabic{imn3}.}}%
 {\endtrivlist}

\def\Z{{\mathbb Z}}

\newcommand{\ob}{\mathrm{ob}}
\newcommand{\V}{{\mathcal V}}

\newcommand{\firstn}{\raisebox{-0.35\height}{\includegraphics[width=0.7cm]{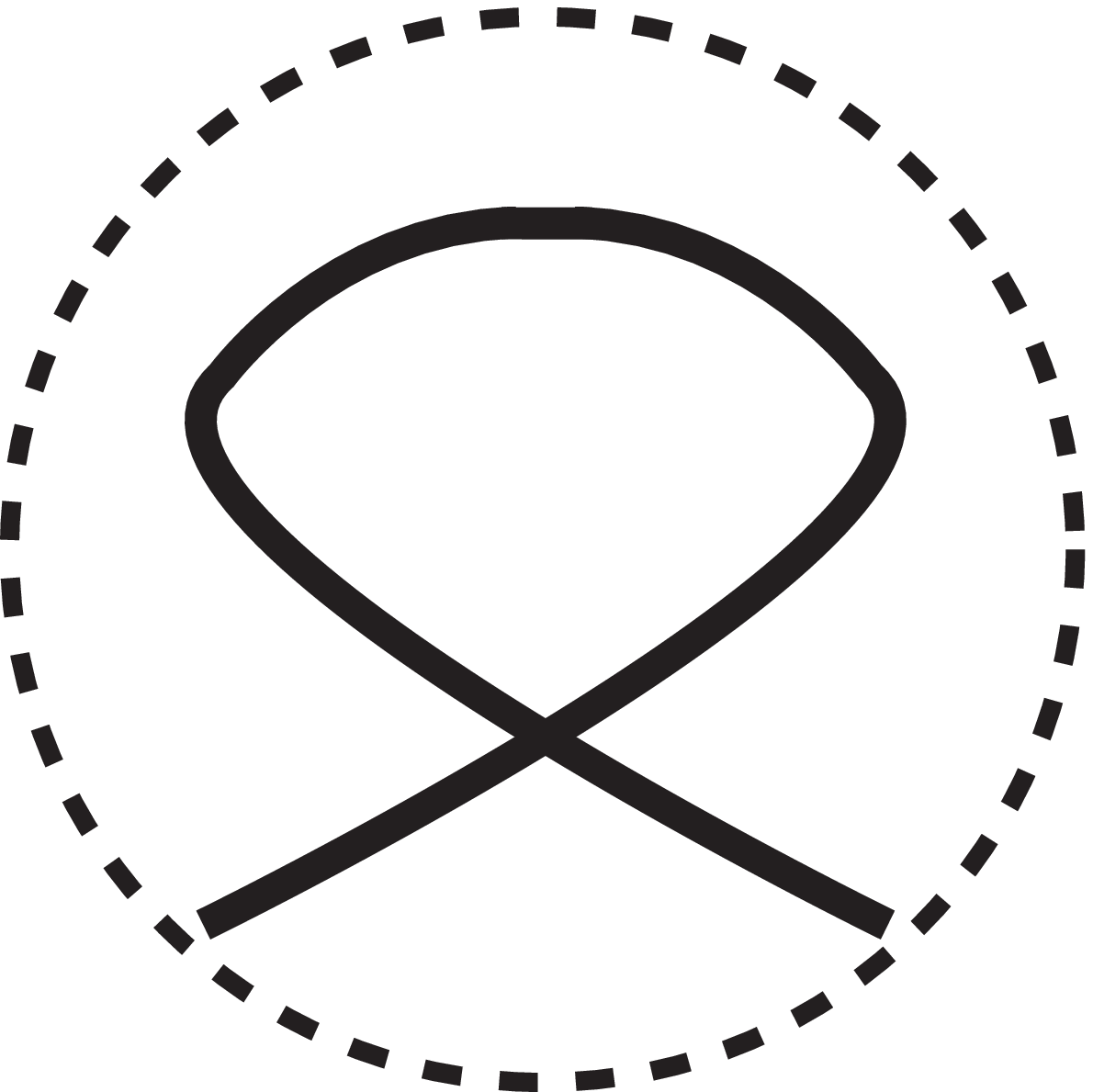}}}
\newcommand{\firstfio}{\raisebox{-0.35\height}{\includegraphics[width=0.7cm]{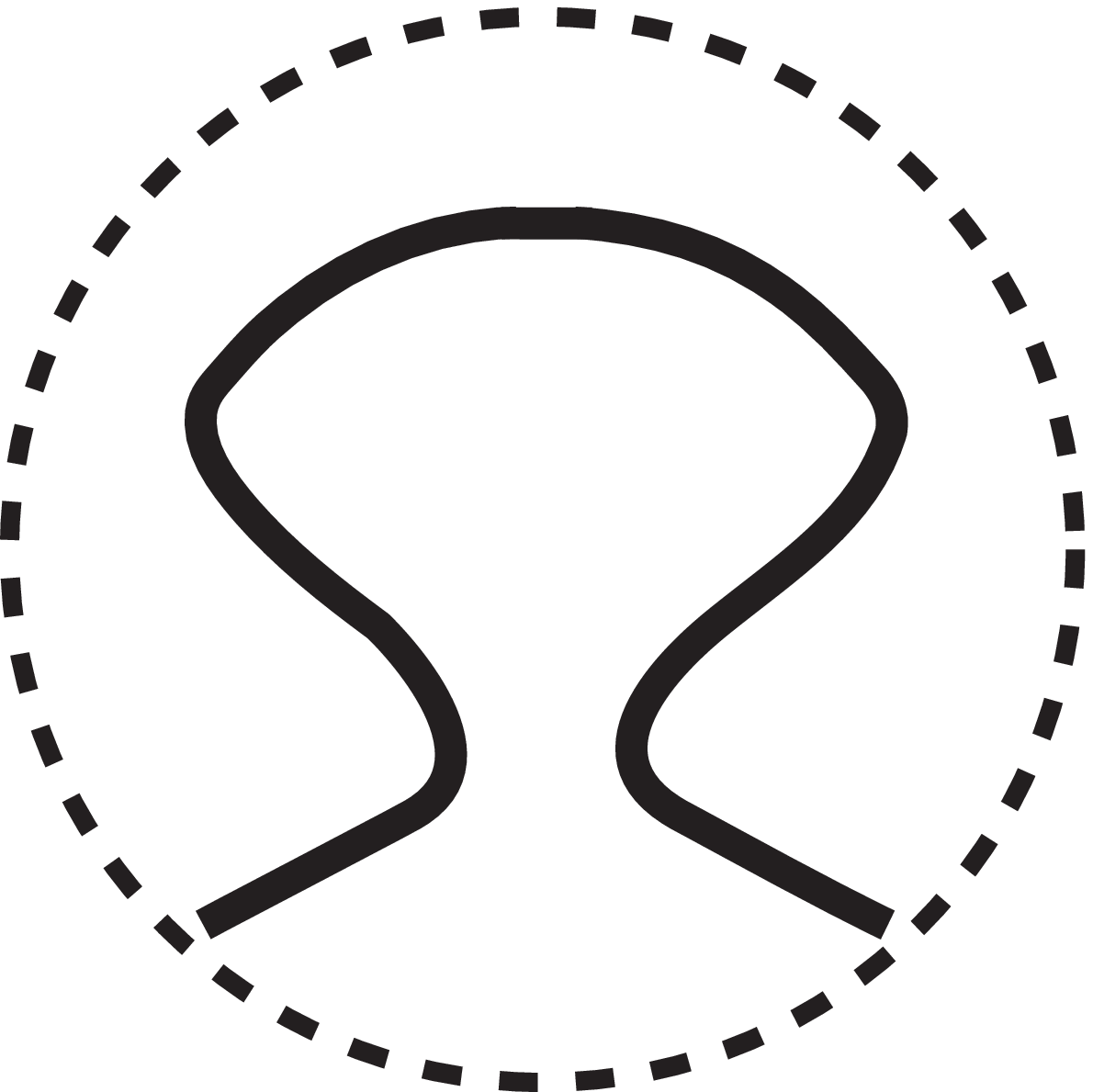}}}
\newcommand{\firstfit}{\raisebox{-0.35\height}{\includegraphics[width=0.7cm]{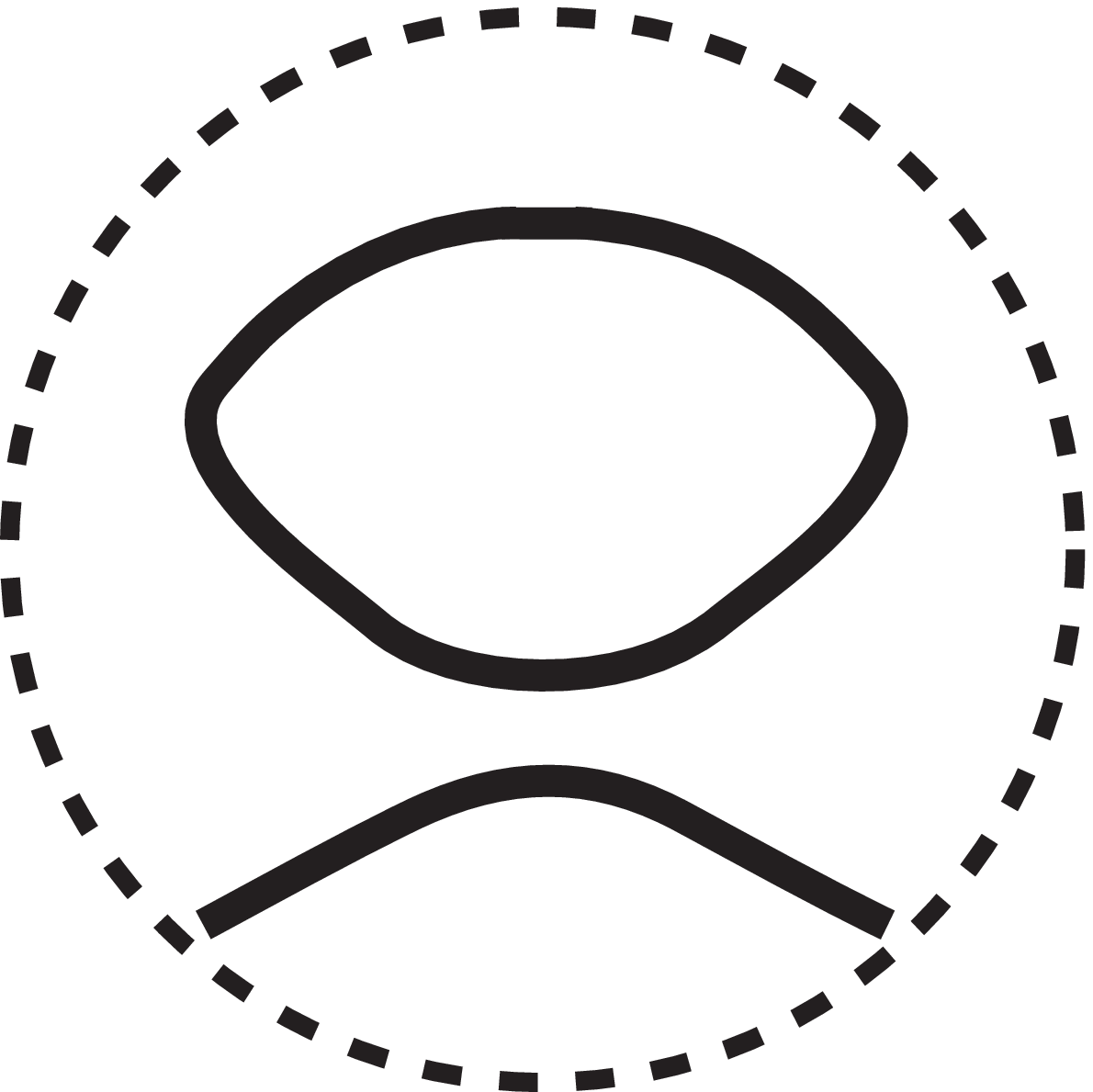}}}

\newcommand{\sectwon}{\raisebox{-0.35\height}{\includegraphics[width=0.7cm]{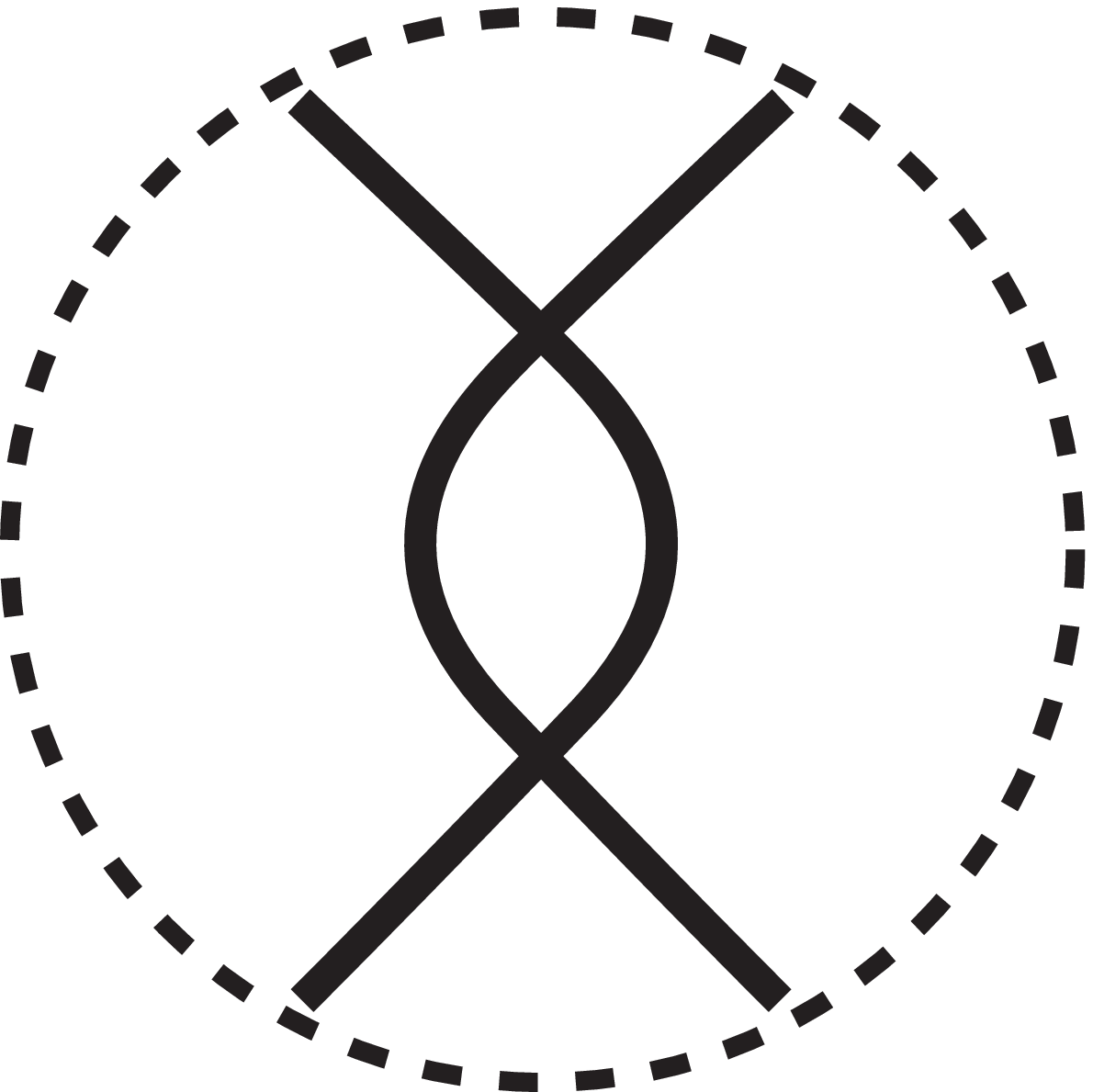}}}
\newcommand{\secfino}{\raisebox{-0.35\height}{\includegraphics[width=0.7cm]{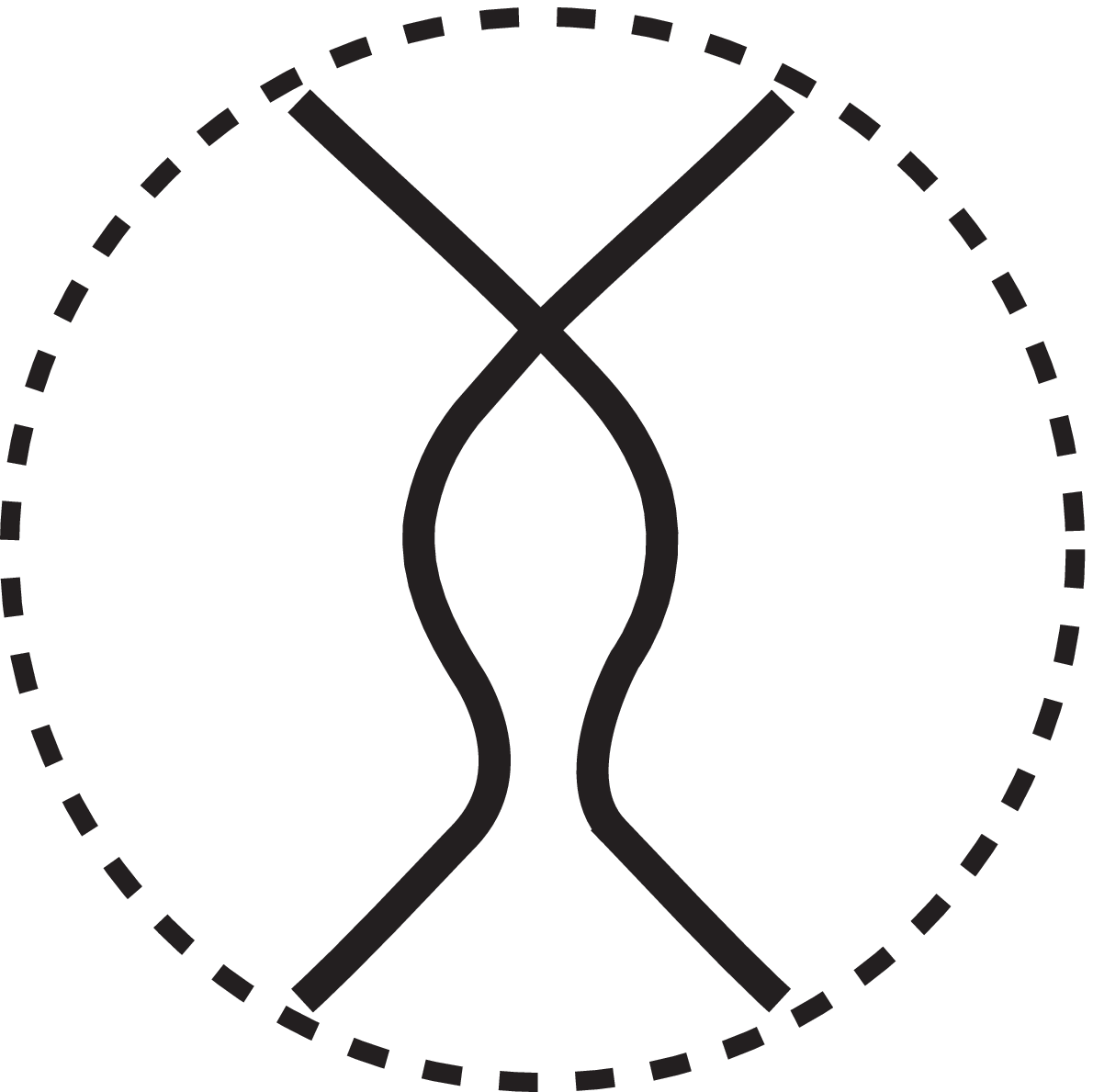}}}
\newcommand{\secsecno}{\raisebox{-0.35\height}{\includegraphics[width=0.7cm]{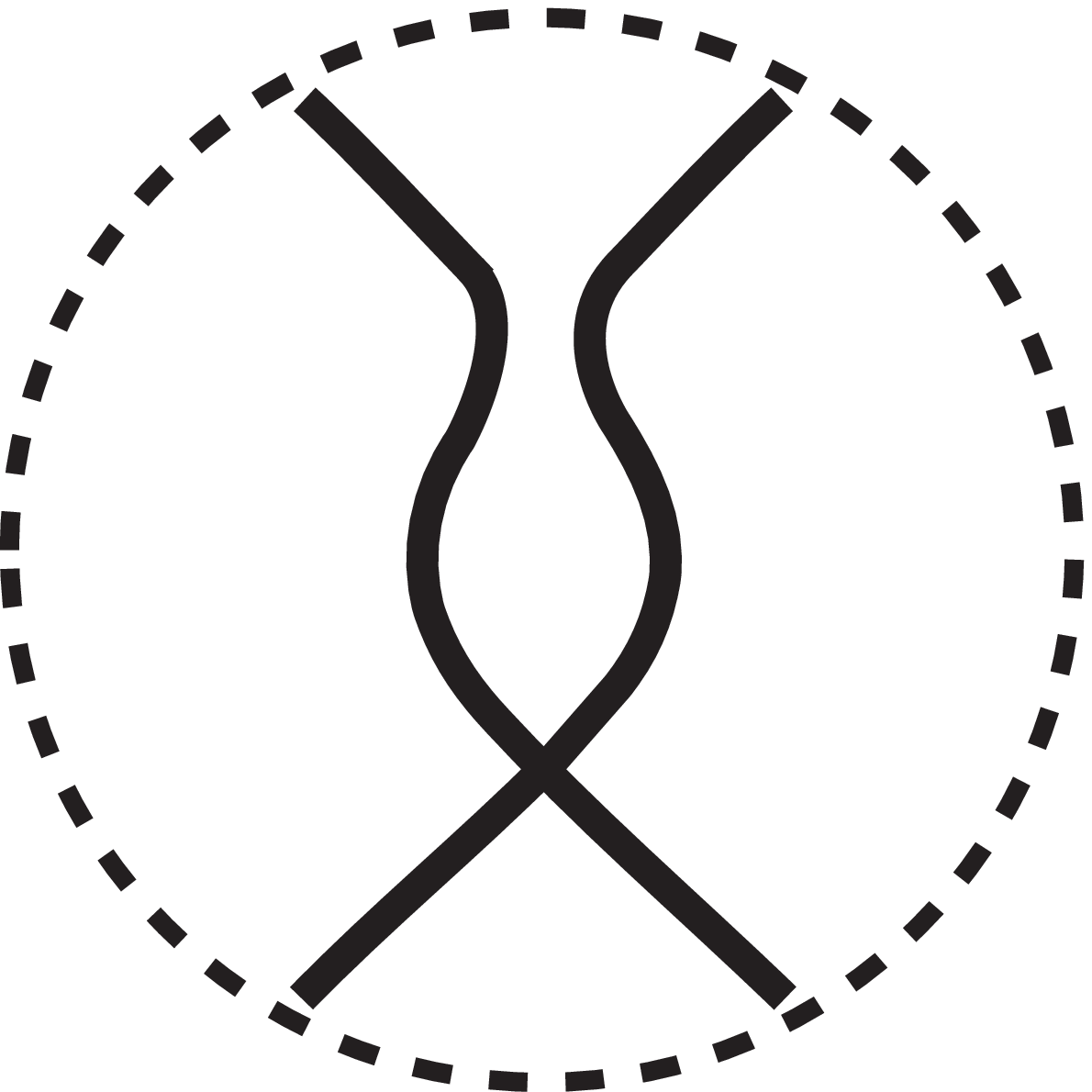}}}
\newcommand{\secfint}{\raisebox{-0.35\height}{\includegraphics[width=0.7cm]{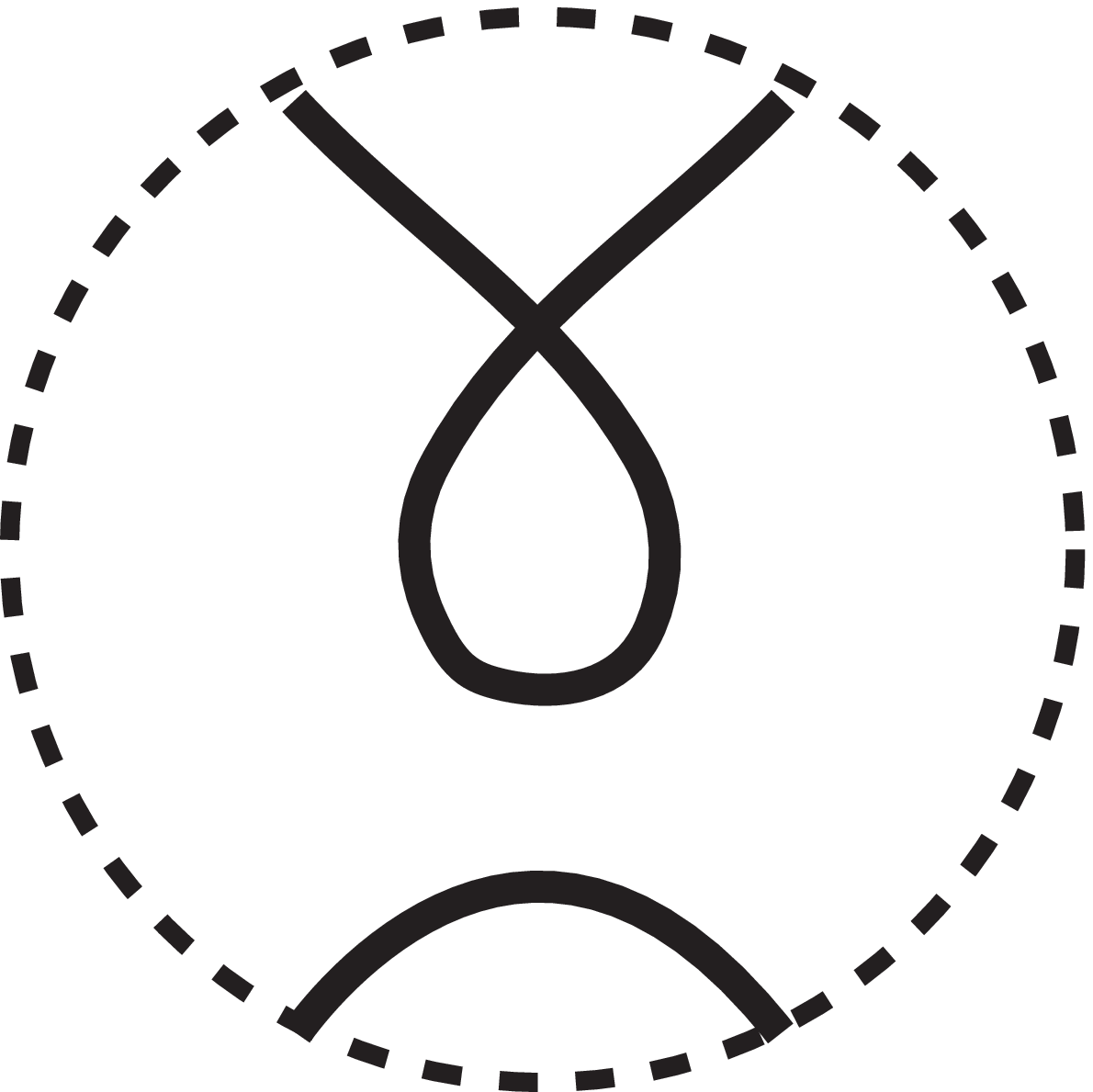}}}
\newcommand{\secsecnt}{\raisebox{-0.35\height}{\includegraphics[width=0.7cm]{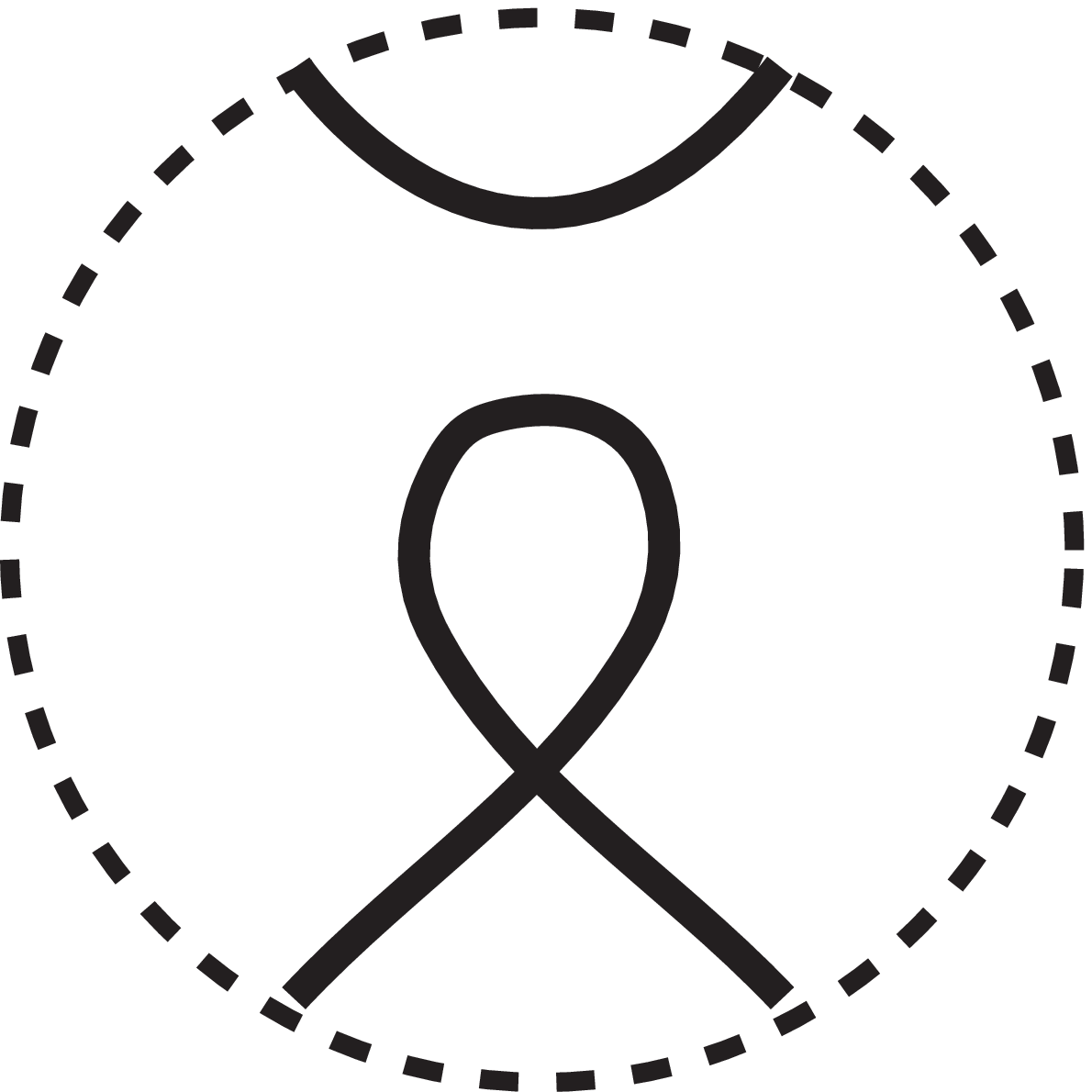}}}
\newcommand{\secfioseco}{\raisebox{-0.35\height}{\includegraphics[width=0.7cm]{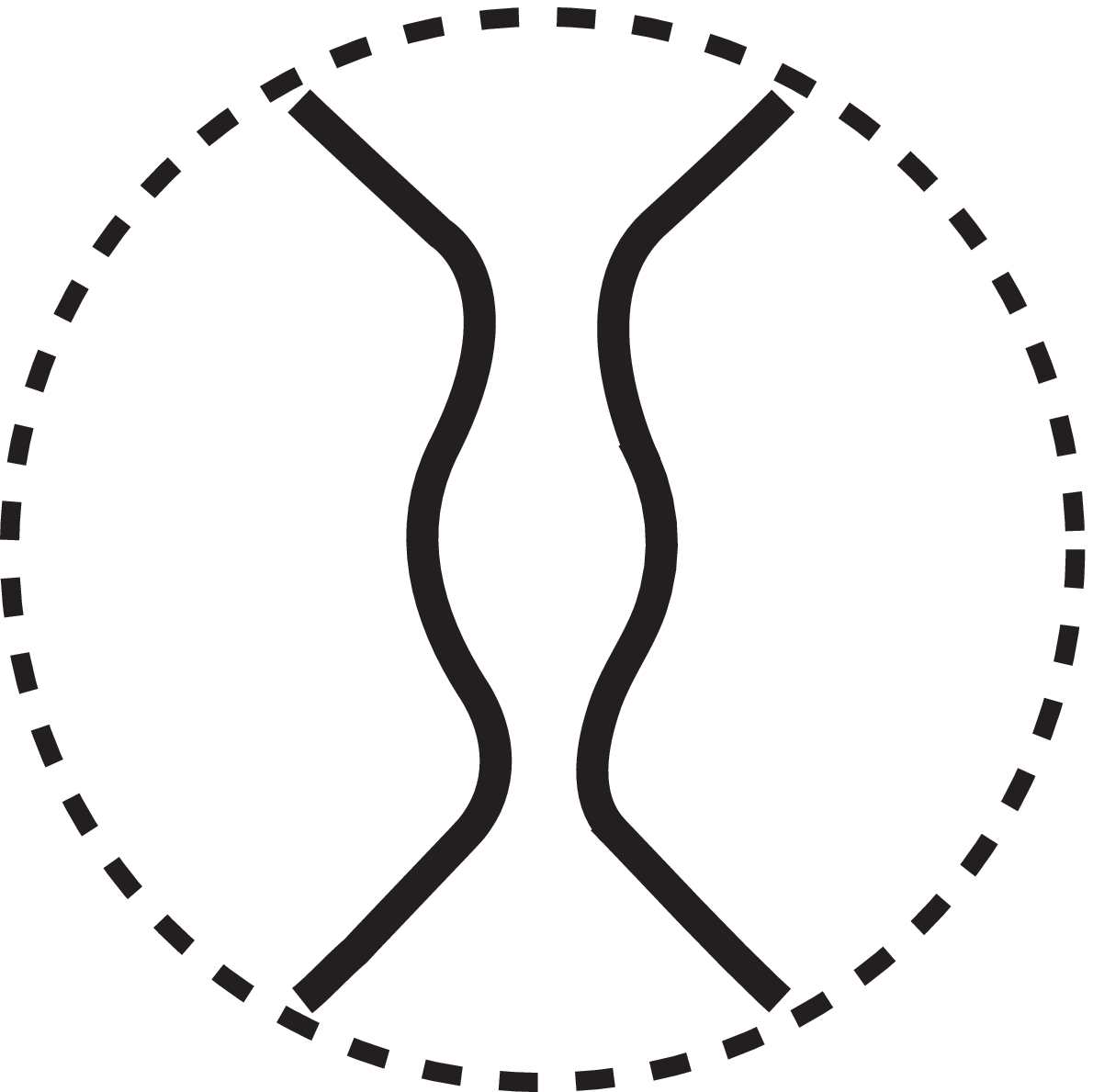}}}
\newcommand{\secfiosect}{\raisebox{-0.35\height}{\includegraphics[width=0.7cm]{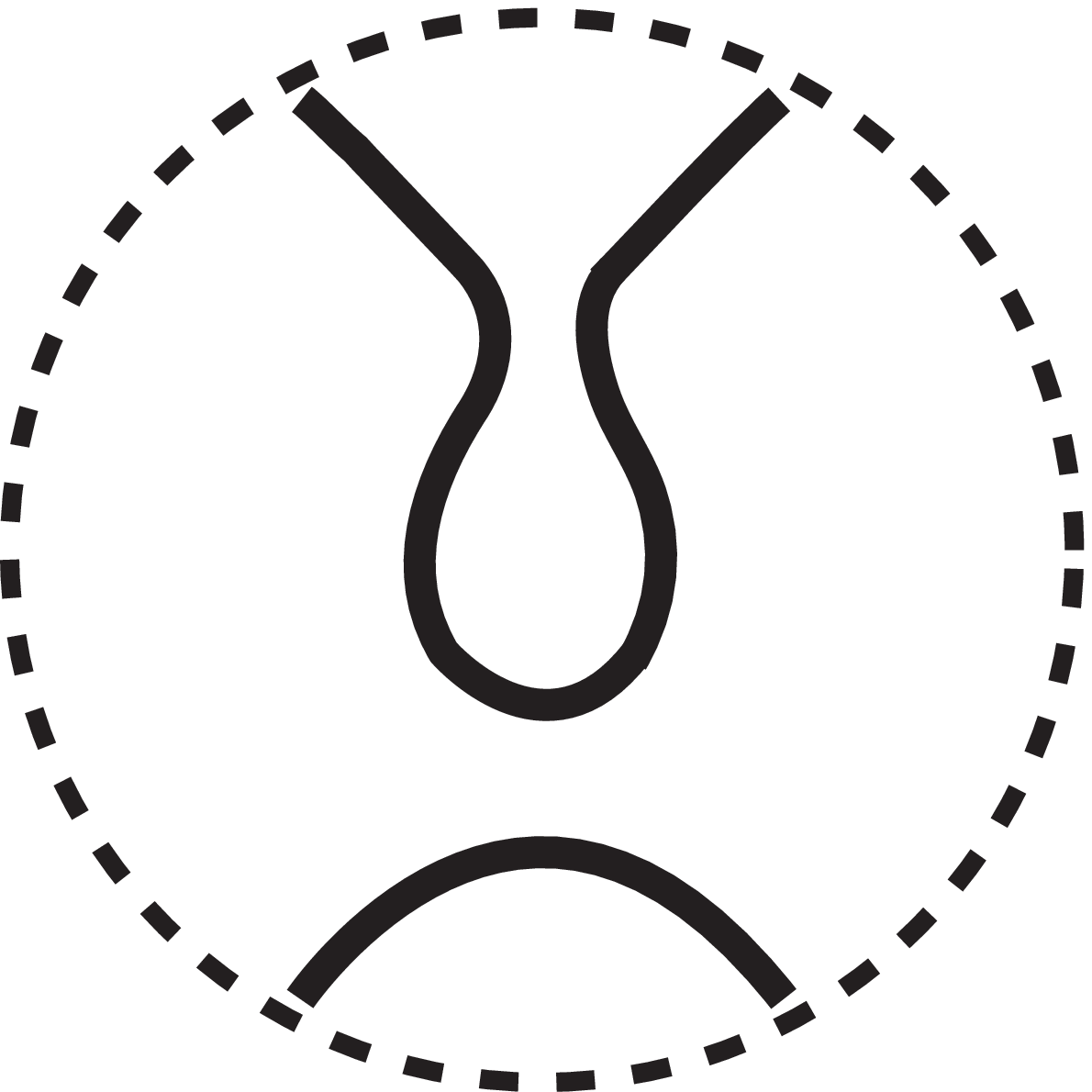}}}
\newcommand{\secfitseco}{\raisebox{-0.35\height}{\includegraphics[width=0.7cm]{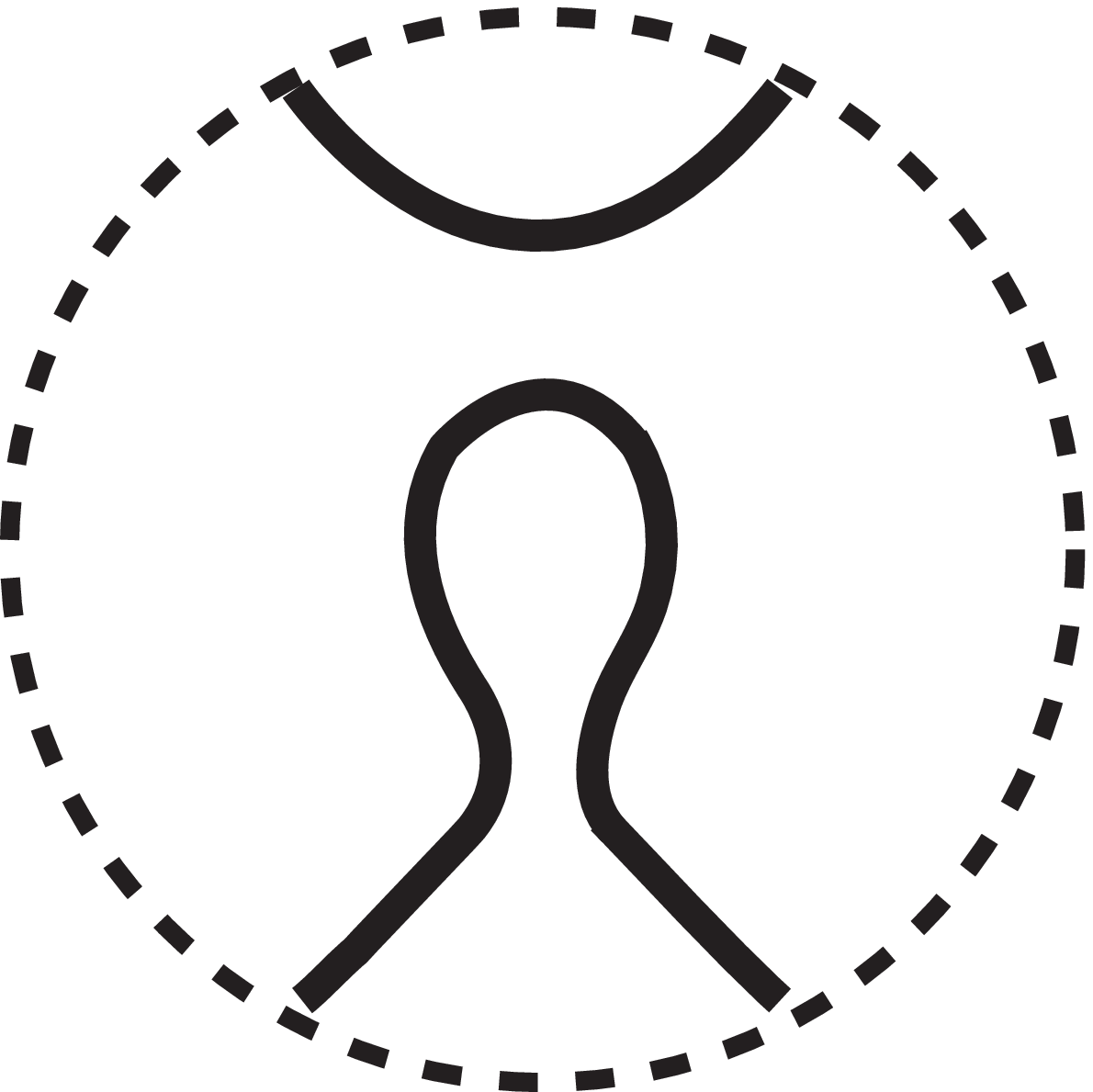}}}
\newcommand{\secfitsect}{\raisebox{-0.35\height}{\includegraphics[width=0.7cm]{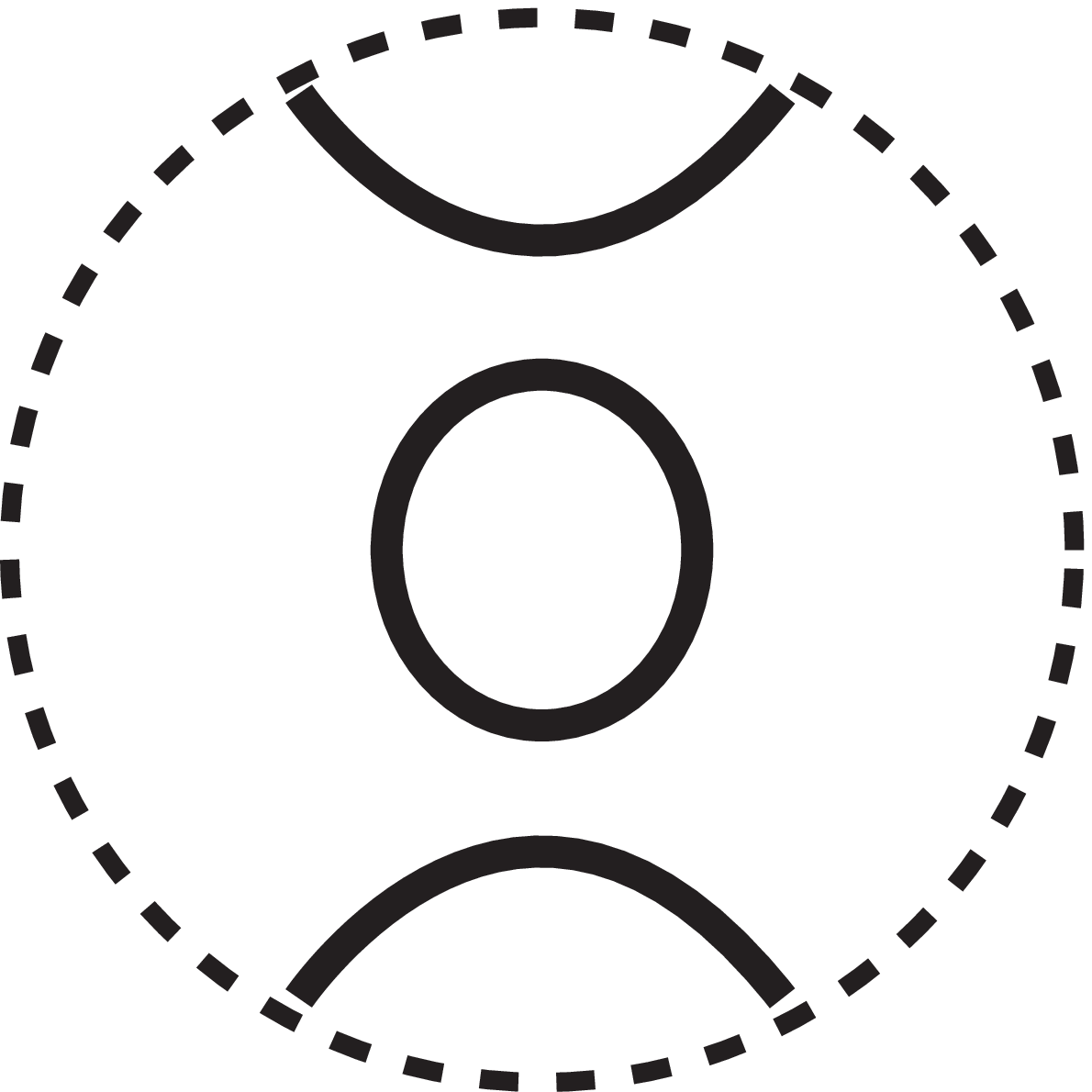}}}

\newcommand{\thirdn}{\raisebox{-0.35\height}{\includegraphics[width=0.7cm]{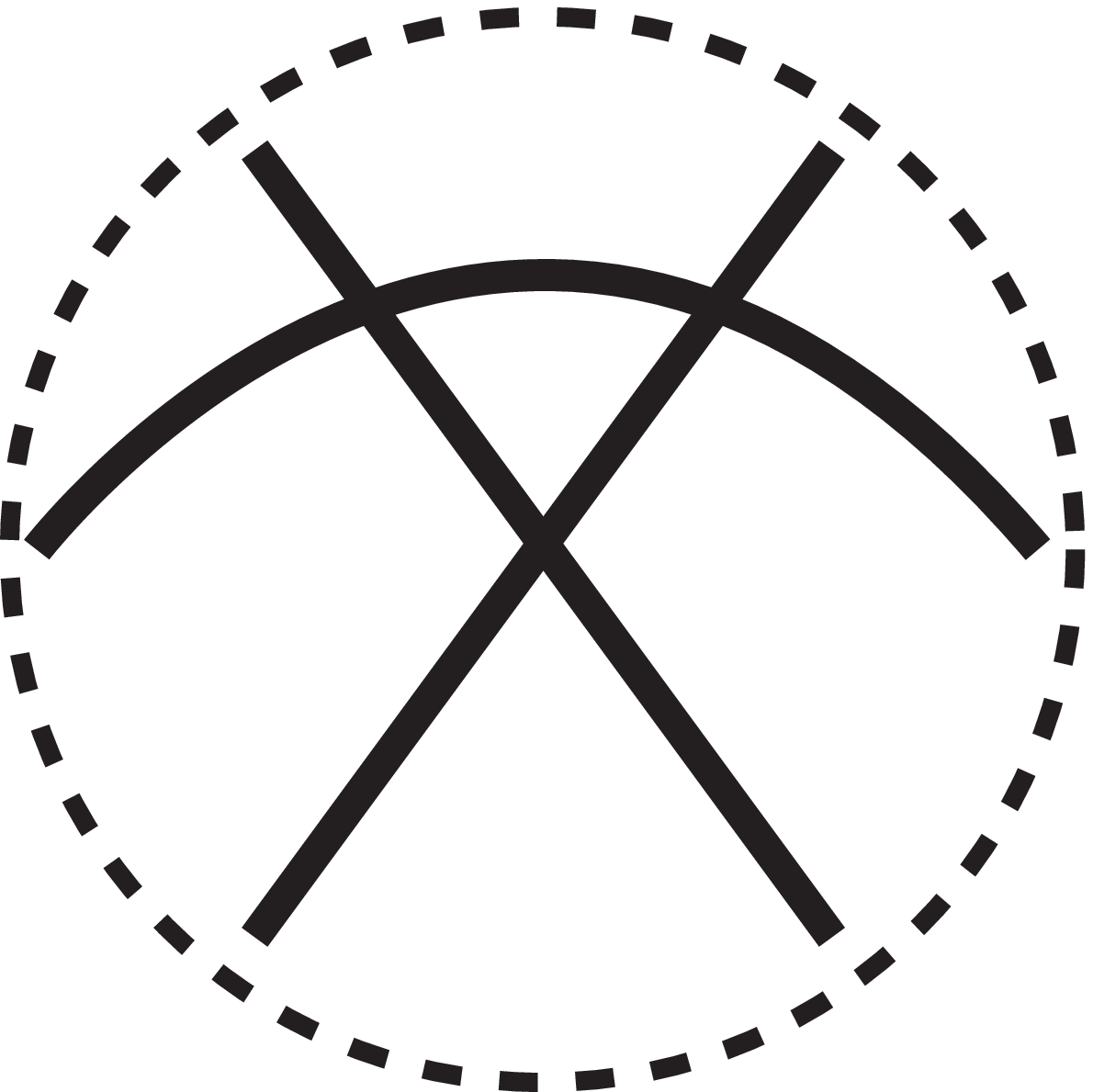}}}
\newcommand{\thirdfinsecntho}{\raisebox{-0.35\height}{\includegraphics[width=0.7cm]{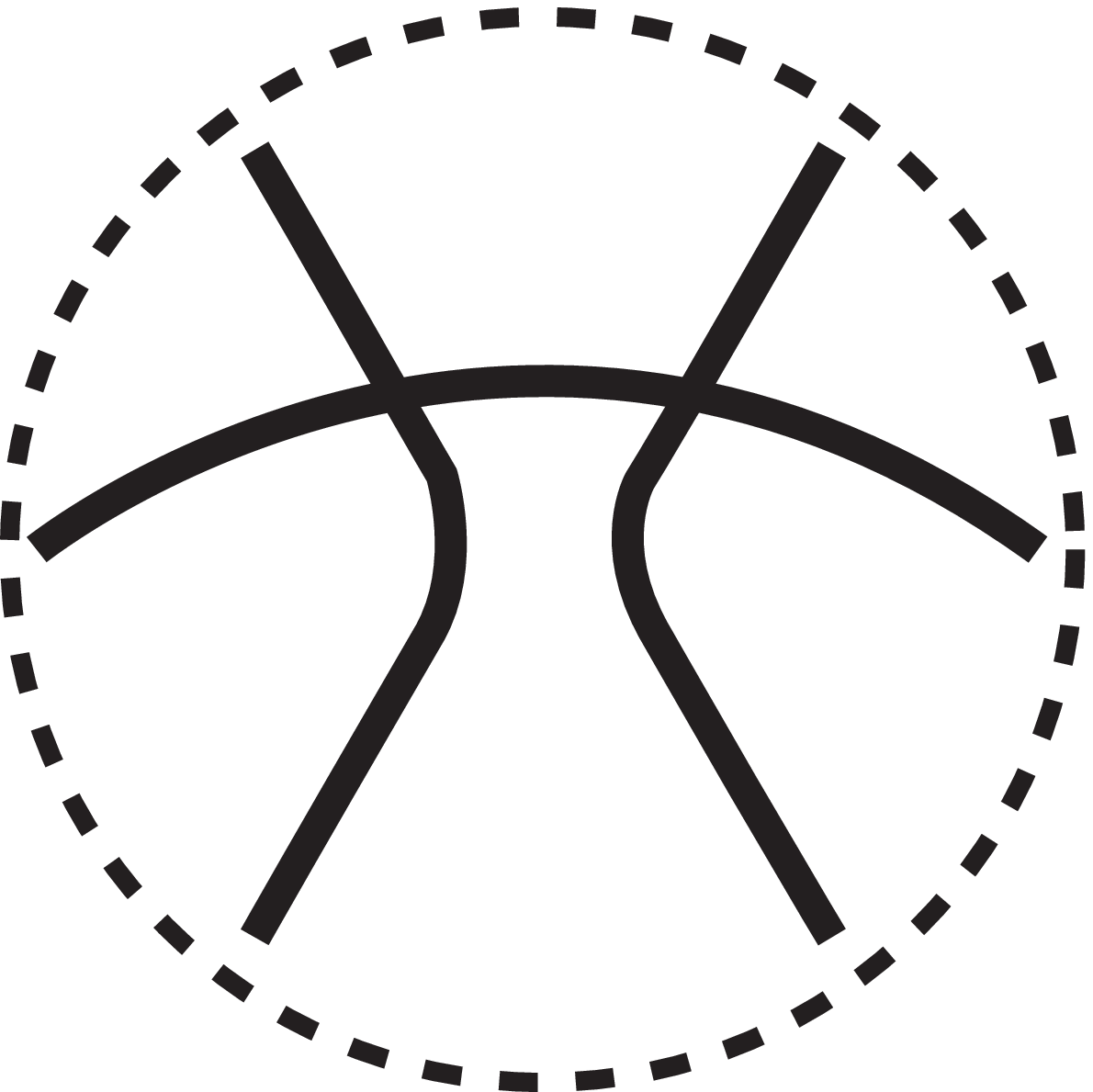}}}
\newcommand{\thirdfinsecntht}{\raisebox{-0.35\height}{\includegraphics[width=0.7cm]{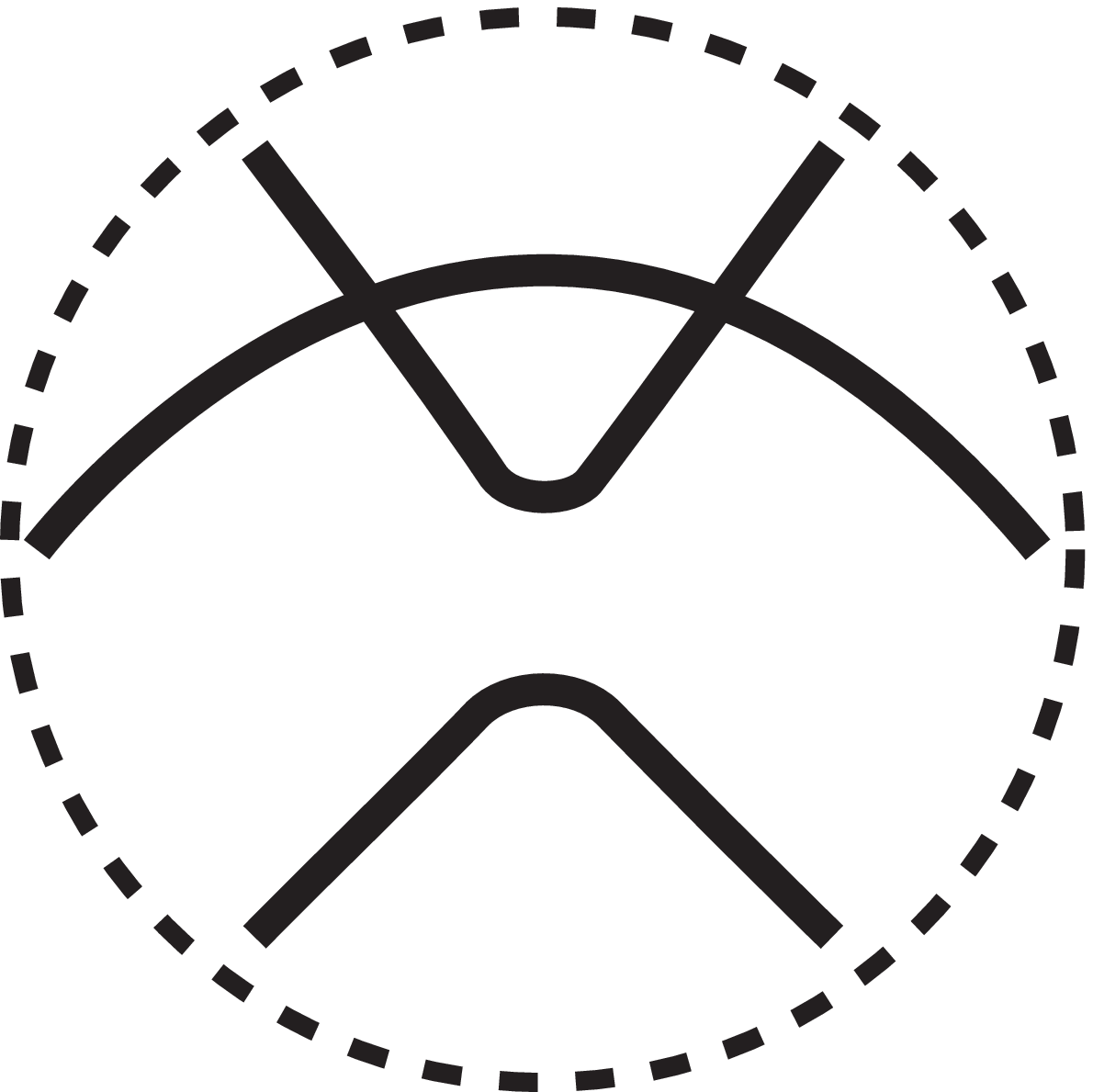}}}
\newcommand{\thirdfinsecothn}{\raisebox{-0.35\height}{\includegraphics[width=0.7cm]{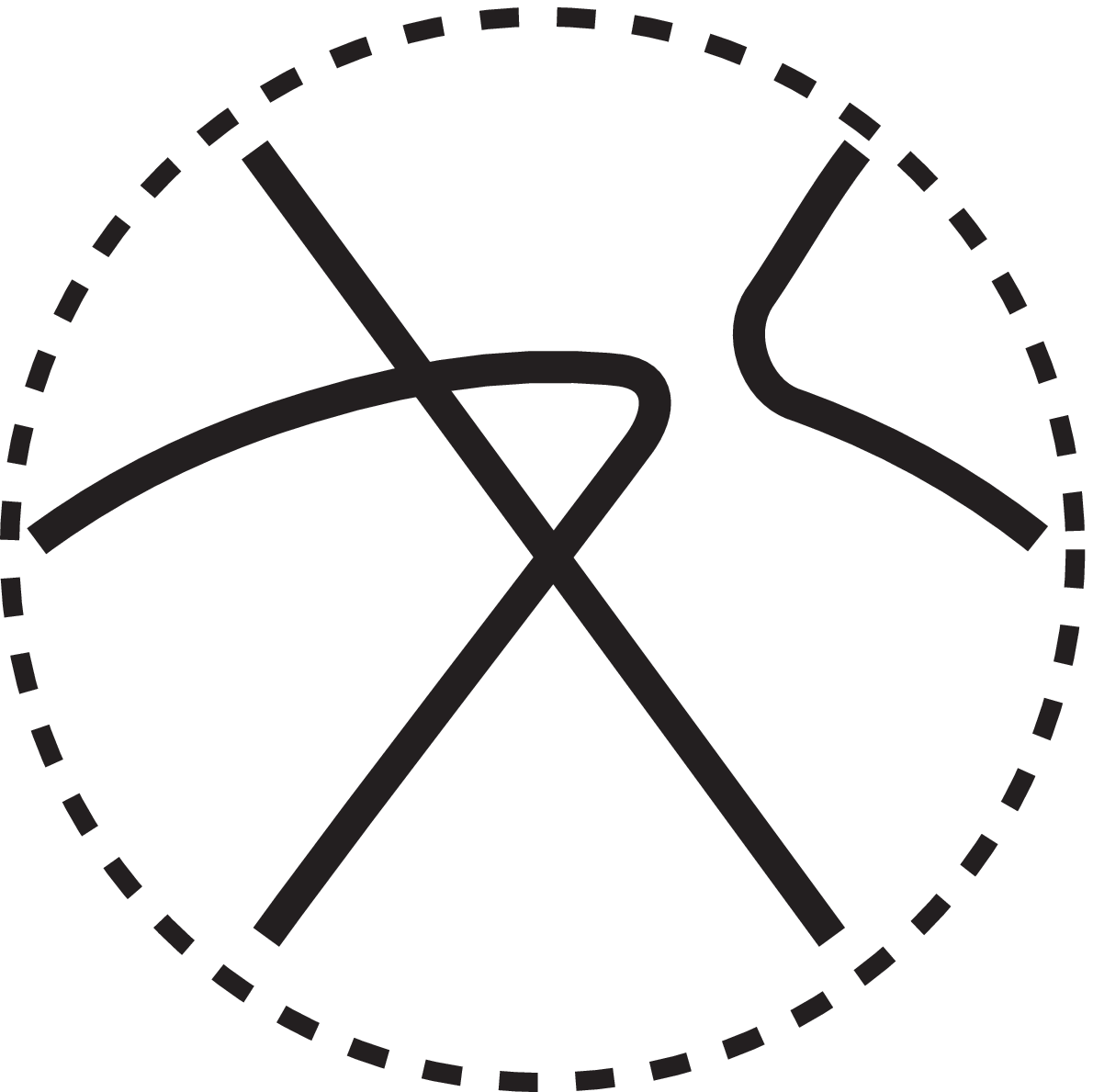}}}
\newcommand{\thirdfinsectthn}{\raisebox{-0.35\height}{\includegraphics[width=0.7cm]{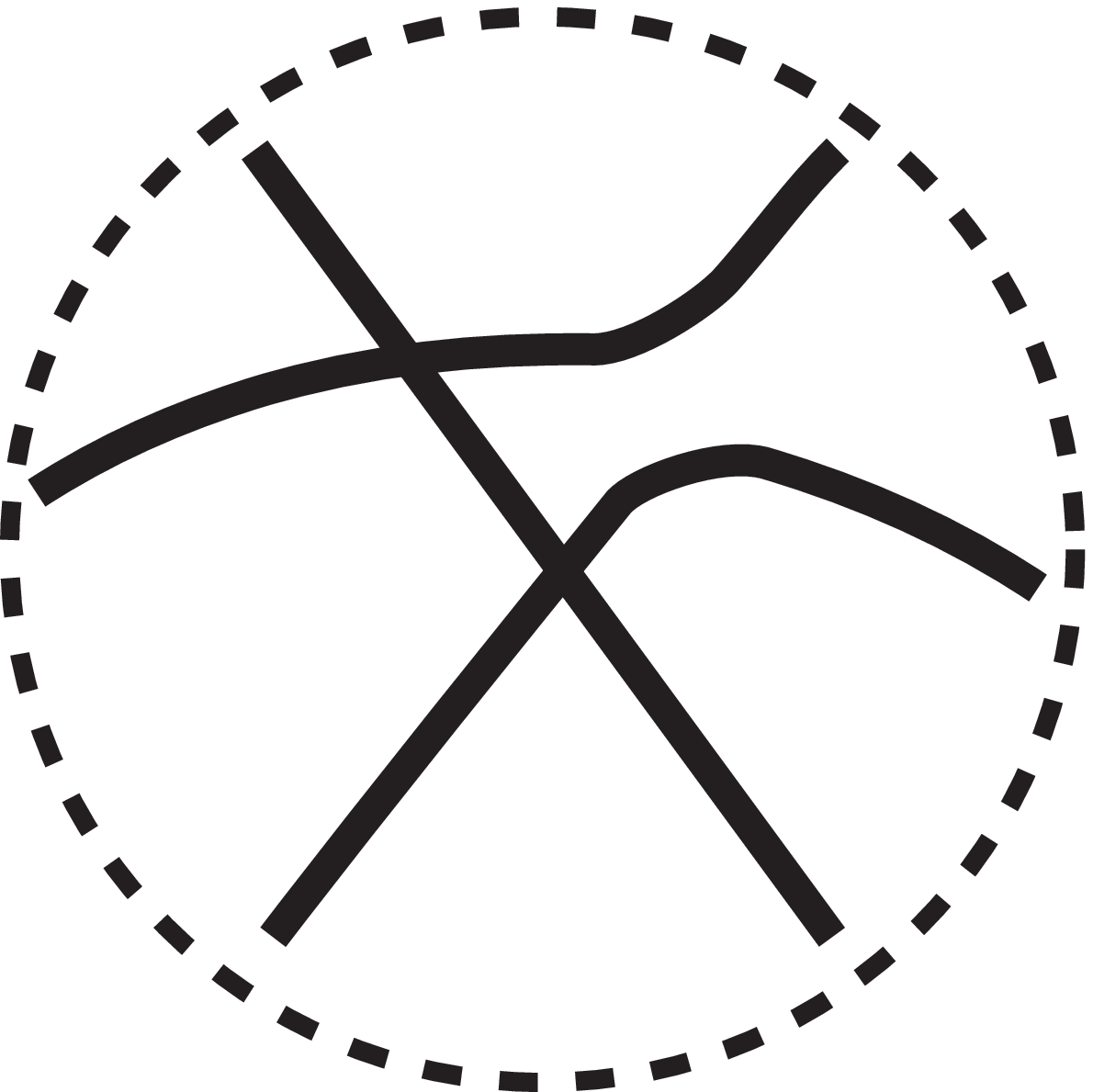}}}
\newcommand{\thirdfiosecnthn}{\raisebox{-0.35\height}{\includegraphics[width=0.7cm]{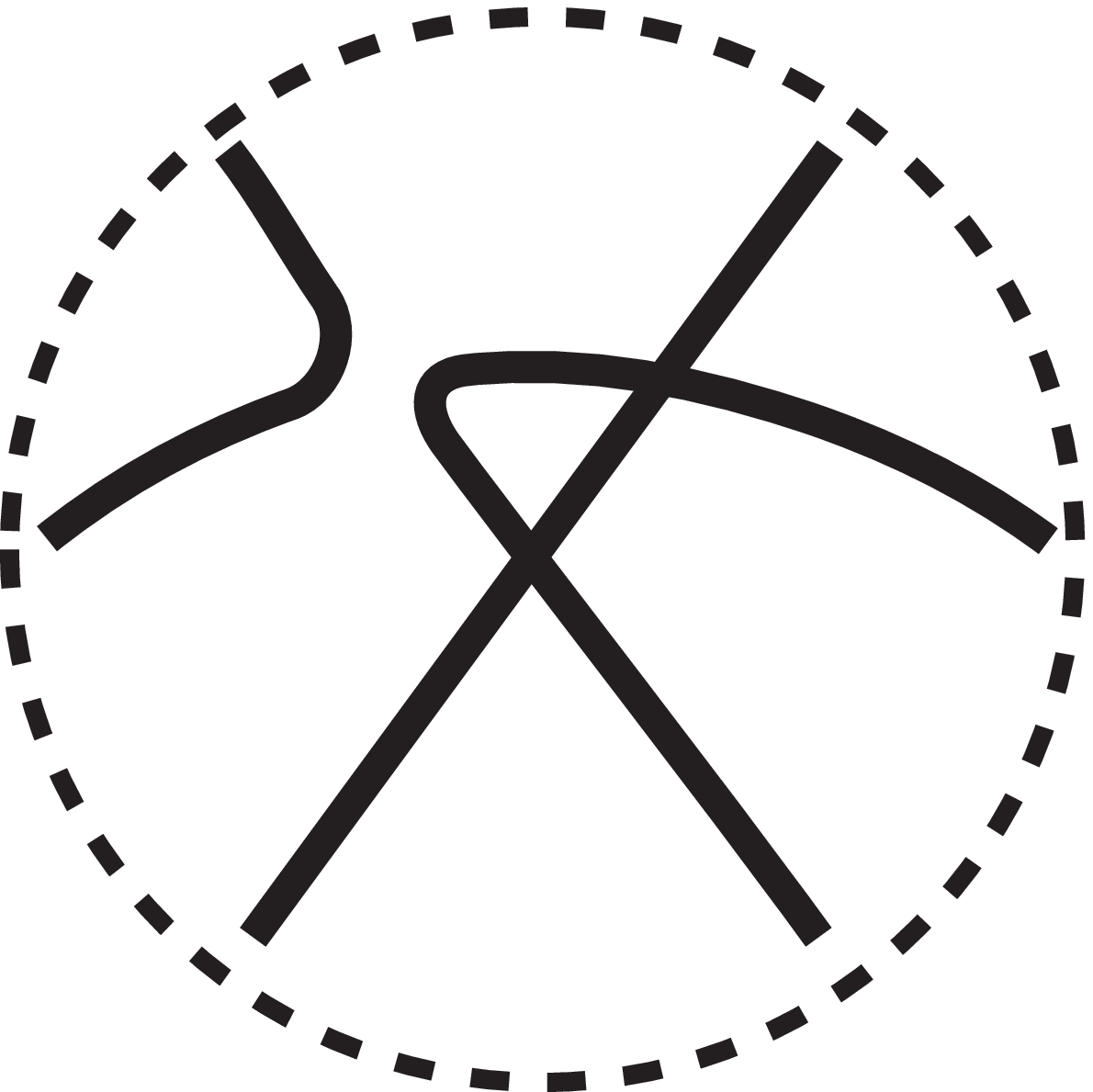}}}
\newcommand{\thirdfiosecothn}{\raisebox{-0.35\height}{\includegraphics[width=0.7cm]{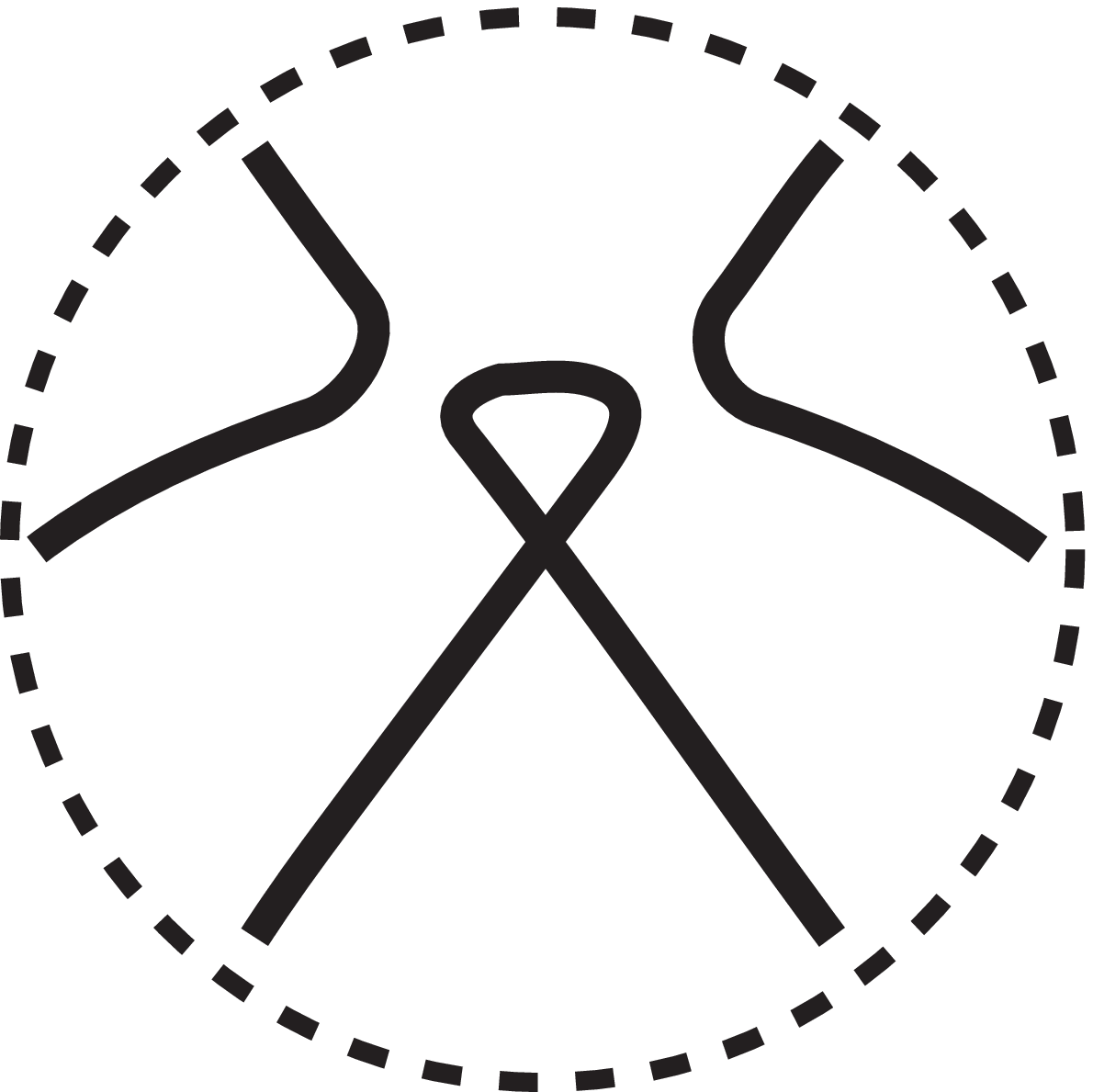}}}
\newcommand{\thirdfiosectthn}{\raisebox{-0.35\height}{\includegraphics[width=0.7cm]{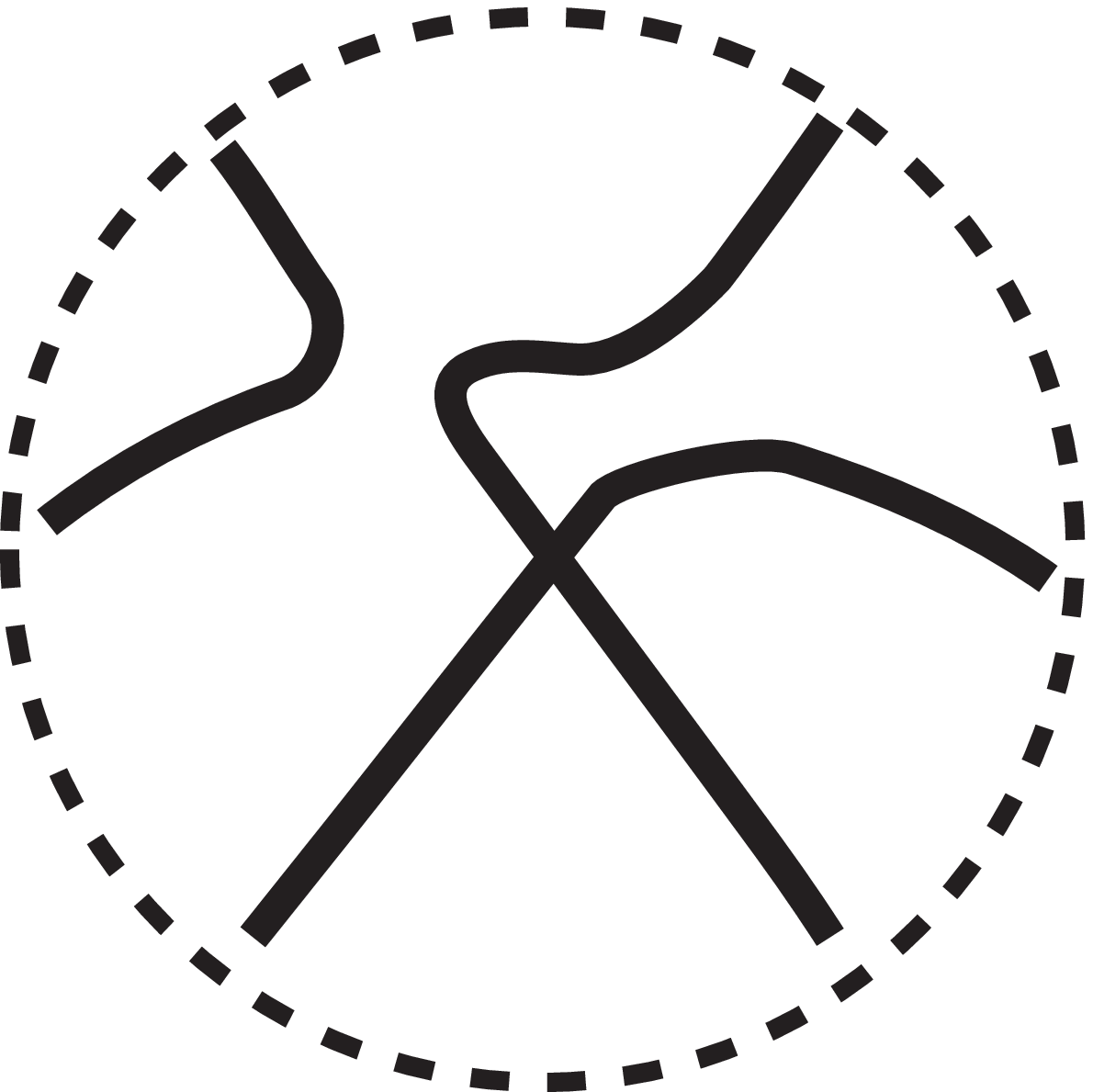}}}
\newcommand{\thirdfiosecttho}{\raisebox{-0.35\height}{\includegraphics[width=0.7cm]{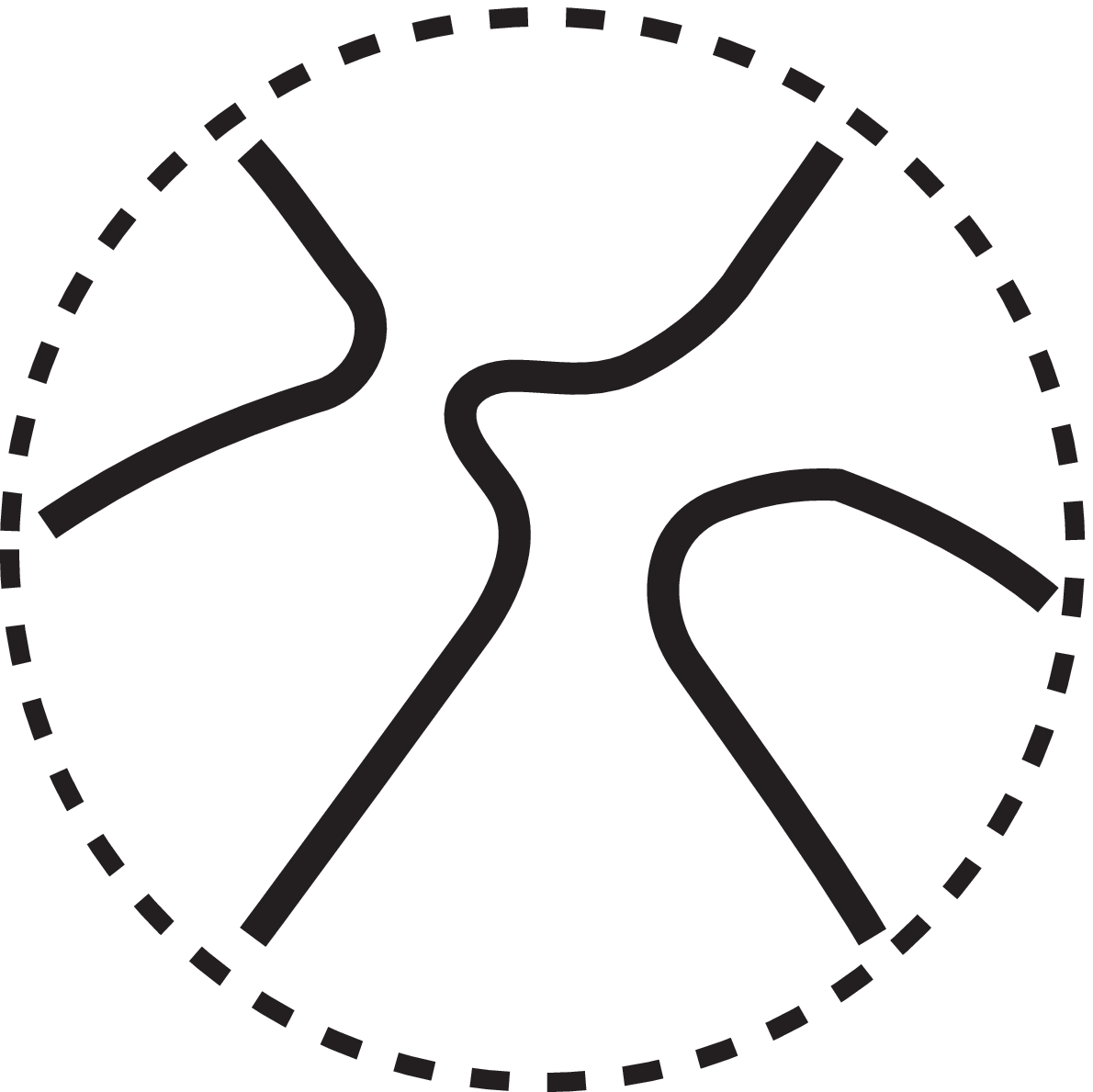}}}
\newcommand{\thirdfiosecttht}{\raisebox{-0.35\height}{\includegraphics[width=0.7cm]{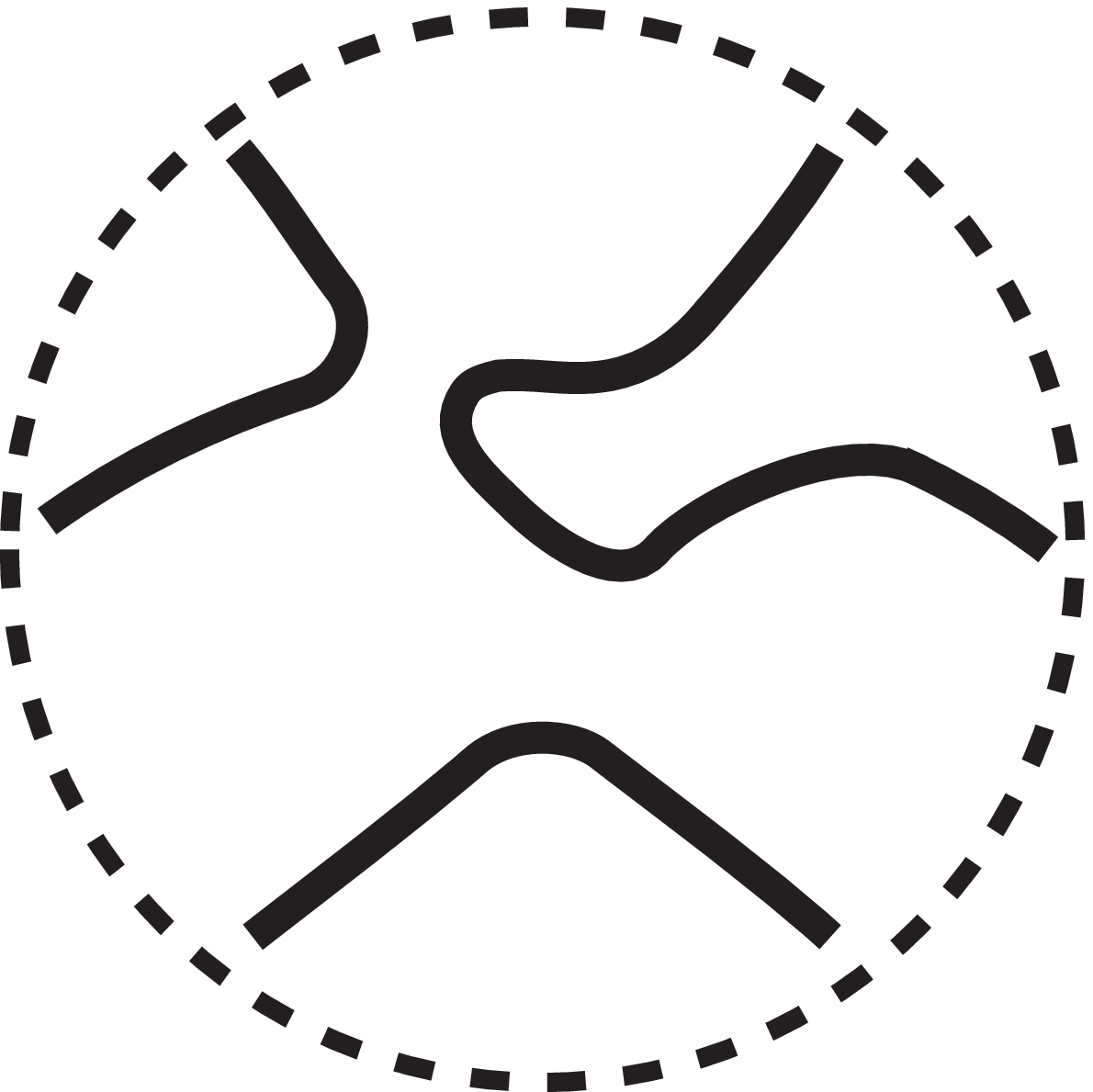}}}
\newcommand{\thirdfitsecnthn}{\raisebox{-0.35\height}{\includegraphics[width=0.7cm]{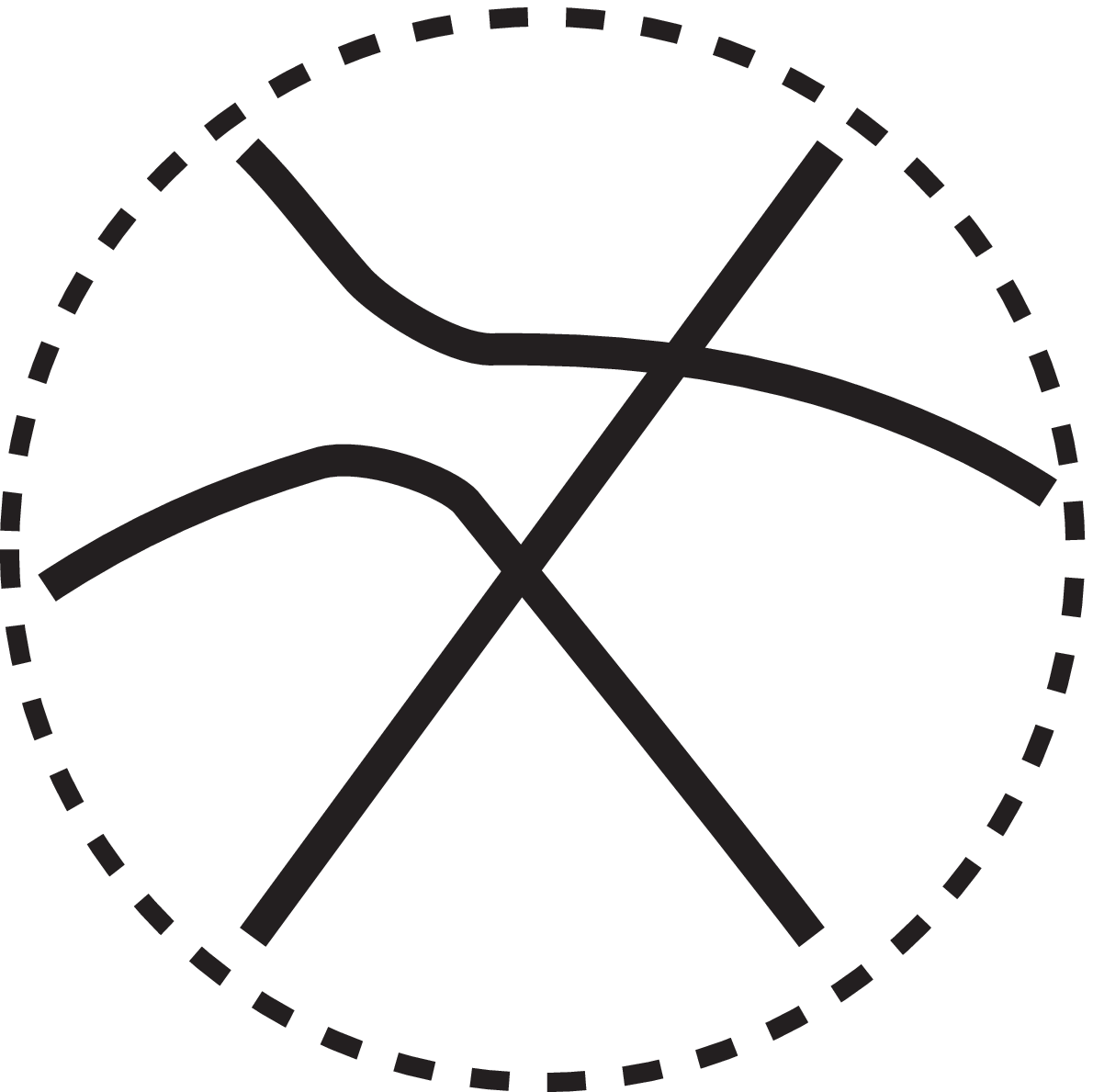}}}
\newcommand{\thirdfitsecothn}{\raisebox{-0.35\height}{\includegraphics[width=0.7cm]{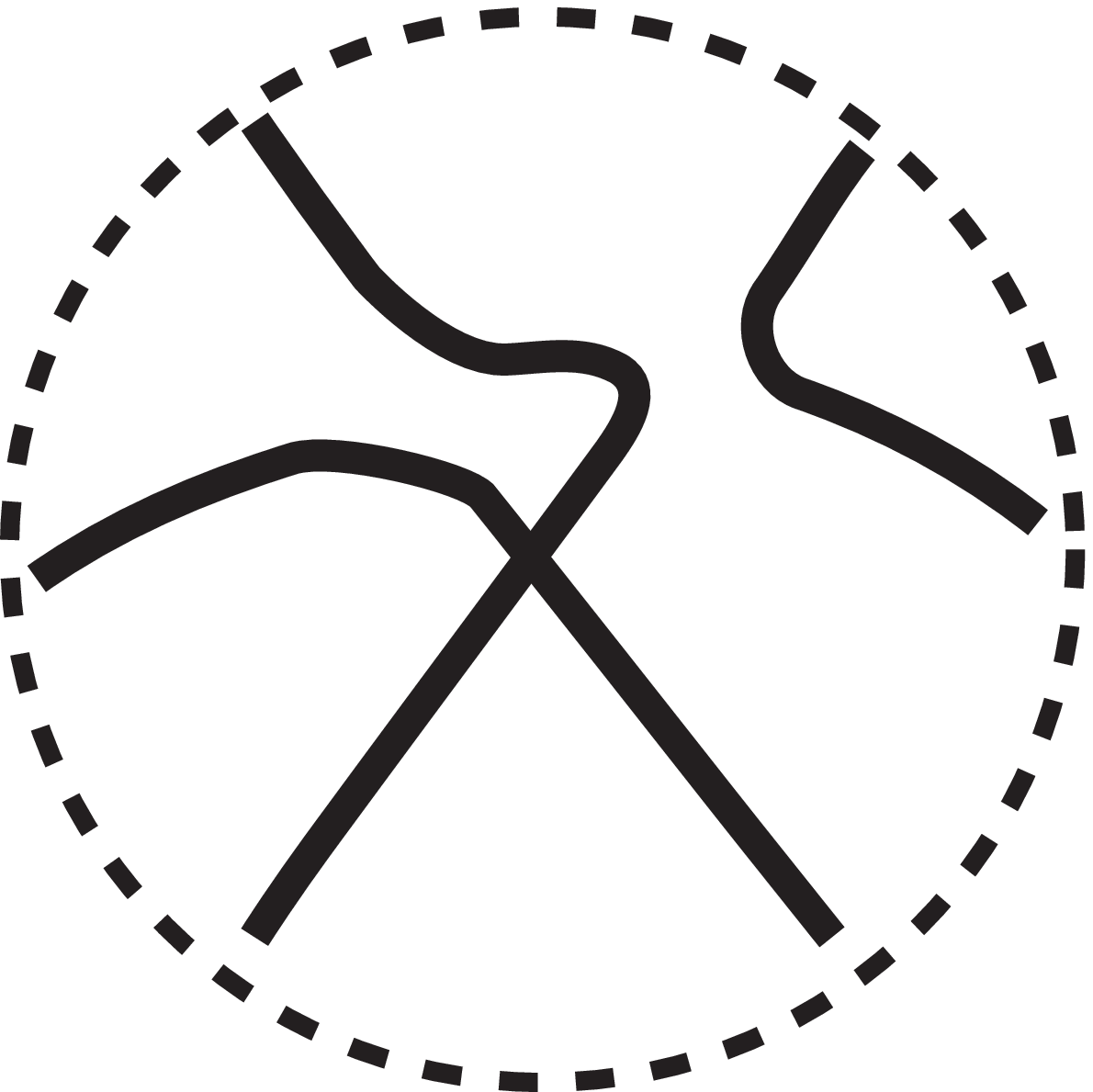}}}
\newcommand{\thirdfitsecotht}{\raisebox{-0.35\height}{\includegraphics[width=0.7cm]{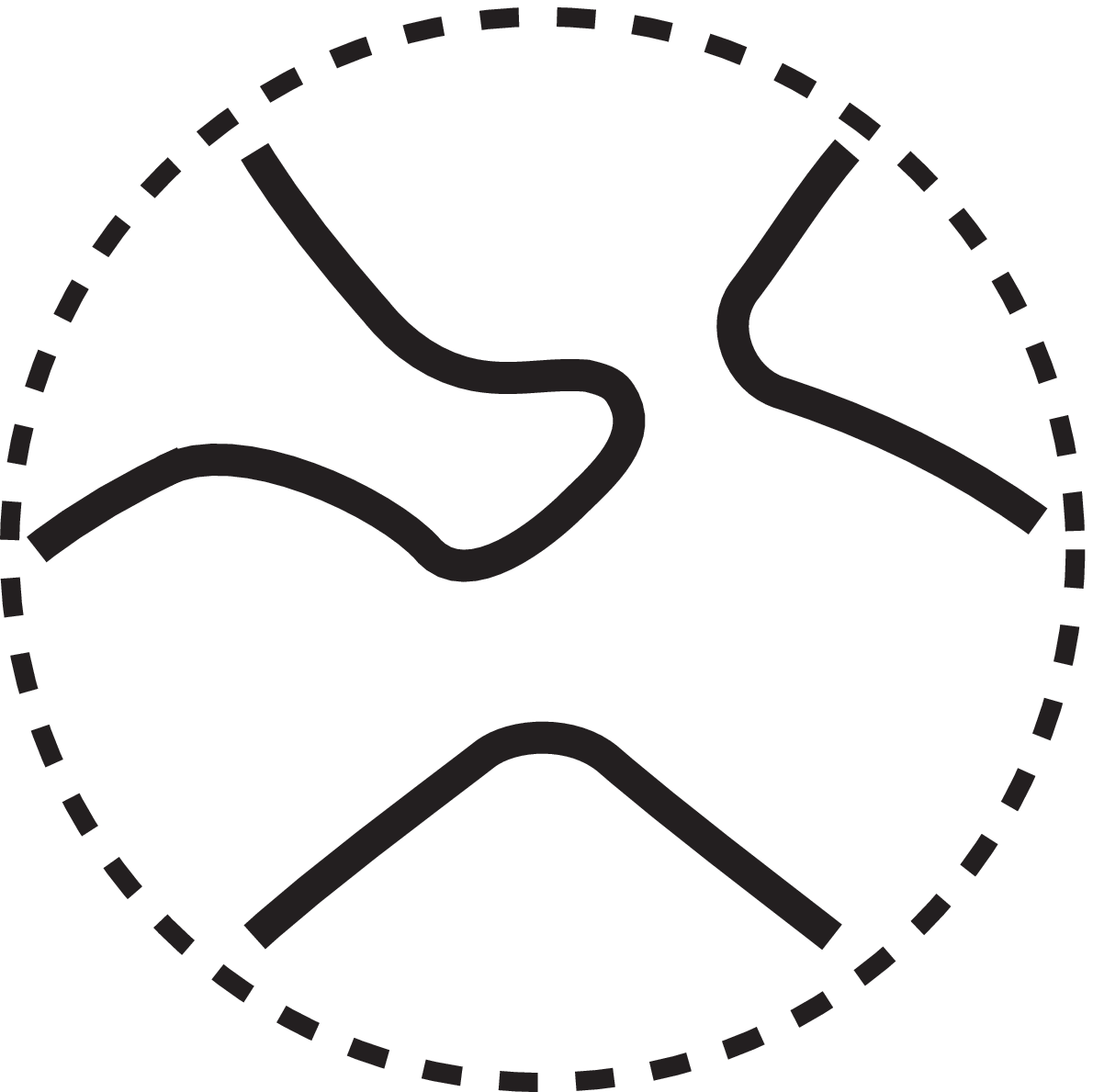}}}
\newcommand{\thirdfitsectthn}{\raisebox{-0.35\height}{\includegraphics[width=0.7cm]{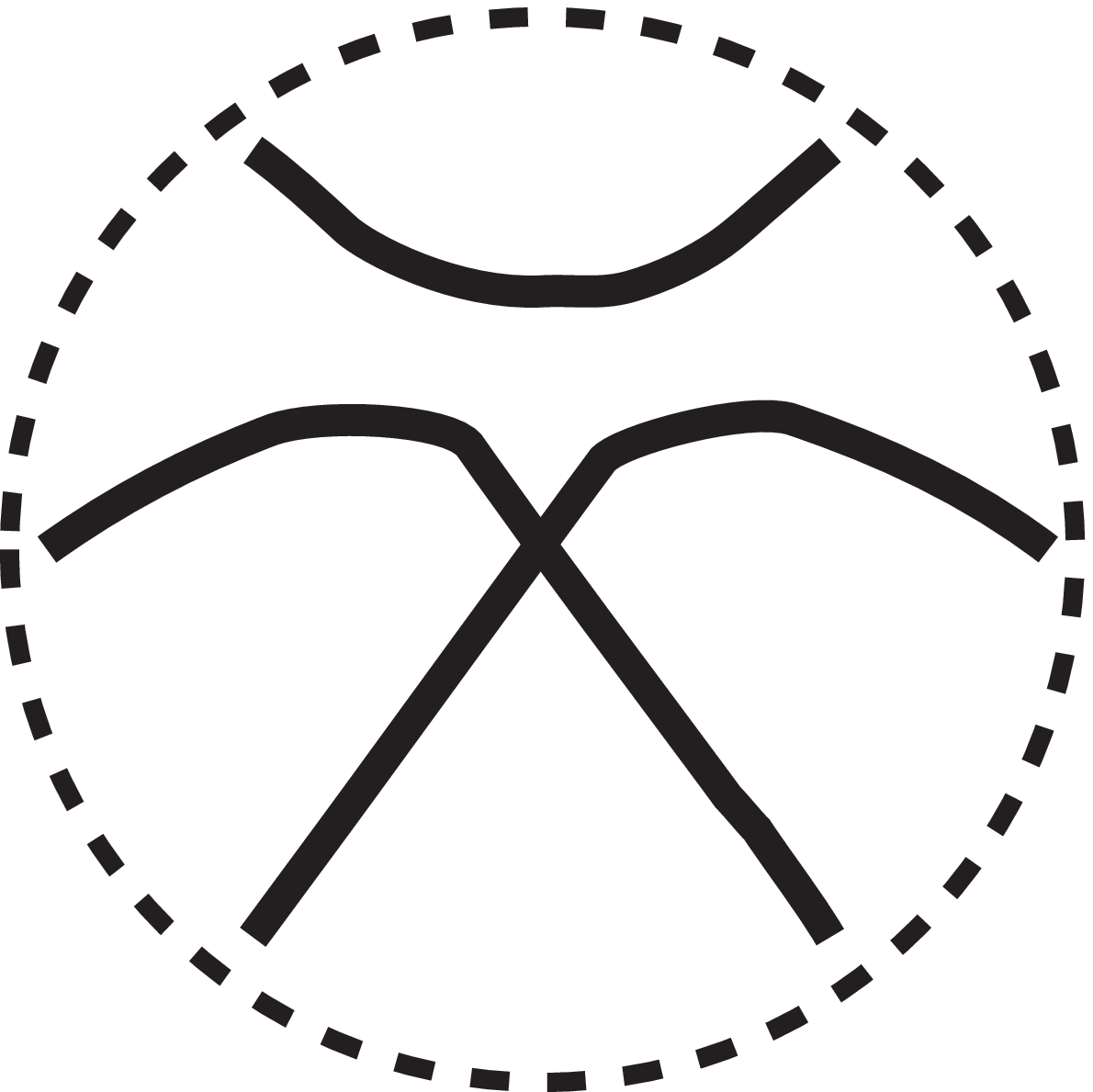}}}
\newcommand{\thirdfiosecntho}{\raisebox{-0.35\height}{\includegraphics[width=0.7cm]{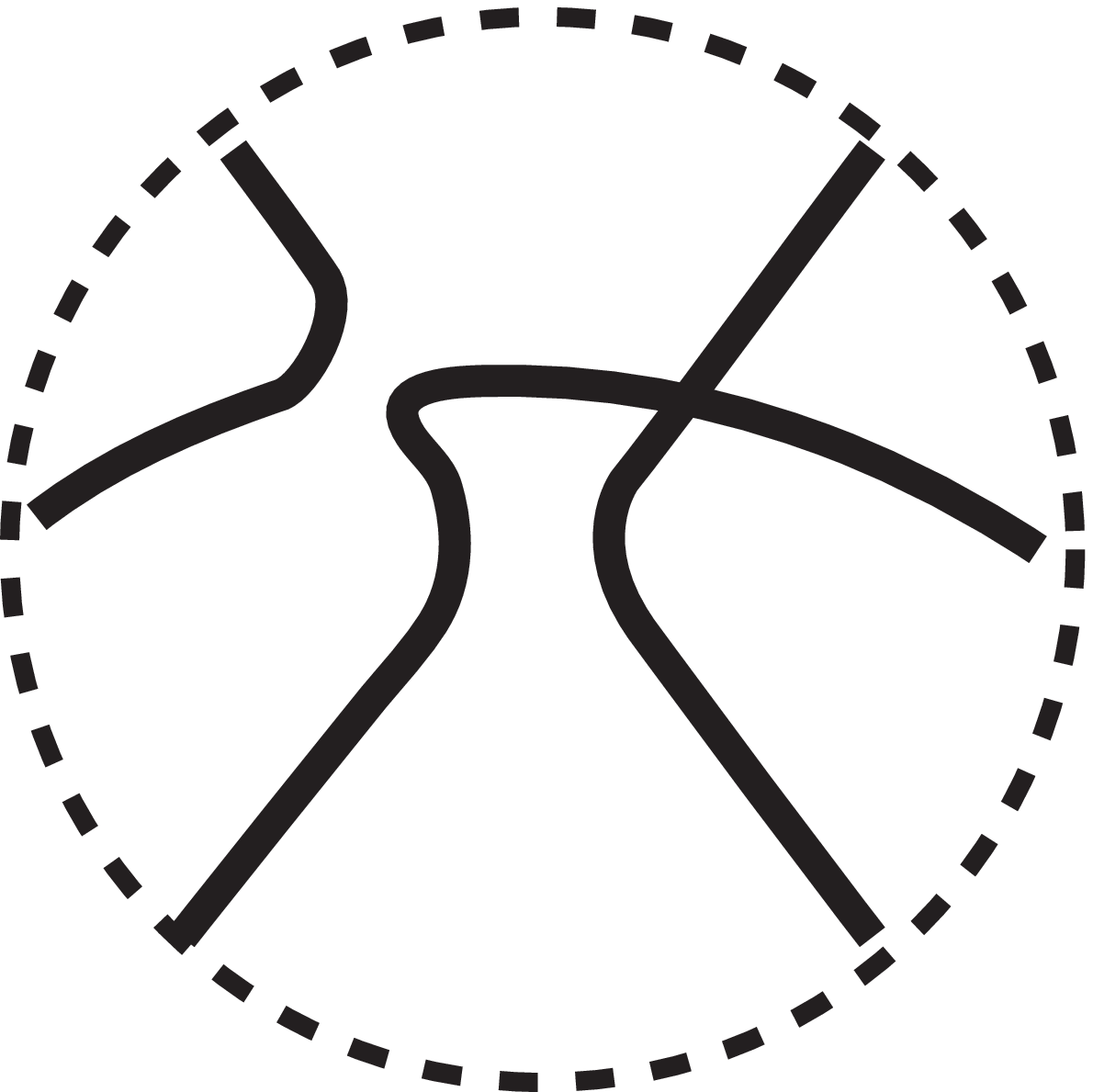}}}
\newcommand{\thirdfitsecntho}{\raisebox{-0.35\height}{\includegraphics[width=0.7cm]{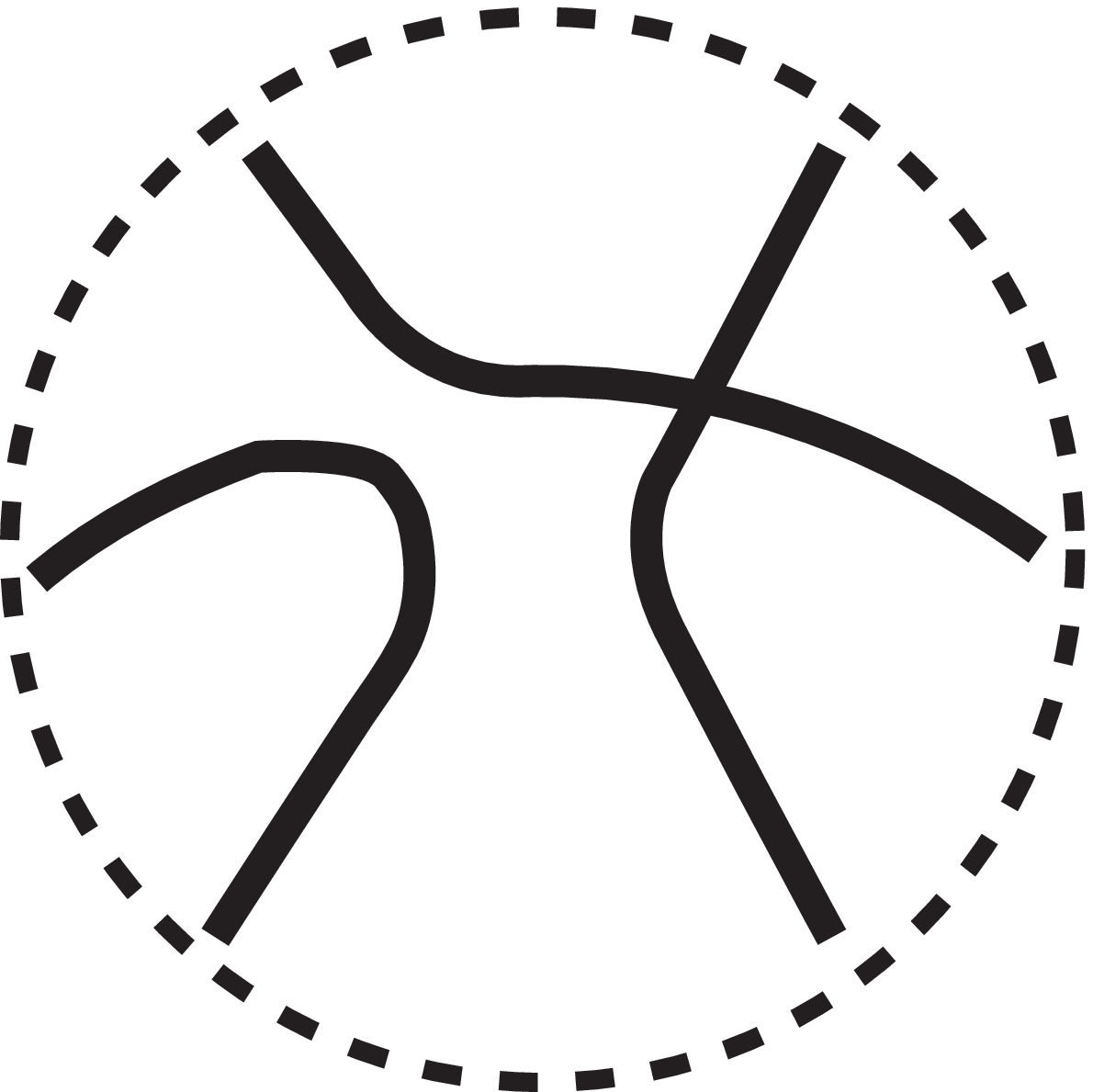}}}
\newcommand{\thirdfinsecotho}{\raisebox{-0.35\height}{\includegraphics[width=0.7cm]{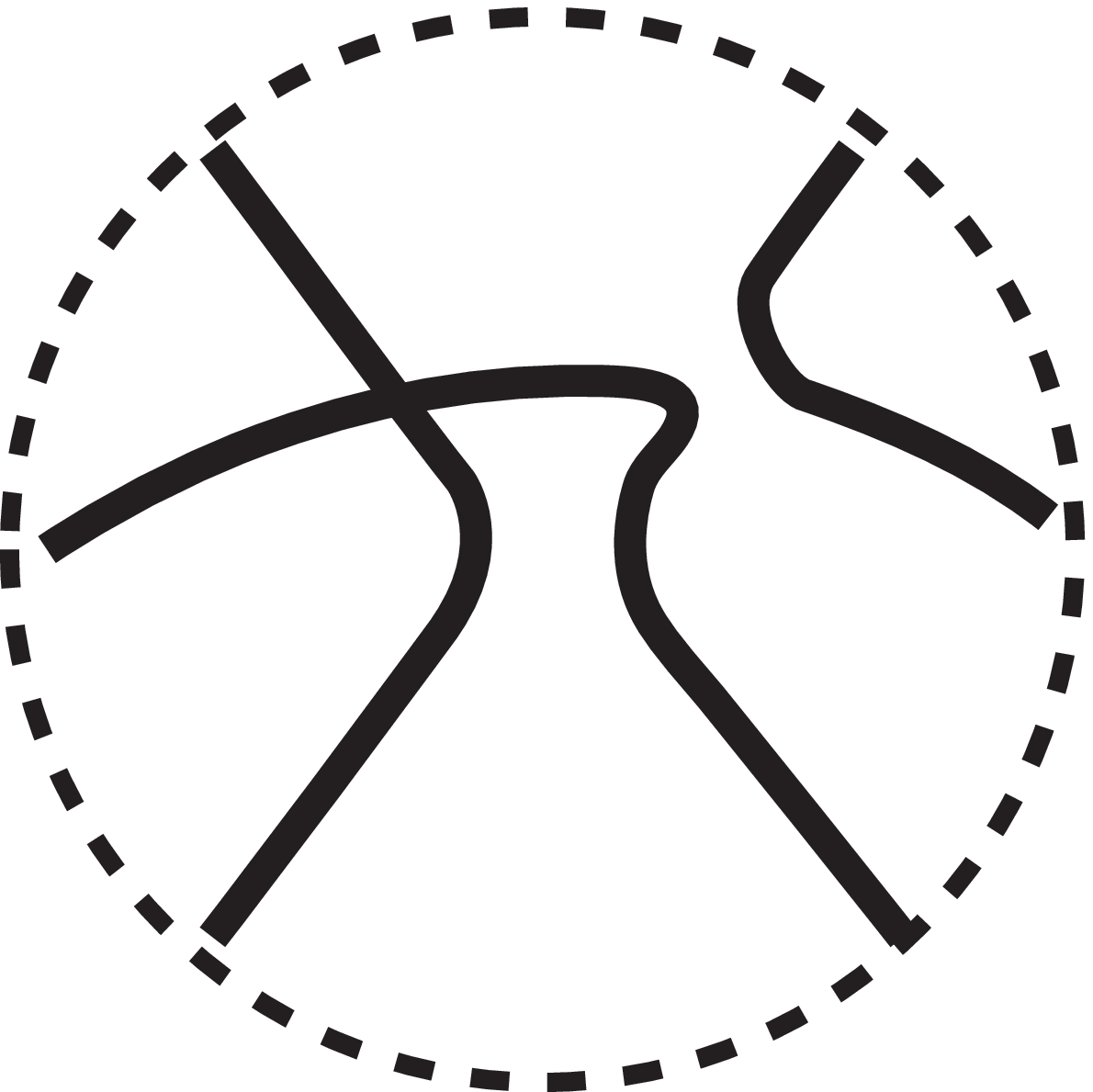}}}
\newcommand{\thirdfinsecttho}{\raisebox{-0.35\height}{\includegraphics[width=0.7cm]{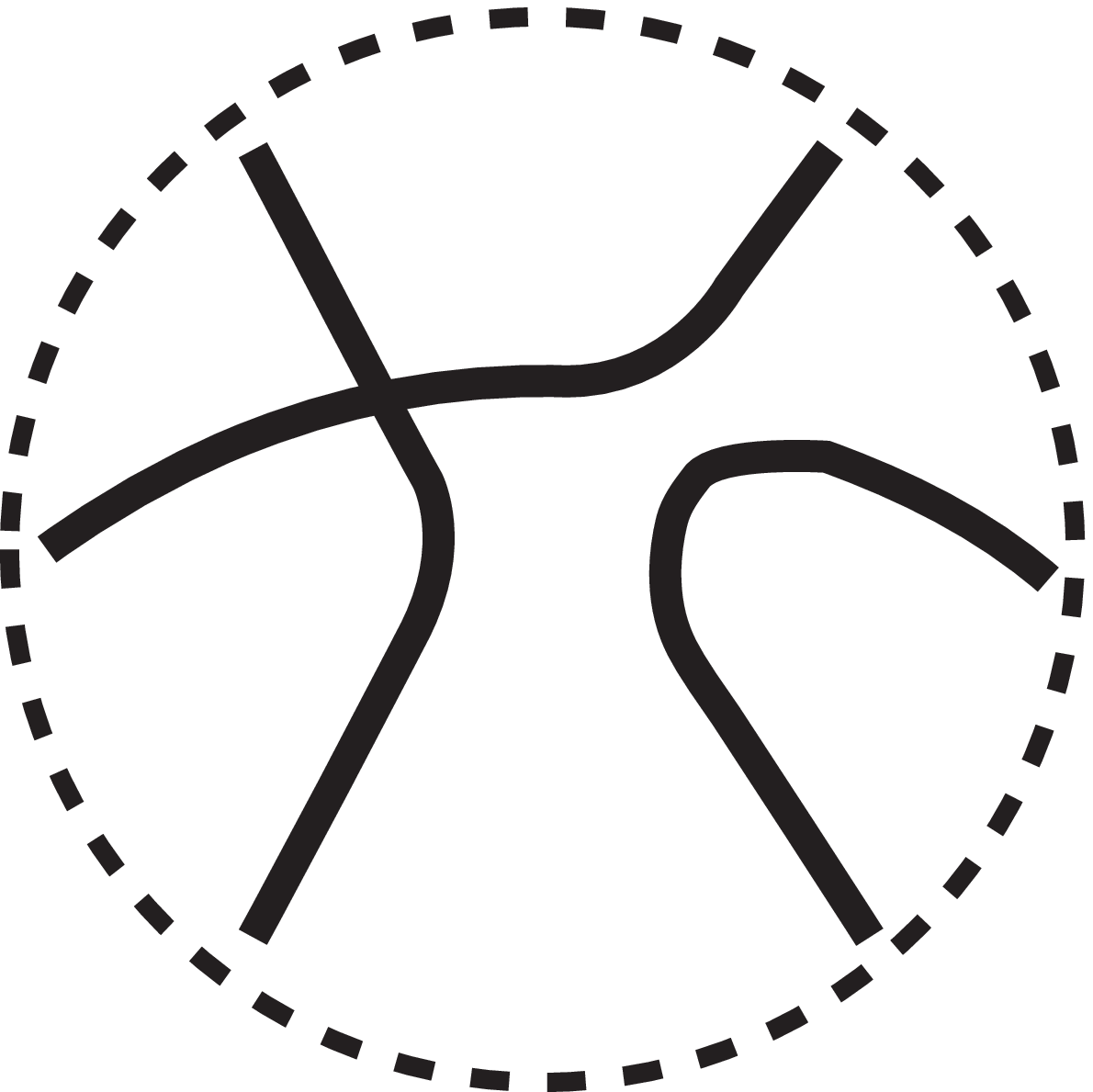}}}
\newcommand{\thirdfitsecntht}{\raisebox{-0.35\height}{\includegraphics[width=0.7cm]{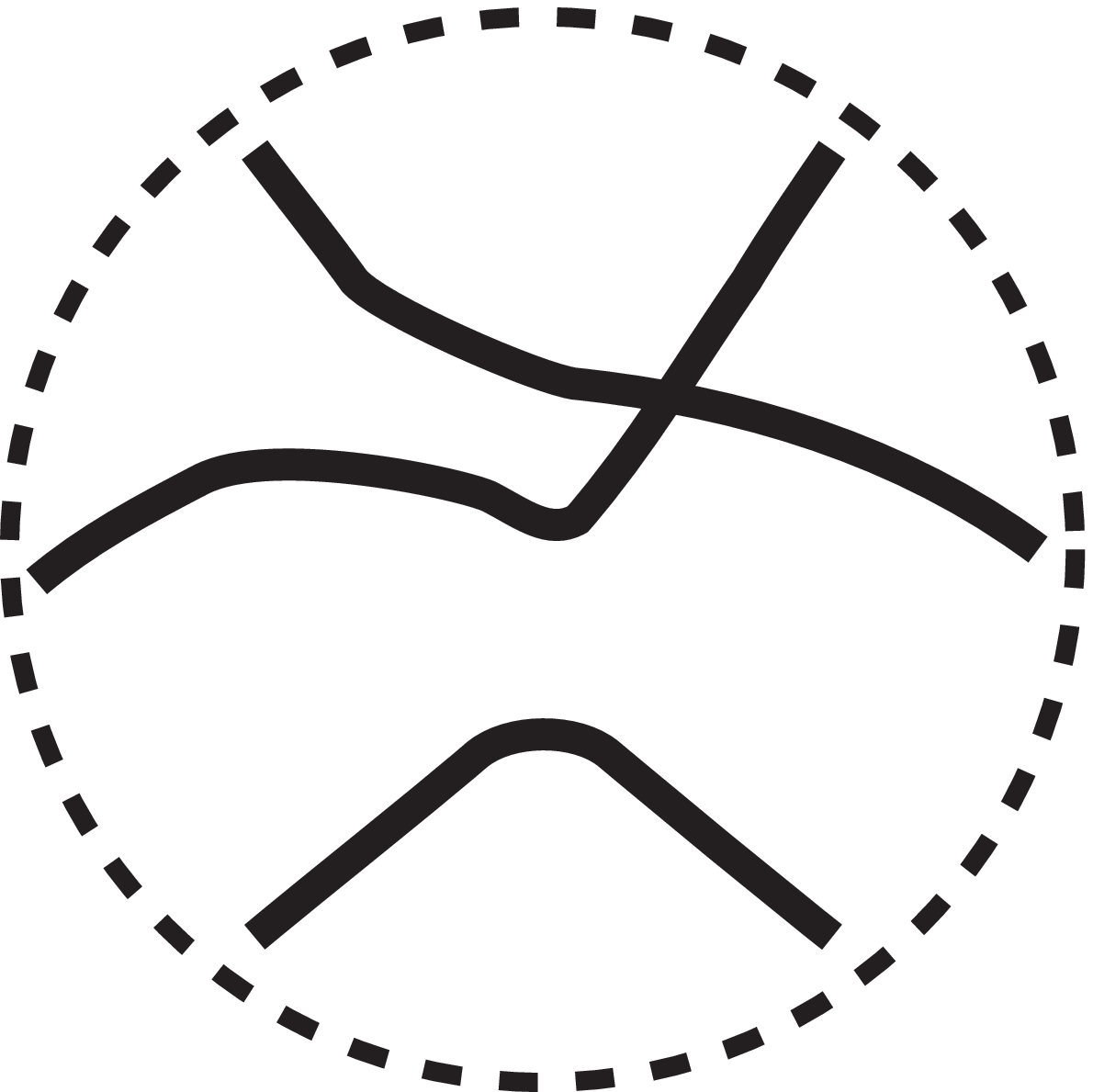}}}
\newcommand{\thirdfinsecttht}{\raisebox{-0.35\height}{\includegraphics[width=0.7cm]{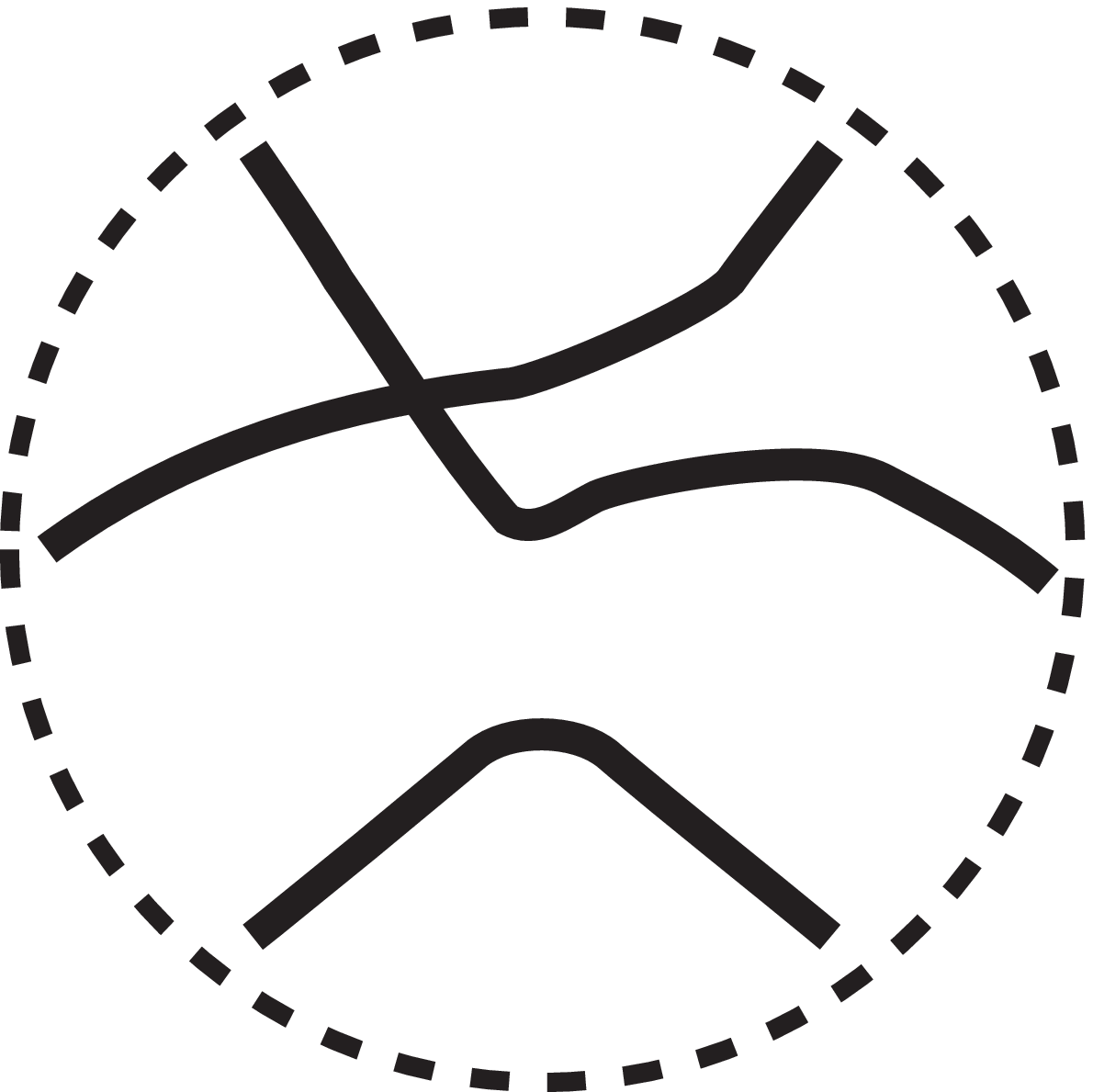}}}
\newcommand{\thirdfiosecntht}{\raisebox{-0.35\height}{\includegraphics[width=0.7cm]{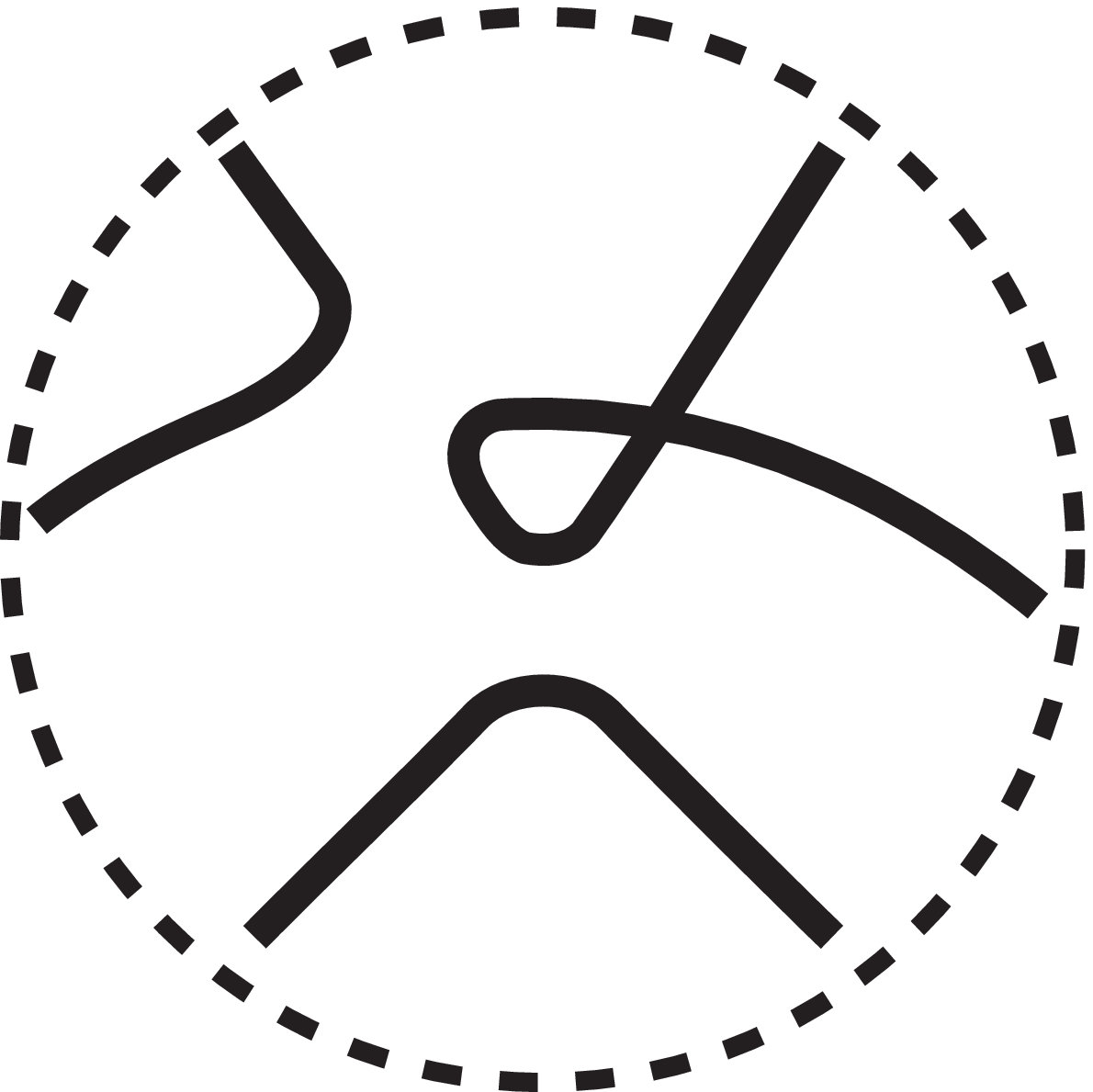}}}
\newcommand{\thirdfinsecotht}{\raisebox{-0.35\height}{\includegraphics[width=0.7cm]{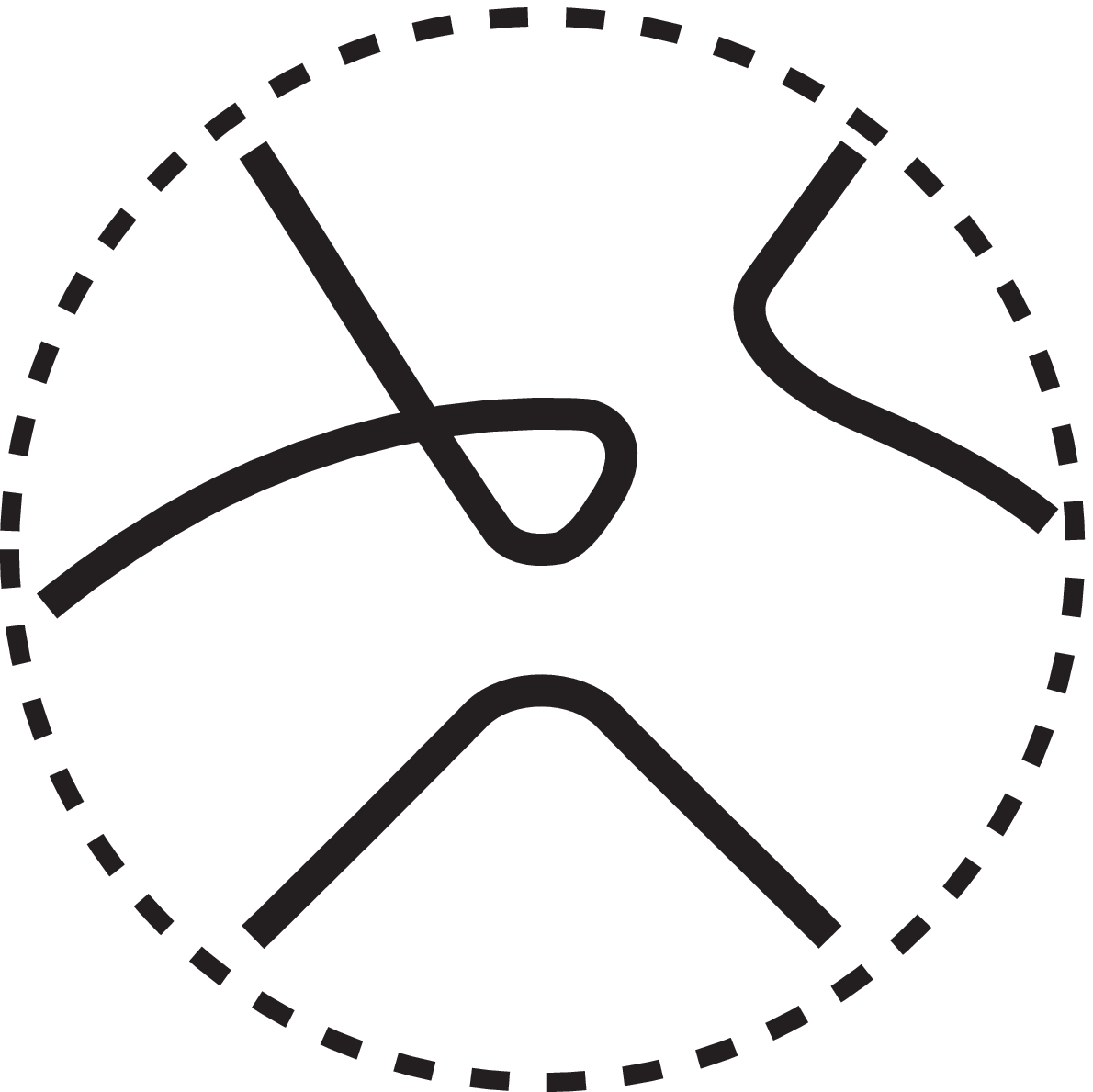}}}
\newcommand{\thirdfiosecotho}{\raisebox{-0.35\height}{\includegraphics[width=0.7cm]{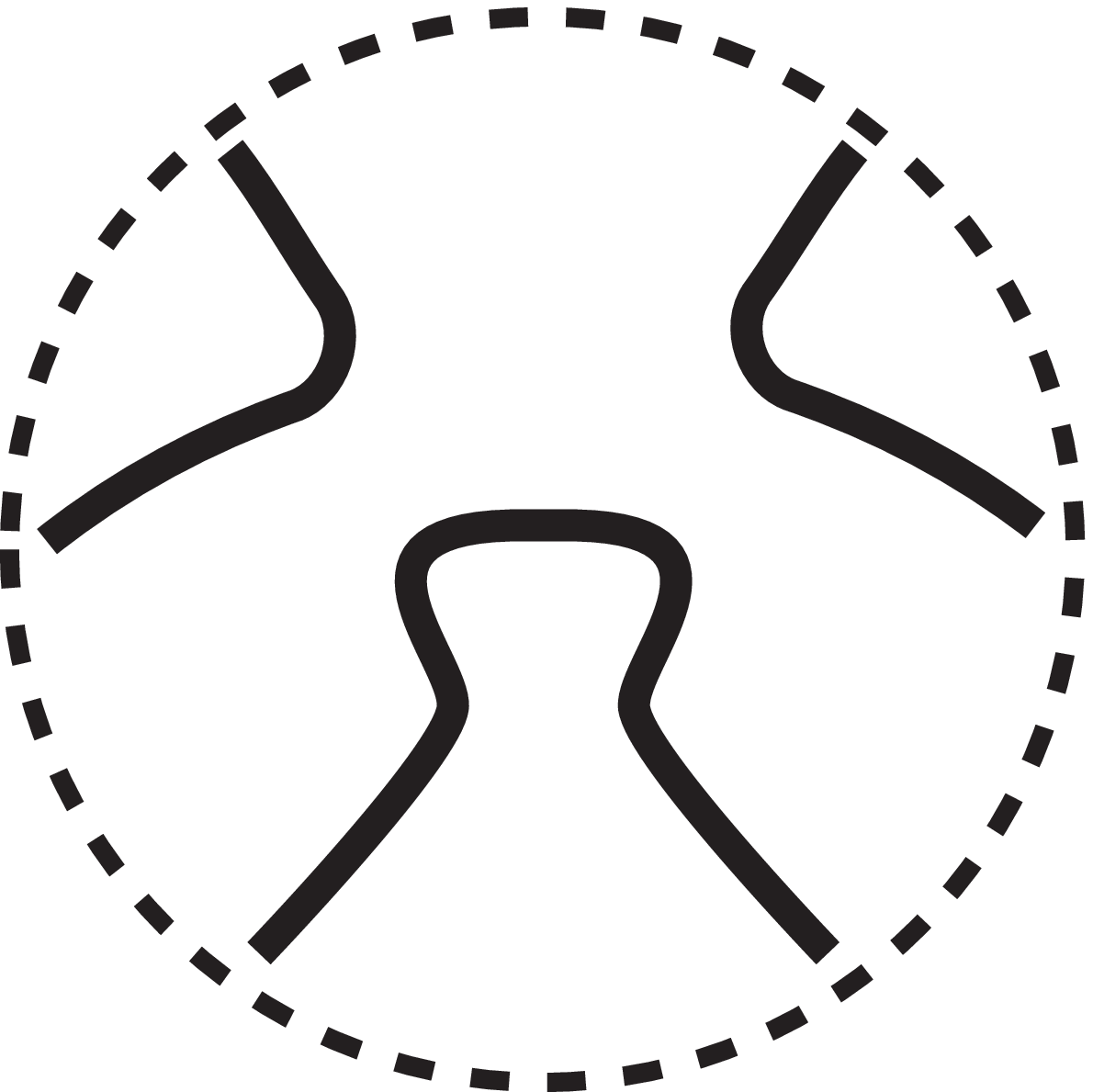}}}
\newcommand{\thirdfitsecttht}{\raisebox{-0.35\height}{\includegraphics[width=0.7cm]{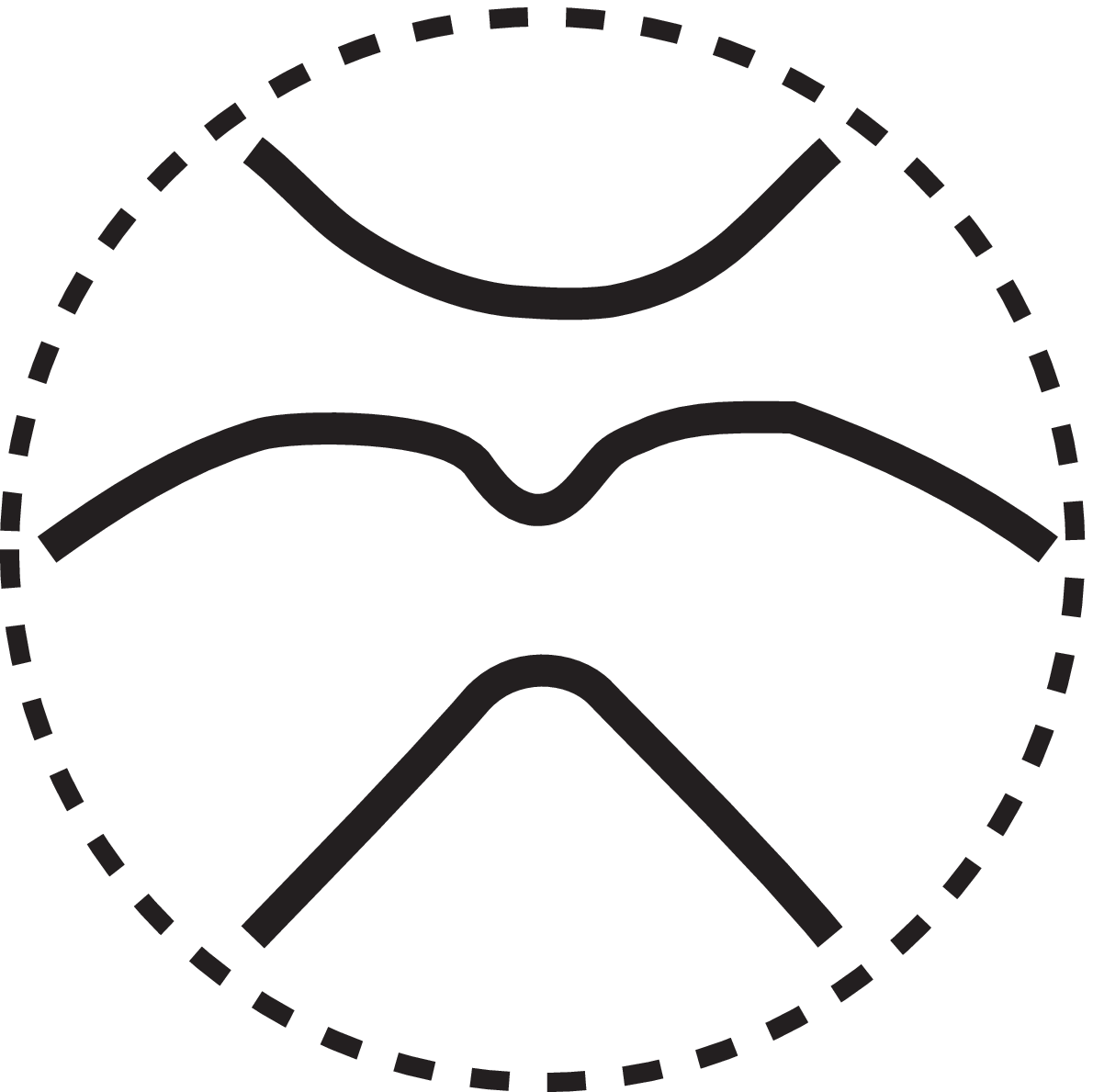}}}
\newcommand{\thirdfiosecotht}{\raisebox{-0.35\height}{\includegraphics[width=0.7cm]{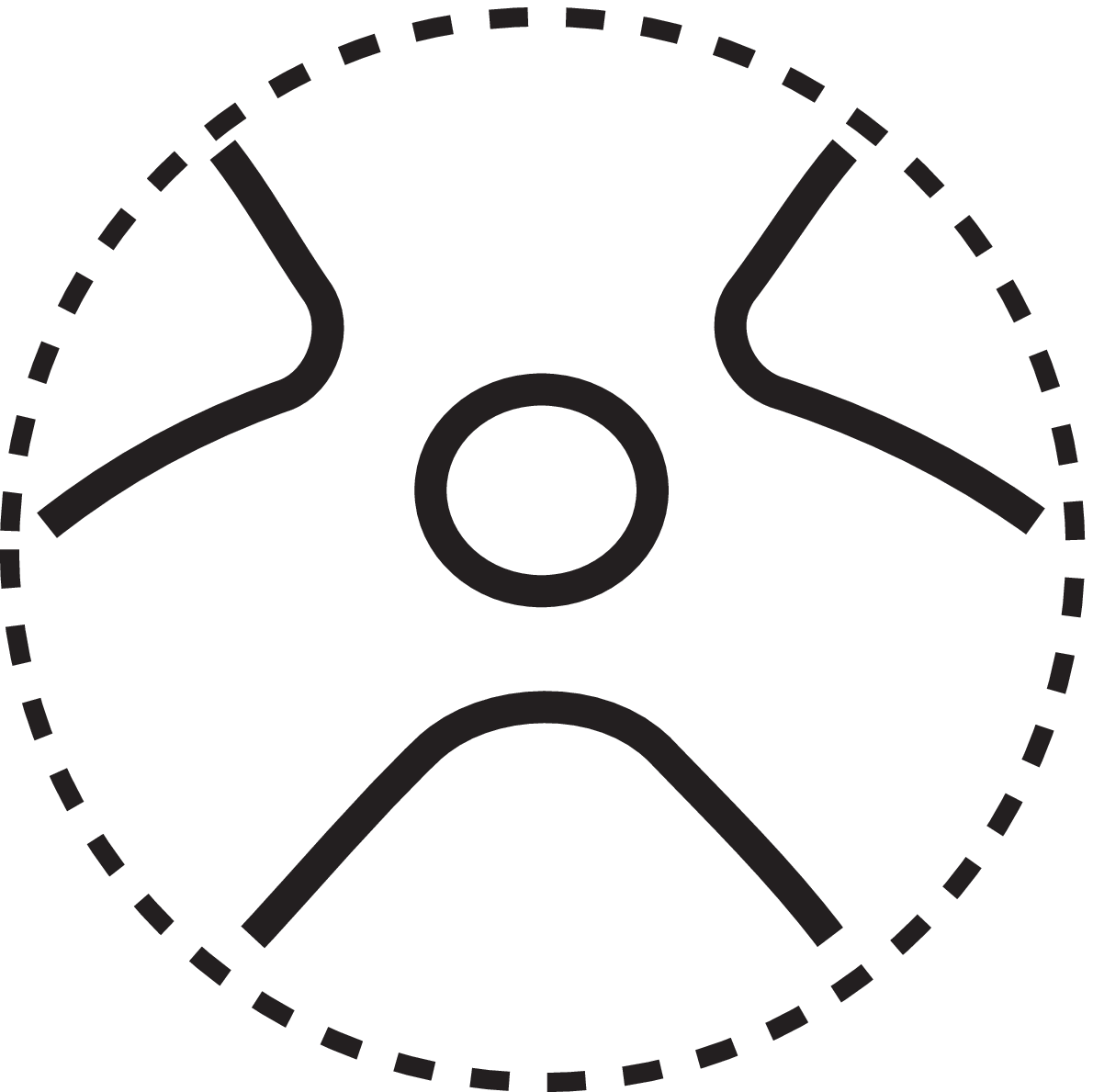}}}
\newcommand{\thirdfitsecttho}{\raisebox{-0.35\height}{\includegraphics[width=0.7cm]{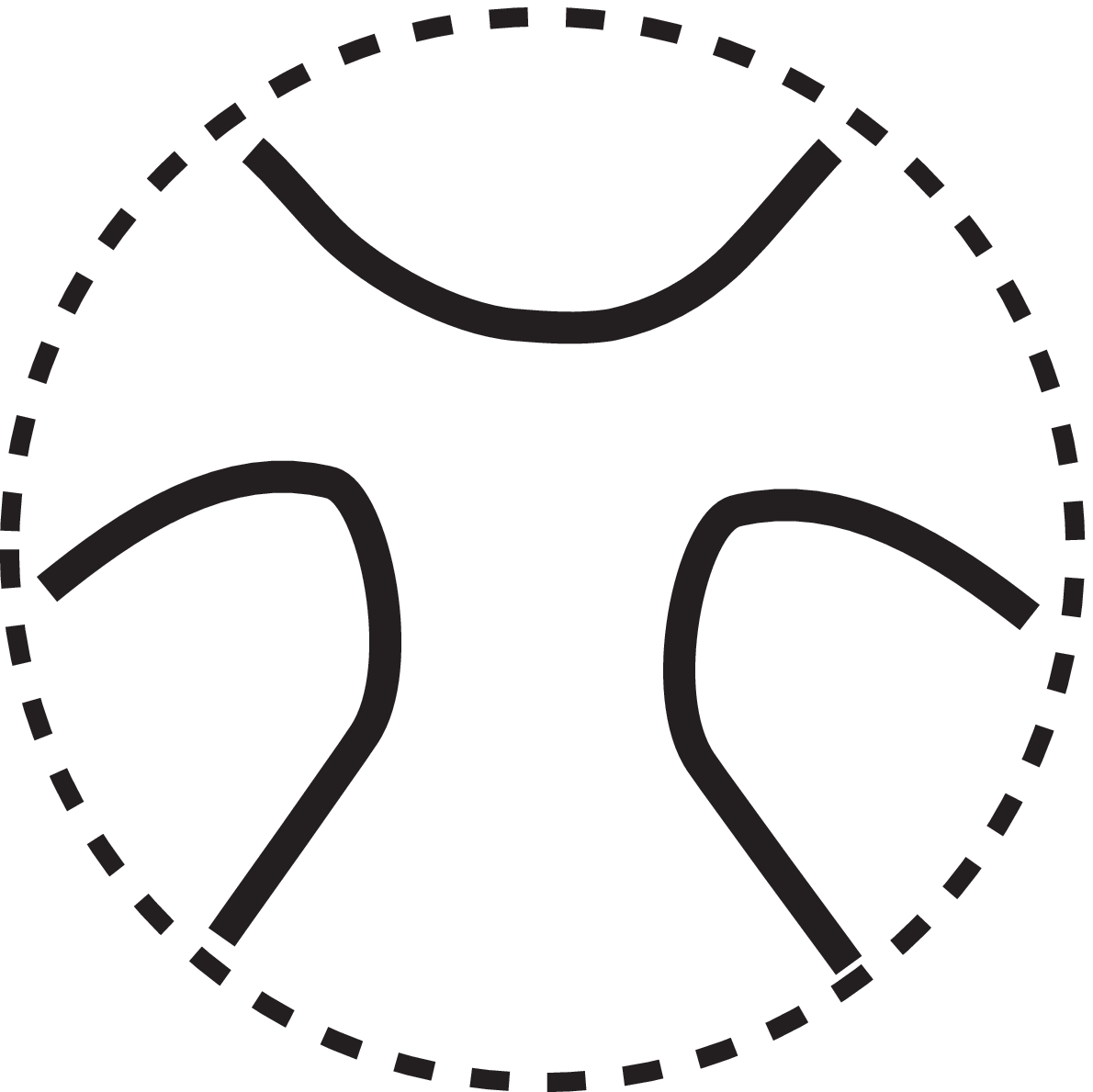}}}
\newcommand{\thirdfitsecotho}{\raisebox{-0.35\height}{\includegraphics[width=0.7cm]{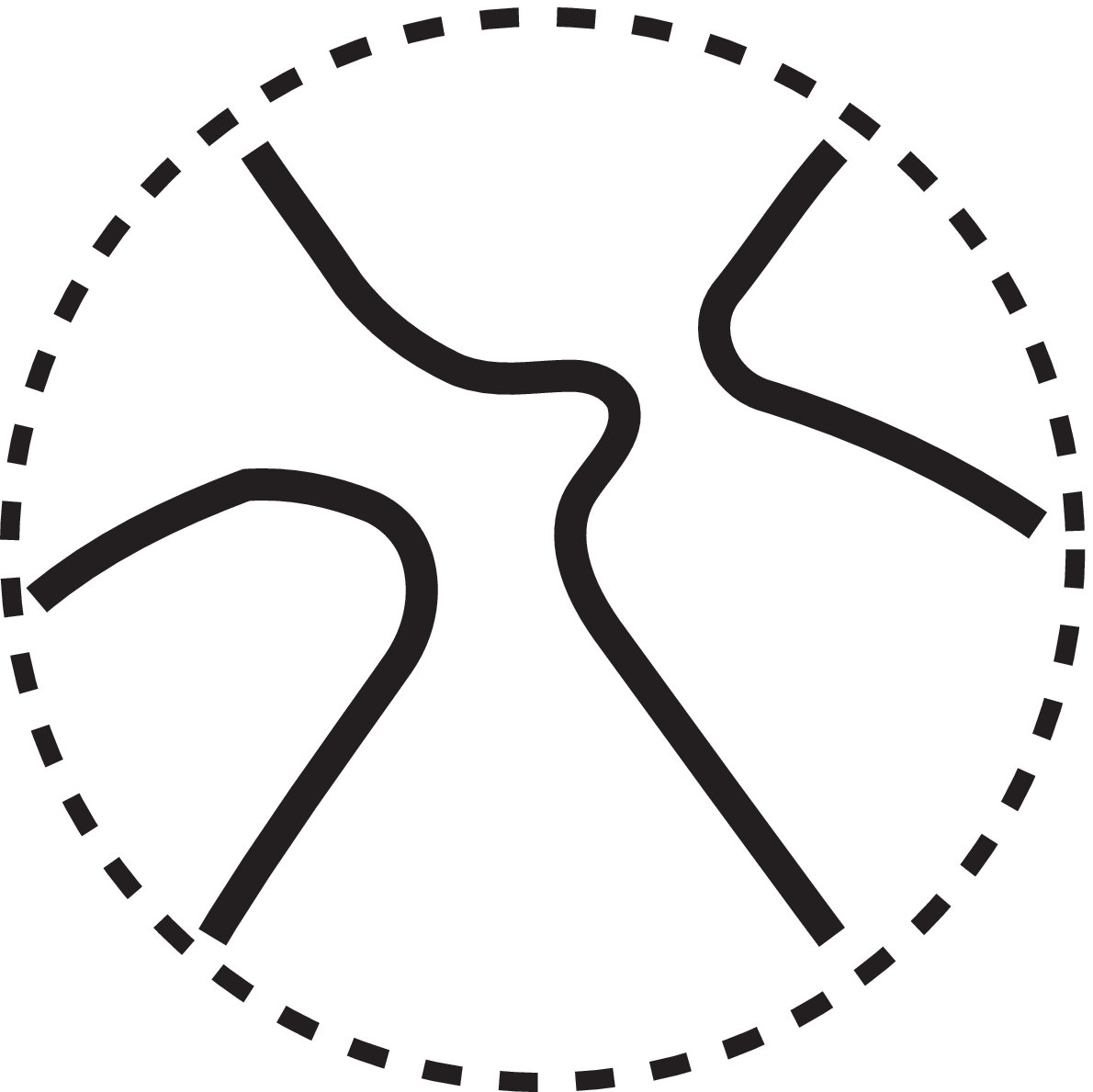}}}

\newcommand{\thirdnpr}{\raisebox{-0.35\height}{\includegraphics[width=0.7cm]{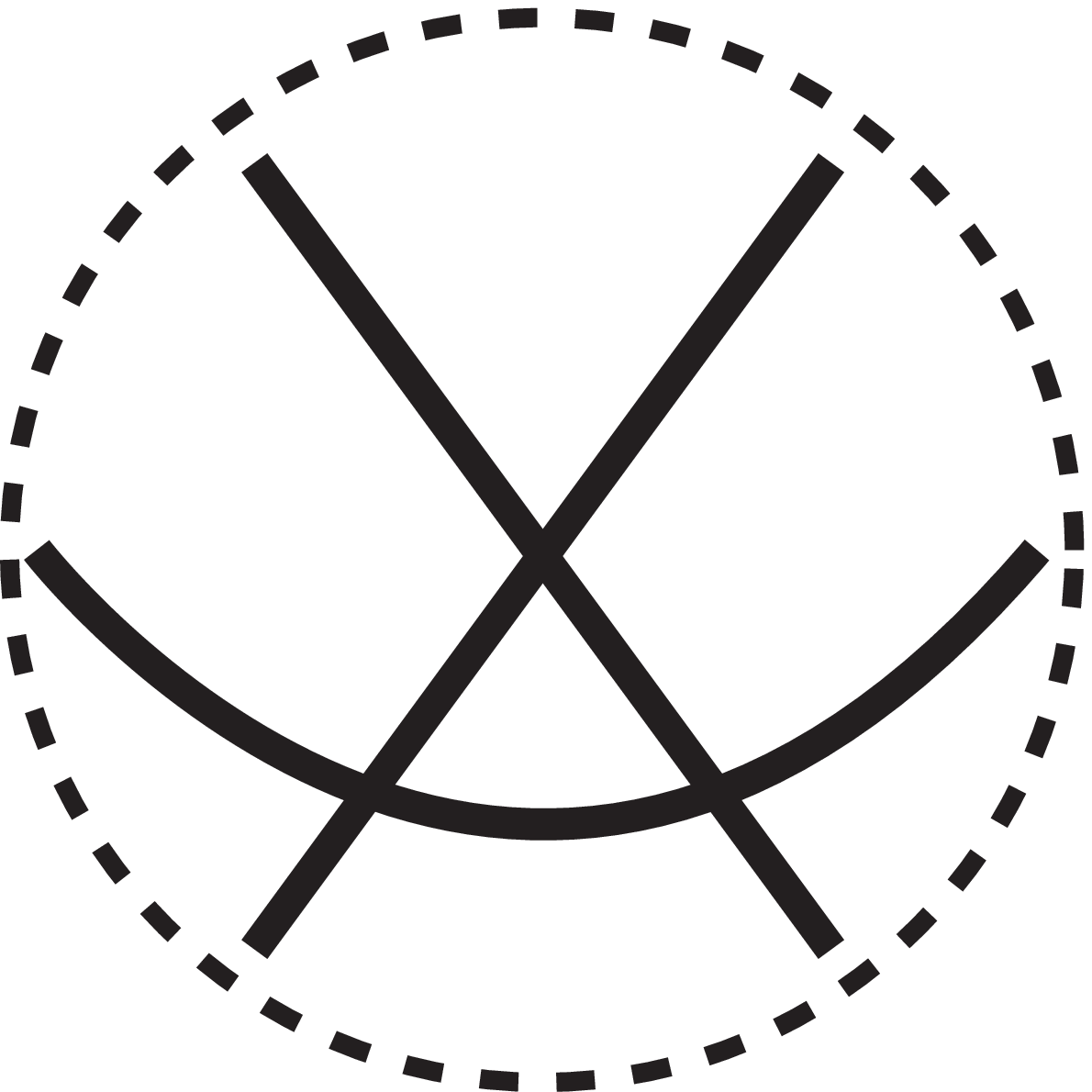}}}
\newcommand{\thirdfinsecnthopr}{\raisebox{-0.35\height}{\includegraphics[width=0.7cm]{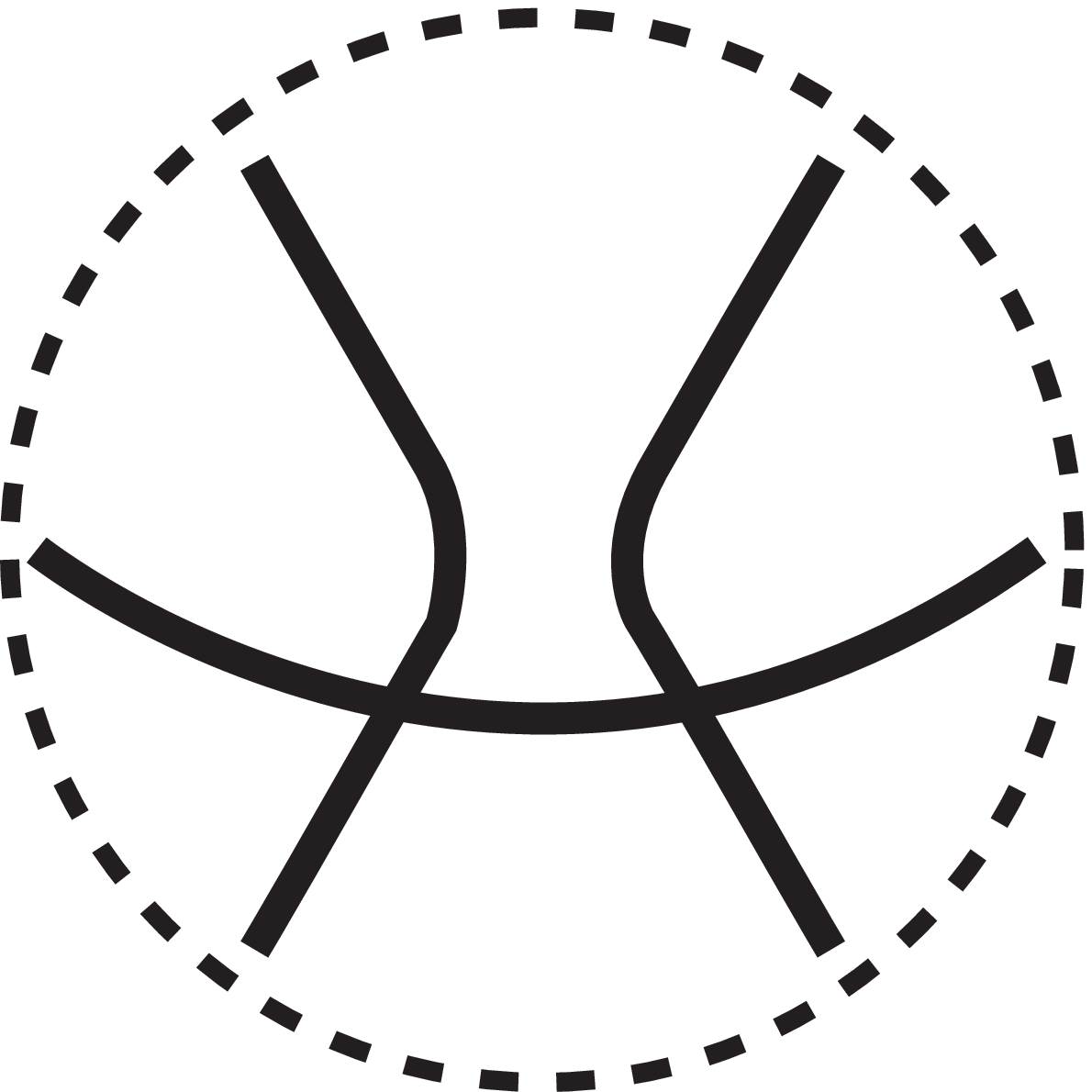}}}
\newcommand{\thirdfinsecnthtpr}{\raisebox{-0.35\height}{\includegraphics[width=0.7cm]{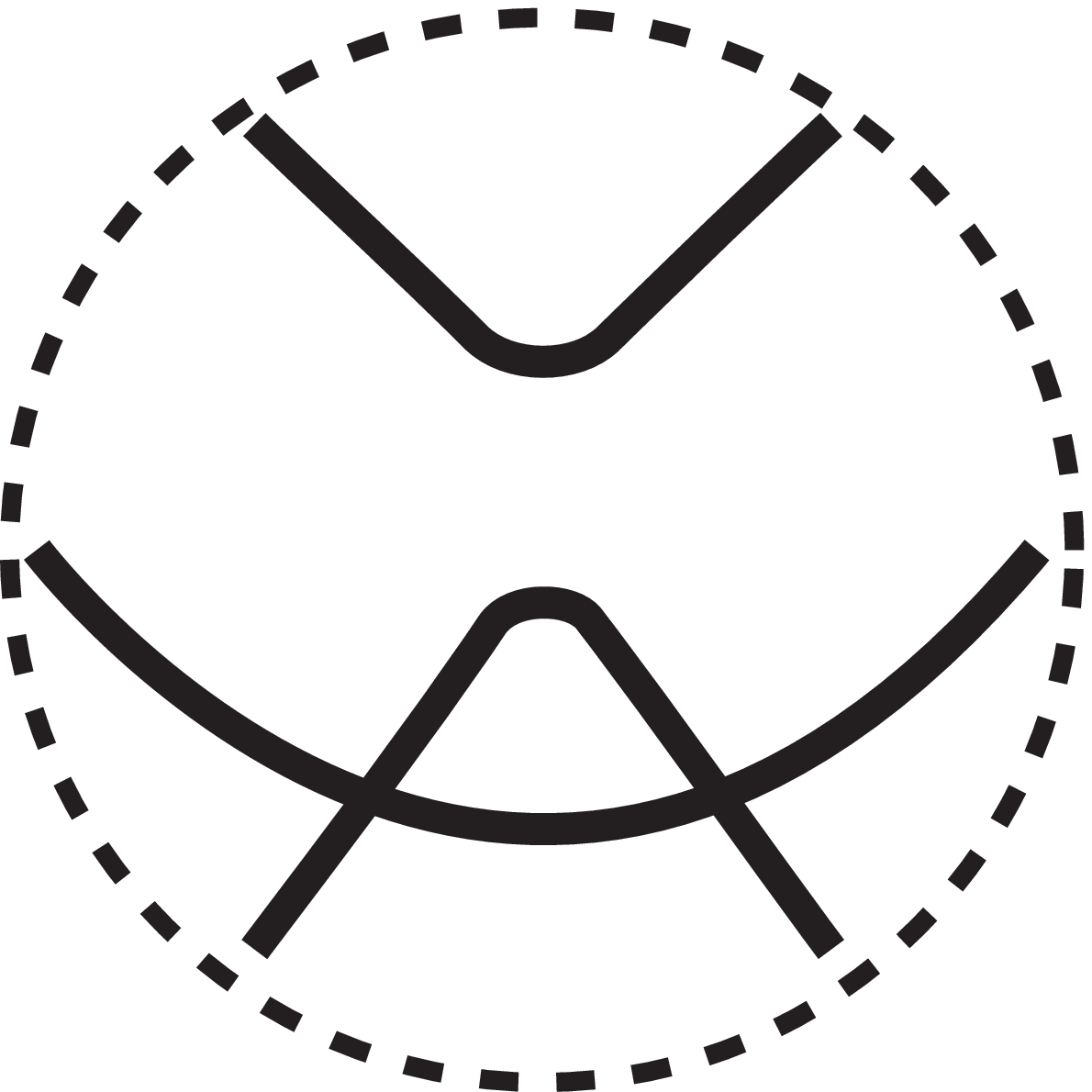}}}
\newcommand{\thirdfinsecothnpr}{\raisebox{-0.35\height}{\includegraphics[width=0.7cm]{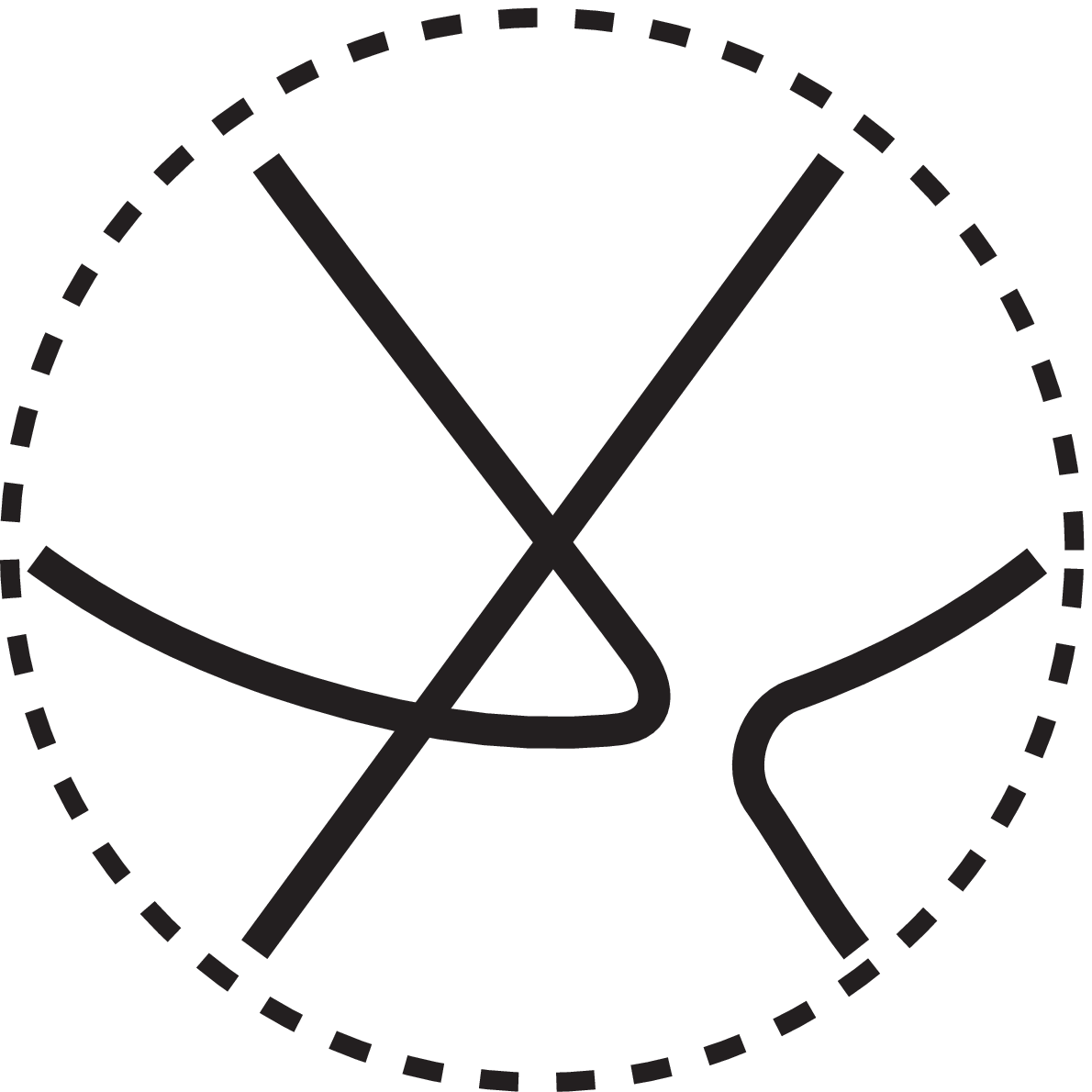}}}
\newcommand{\thirdfinsectthnpr}{\raisebox{-0.35\height}{\includegraphics[width=0.7cm]{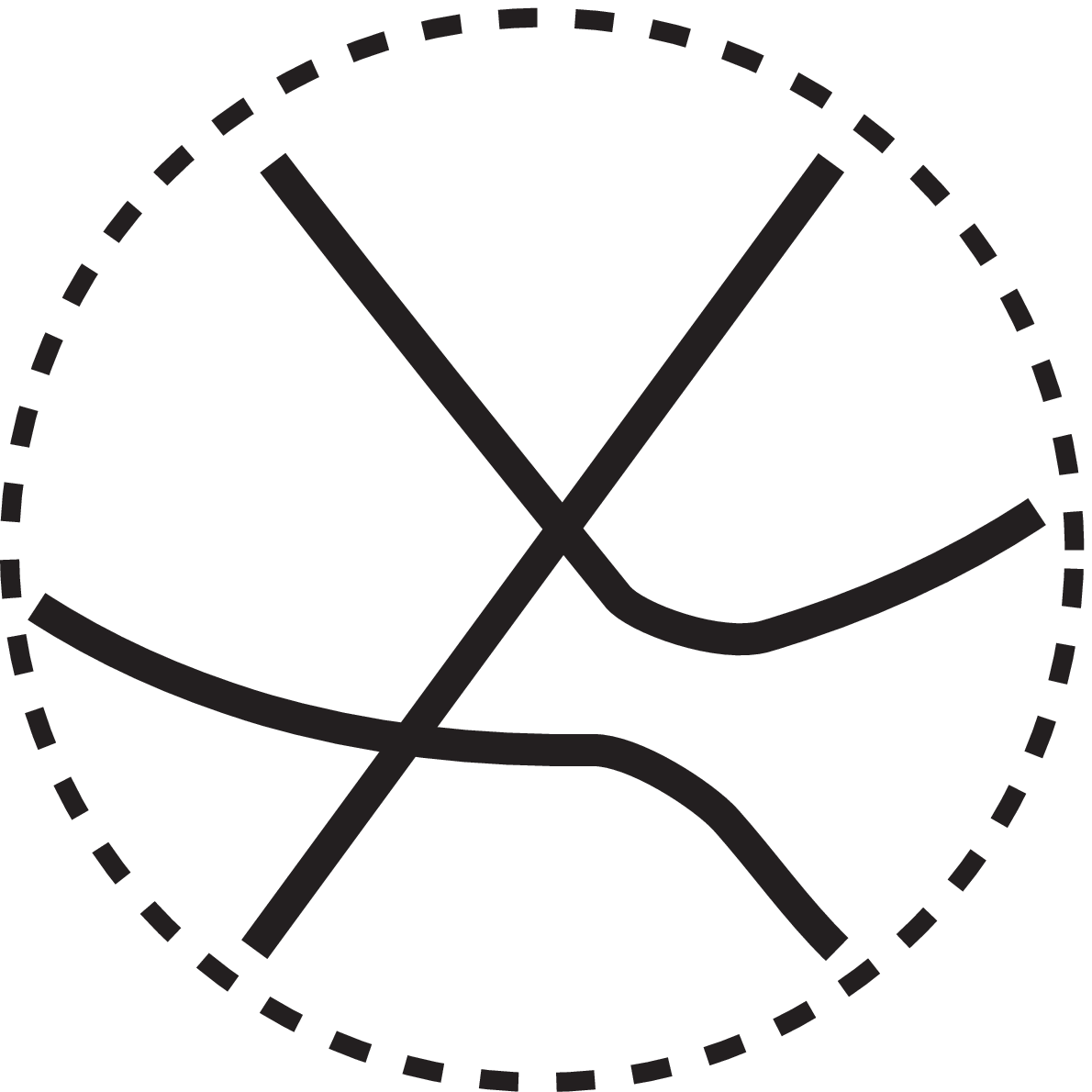}}}
\newcommand{\thirdfiosecnthnpr}{\raisebox{-0.35\height}{\includegraphics[width=0.7cm]{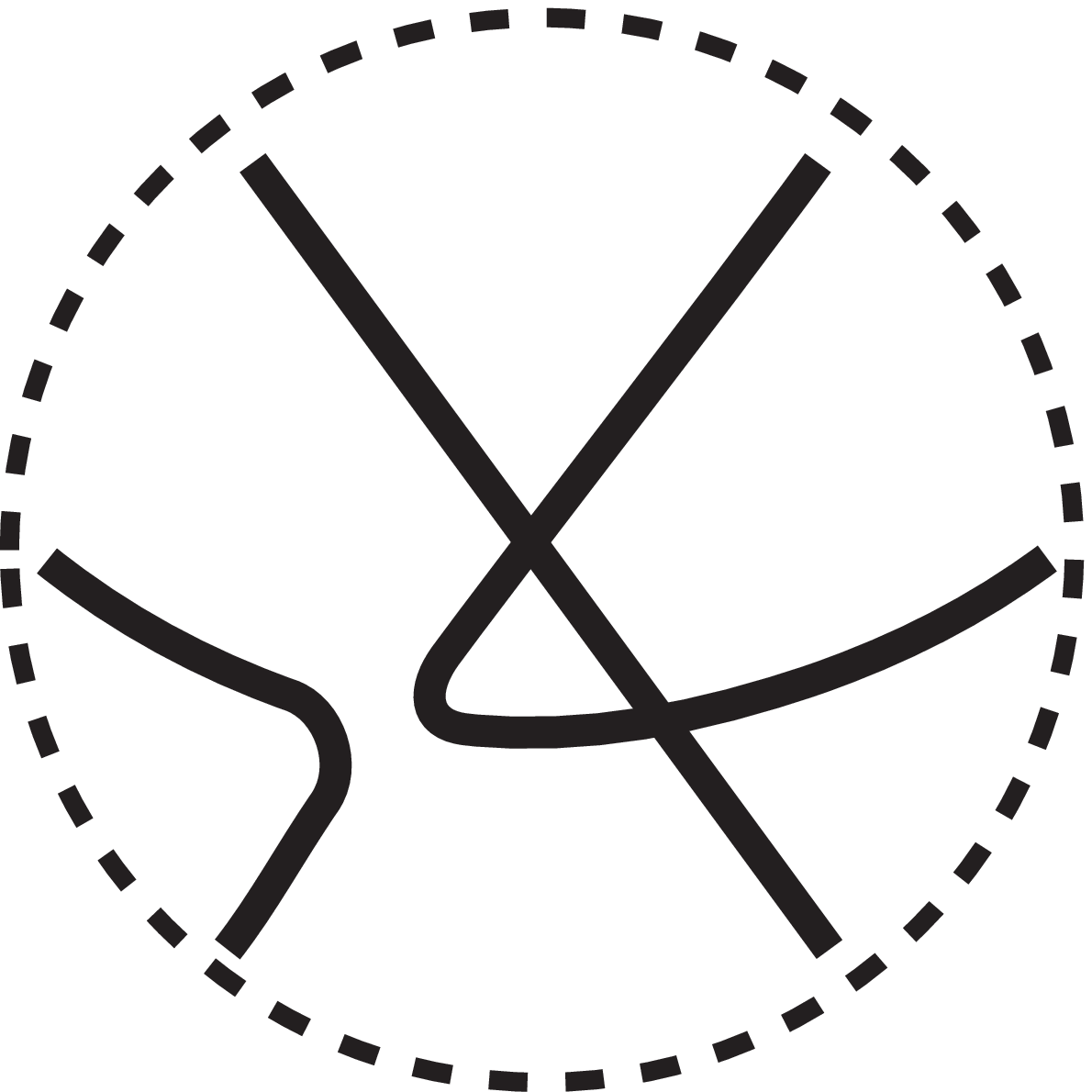}}}
\newcommand{\thirdfiosecothnpr}{\raisebox{-0.35\height}{\includegraphics[width=0.7cm]{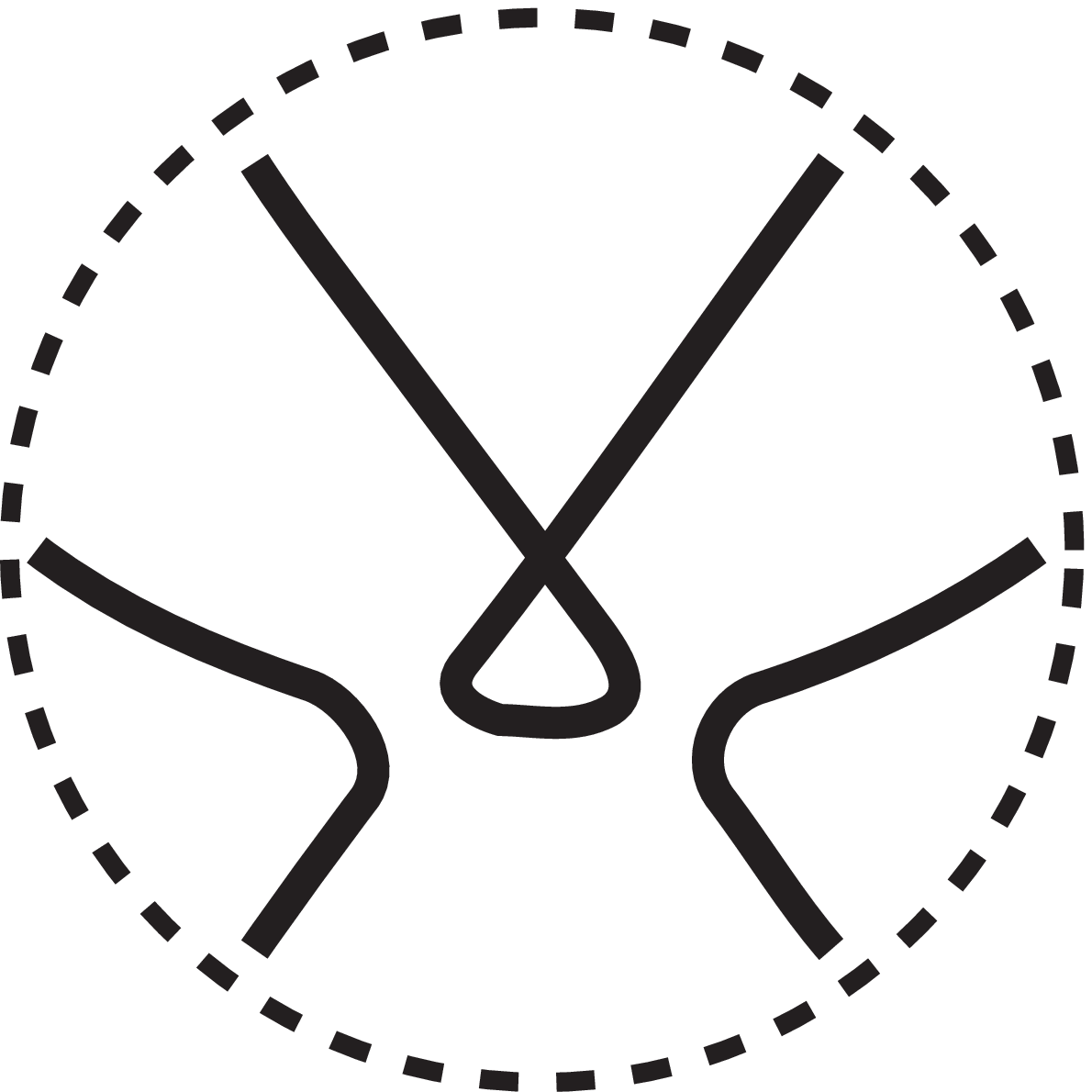}}}
\newcommand{\thirdfiosectthnpr}{\raisebox{-0.35\height}{\includegraphics[width=0.7cm]{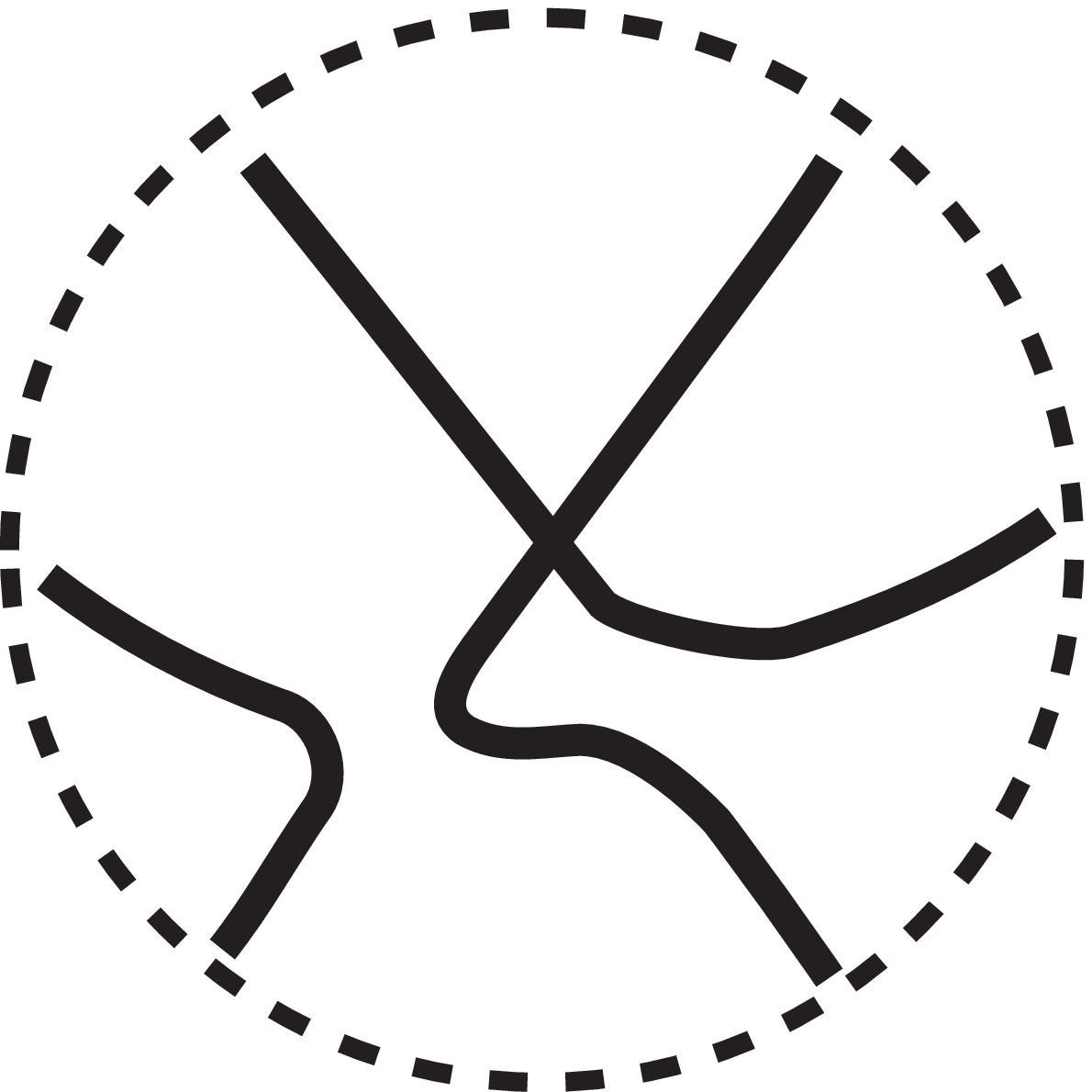}}}
\newcommand{\thirdfiosectthopr}{\raisebox{-0.35\height}{\includegraphics[width=0.7cm]{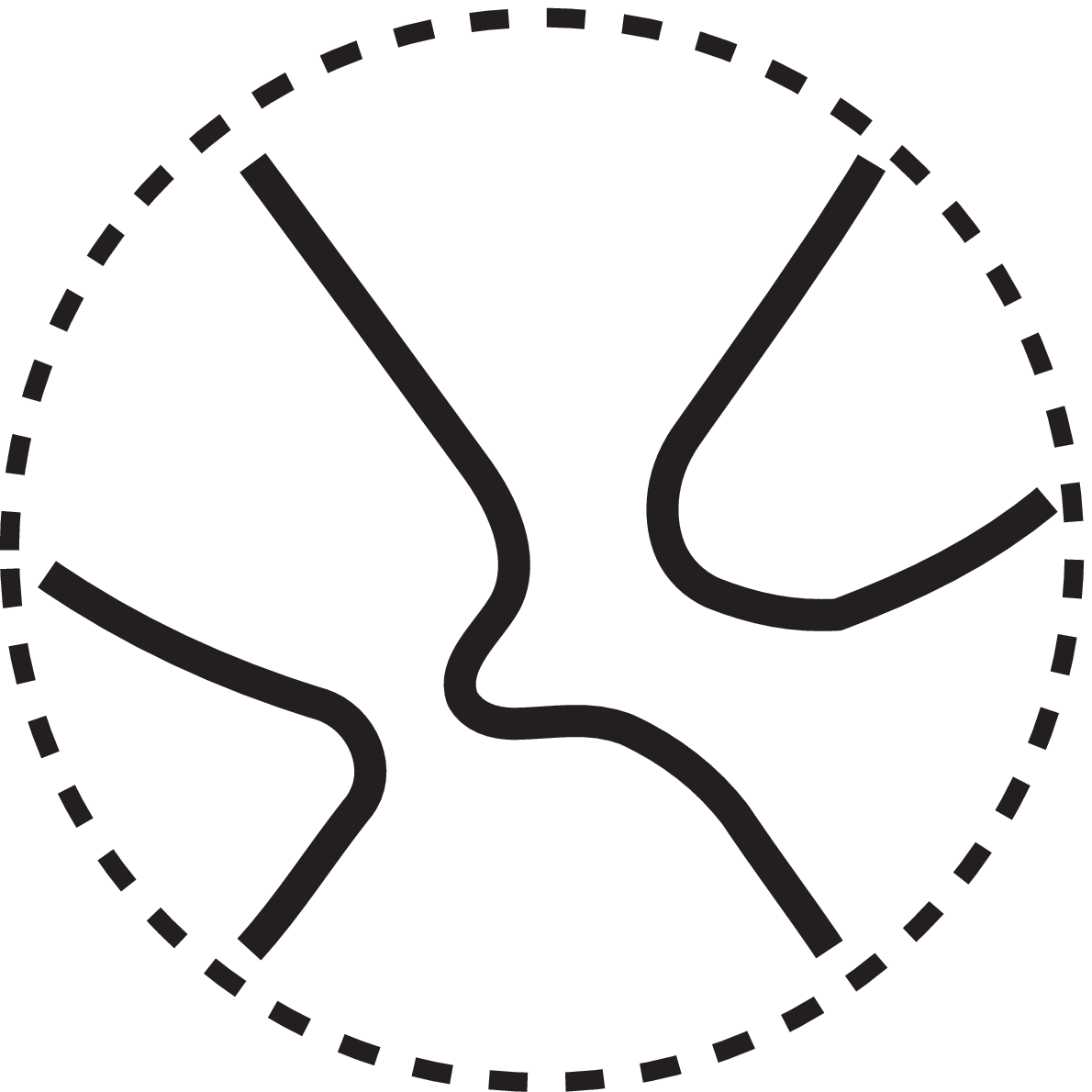}}}
\newcommand{\thirdfiosectthtpr}{\raisebox{-0.35\height}{\includegraphics[width=0.7cm]{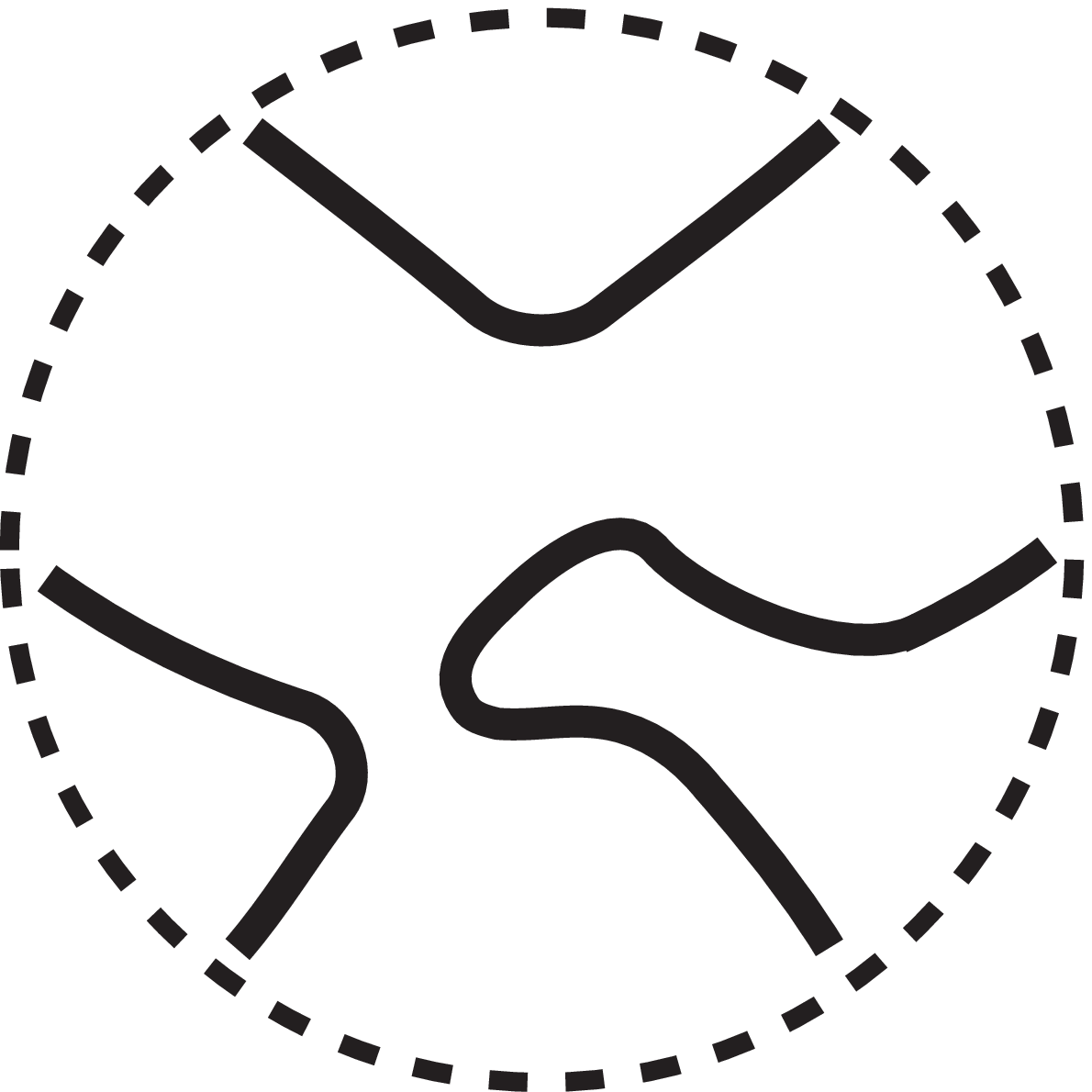}}}
\newcommand{\thirdfitsecnthnpr}{\raisebox{-0.35\height}{\includegraphics[width=0.7cm]{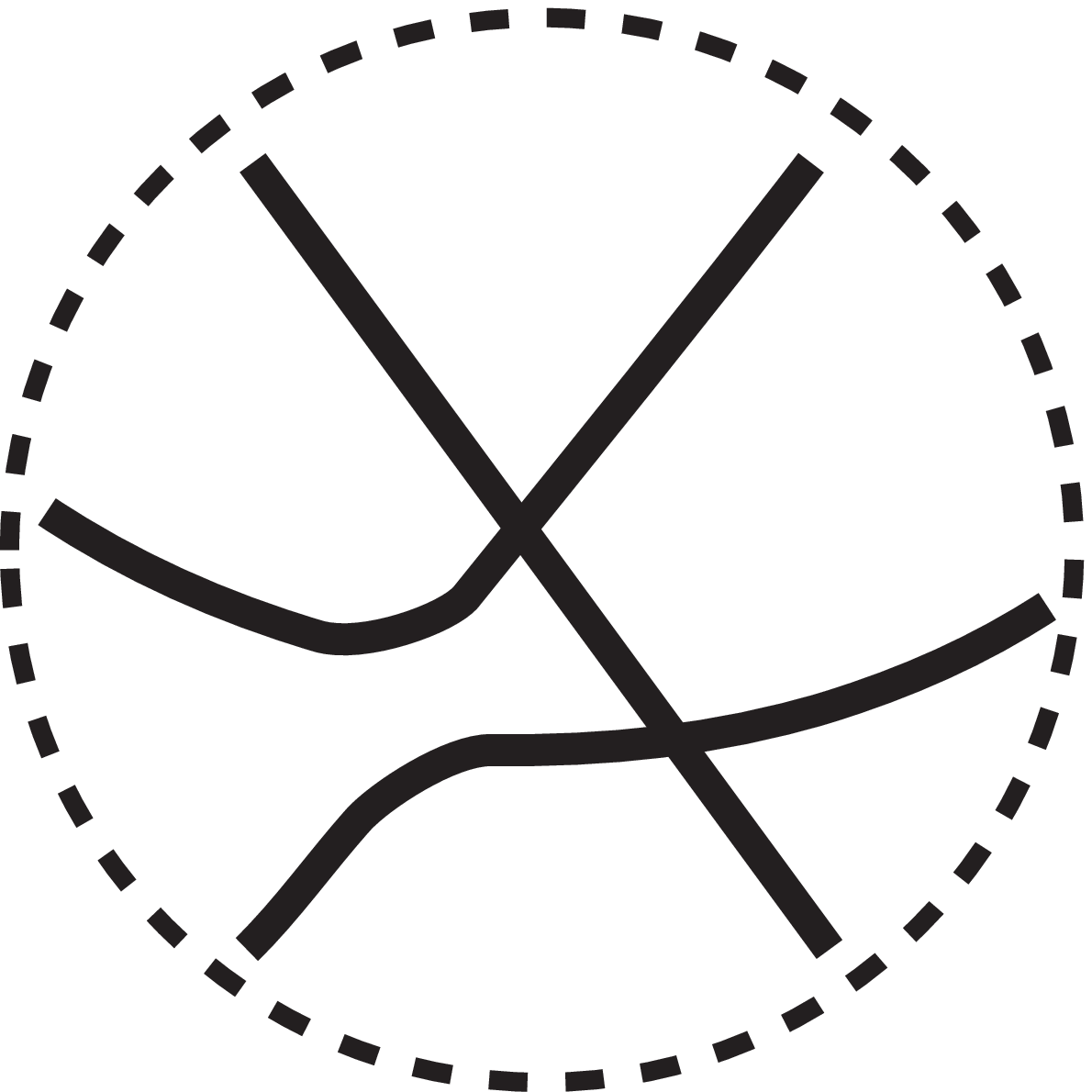}}}
\newcommand{\thirdfitsecothnpr}{\raisebox{-0.35\height}{\includegraphics[width=0.7cm]{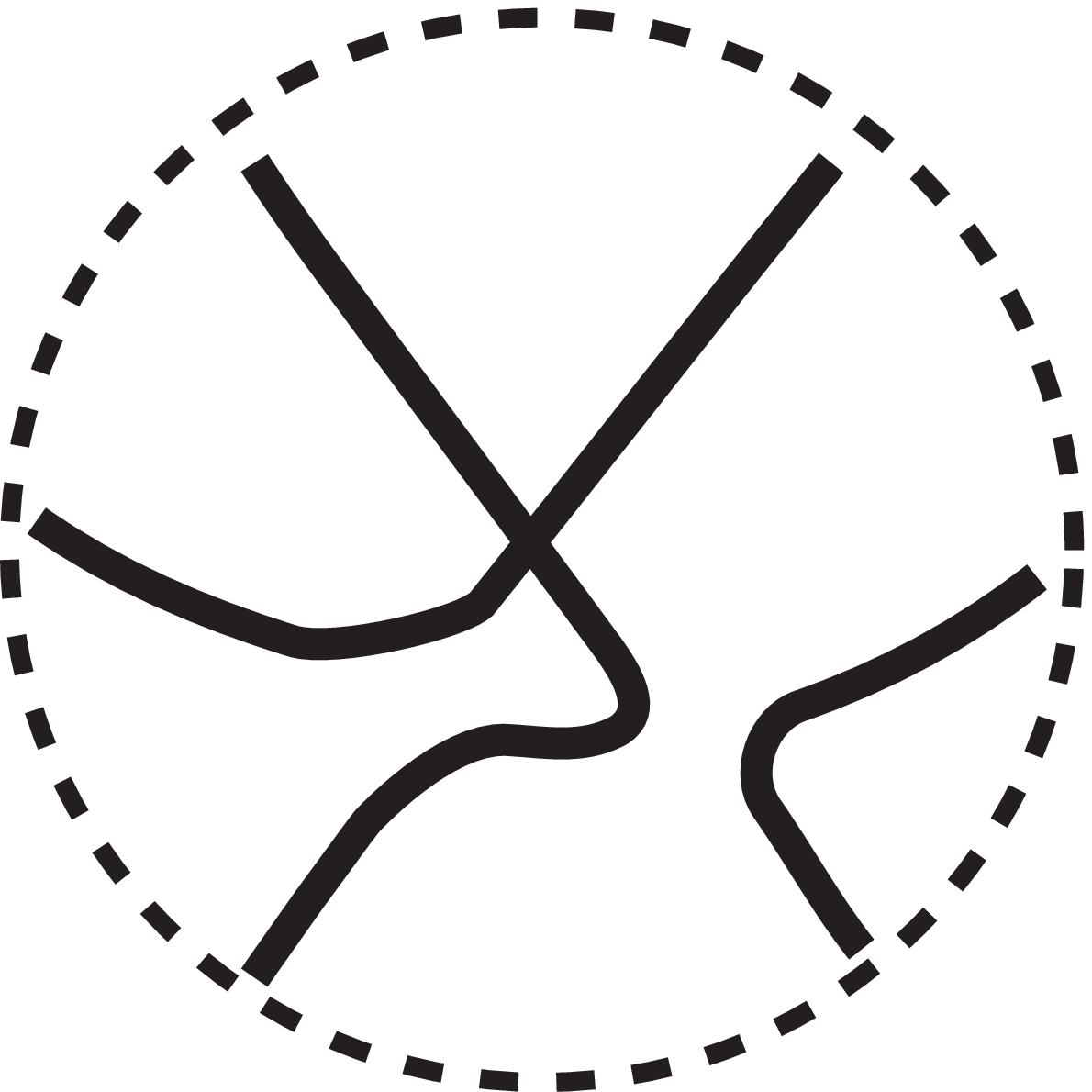}}}
\newcommand{\thirdfitsecothtpr}{\raisebox{-0.35\height}{\includegraphics[width=0.7cm]{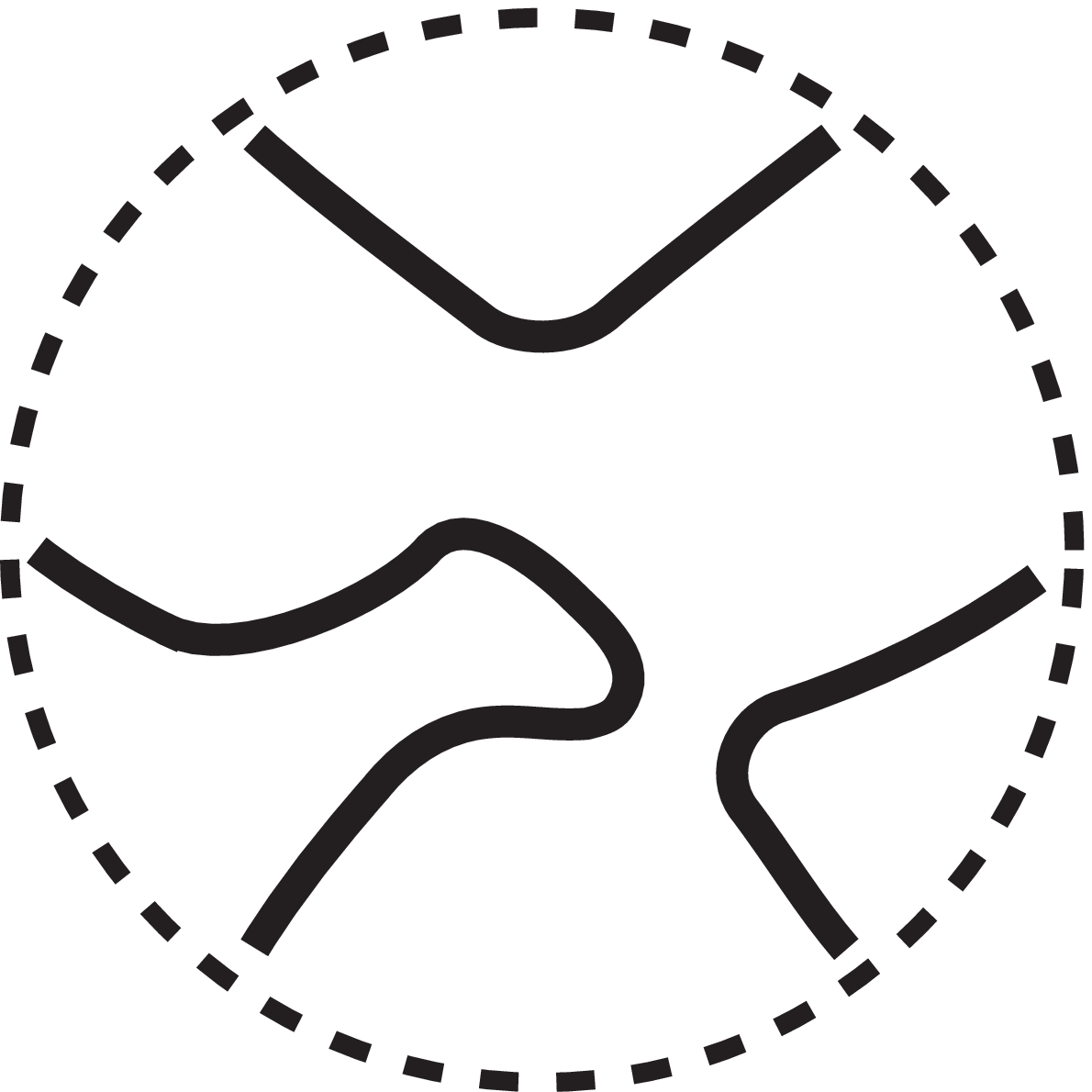}}}
\newcommand{\thirdfitsectthnpr}{\raisebox{-0.35\height}{\includegraphics[width=0.7cm]{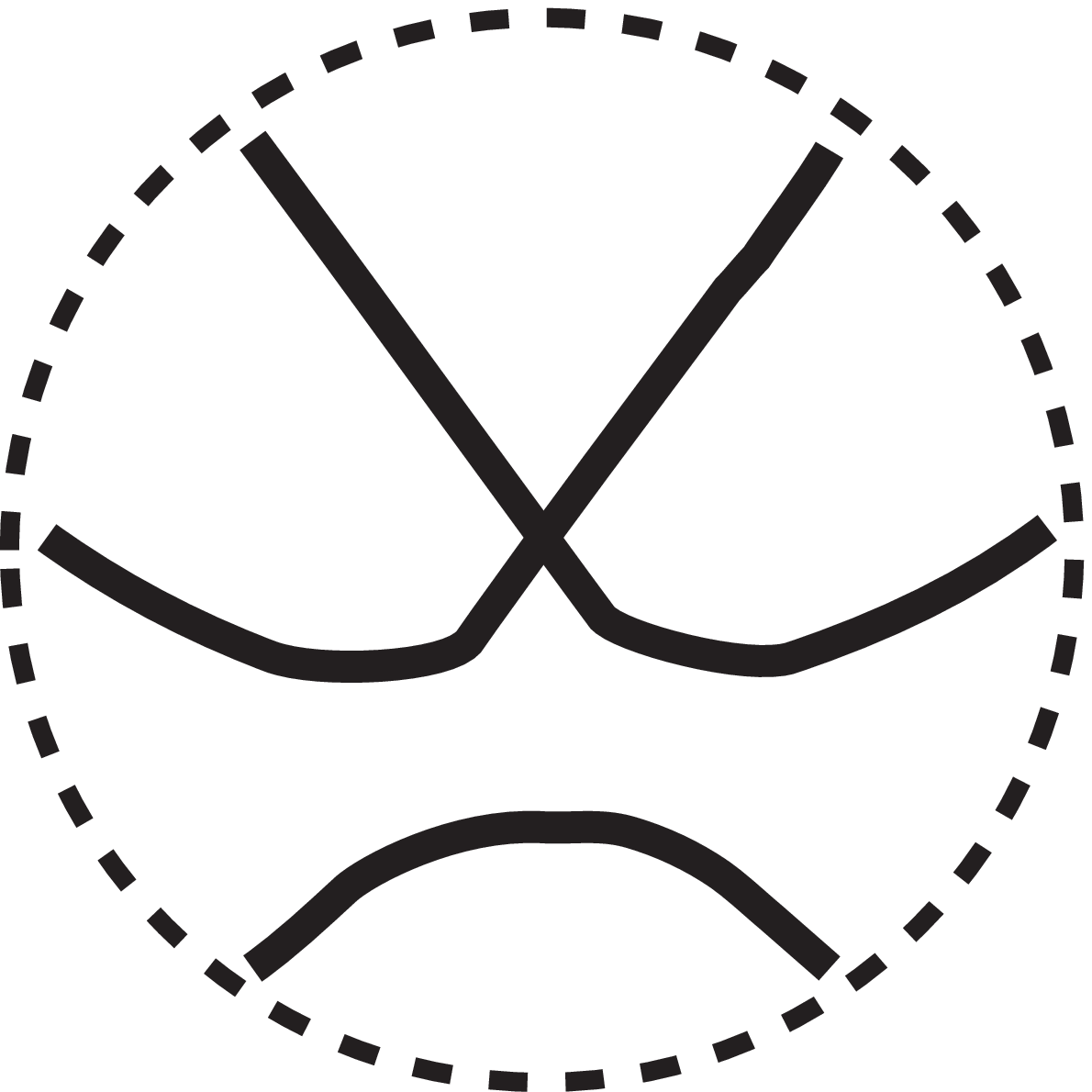}}}
\newcommand{\thirdfiosecnthopr}{\raisebox{-0.35\height}{\includegraphics[width=0.7cm]{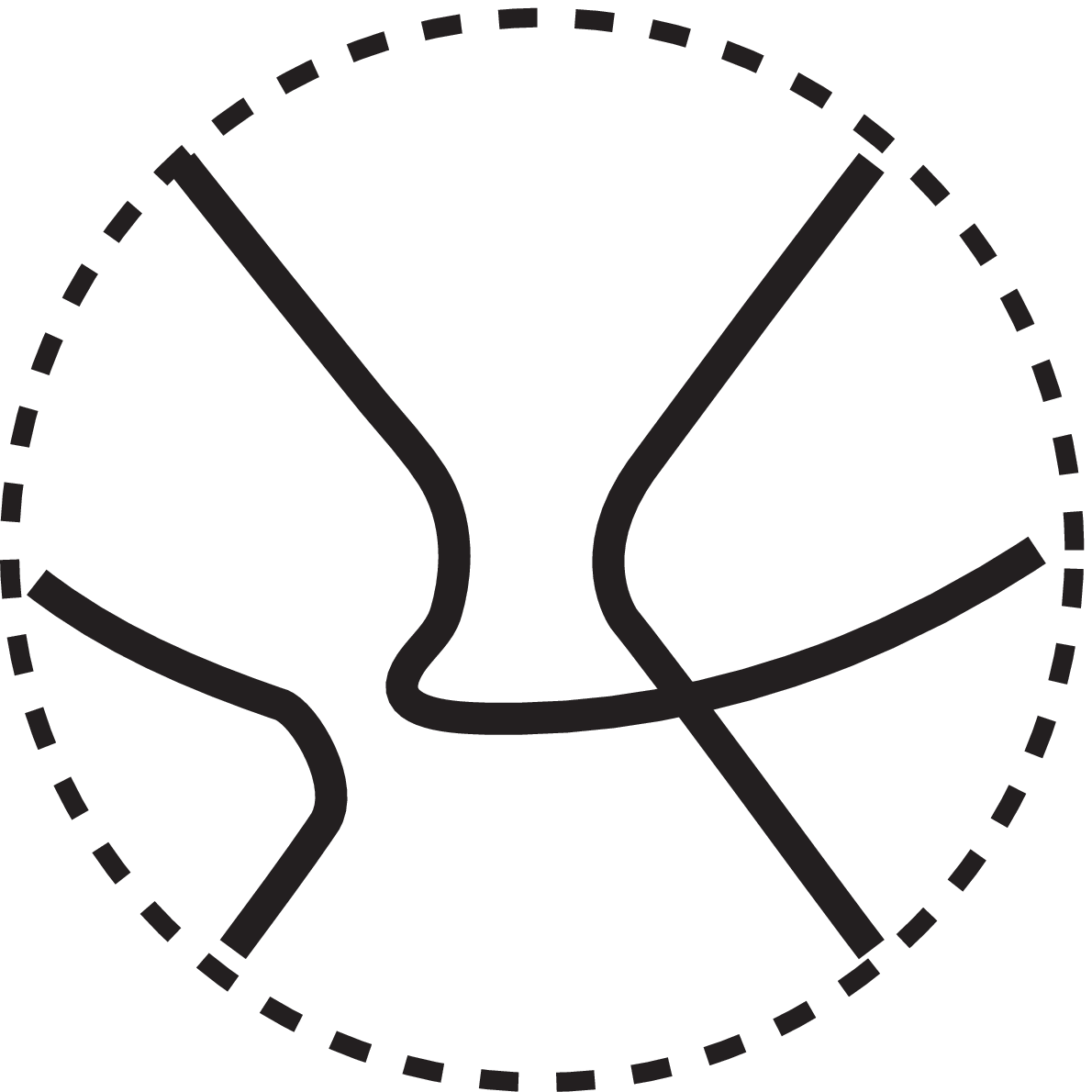}}}
\newcommand{\thirdfitsecnthopr}{\raisebox{-0.35\height}{\includegraphics[width=0.7cm]{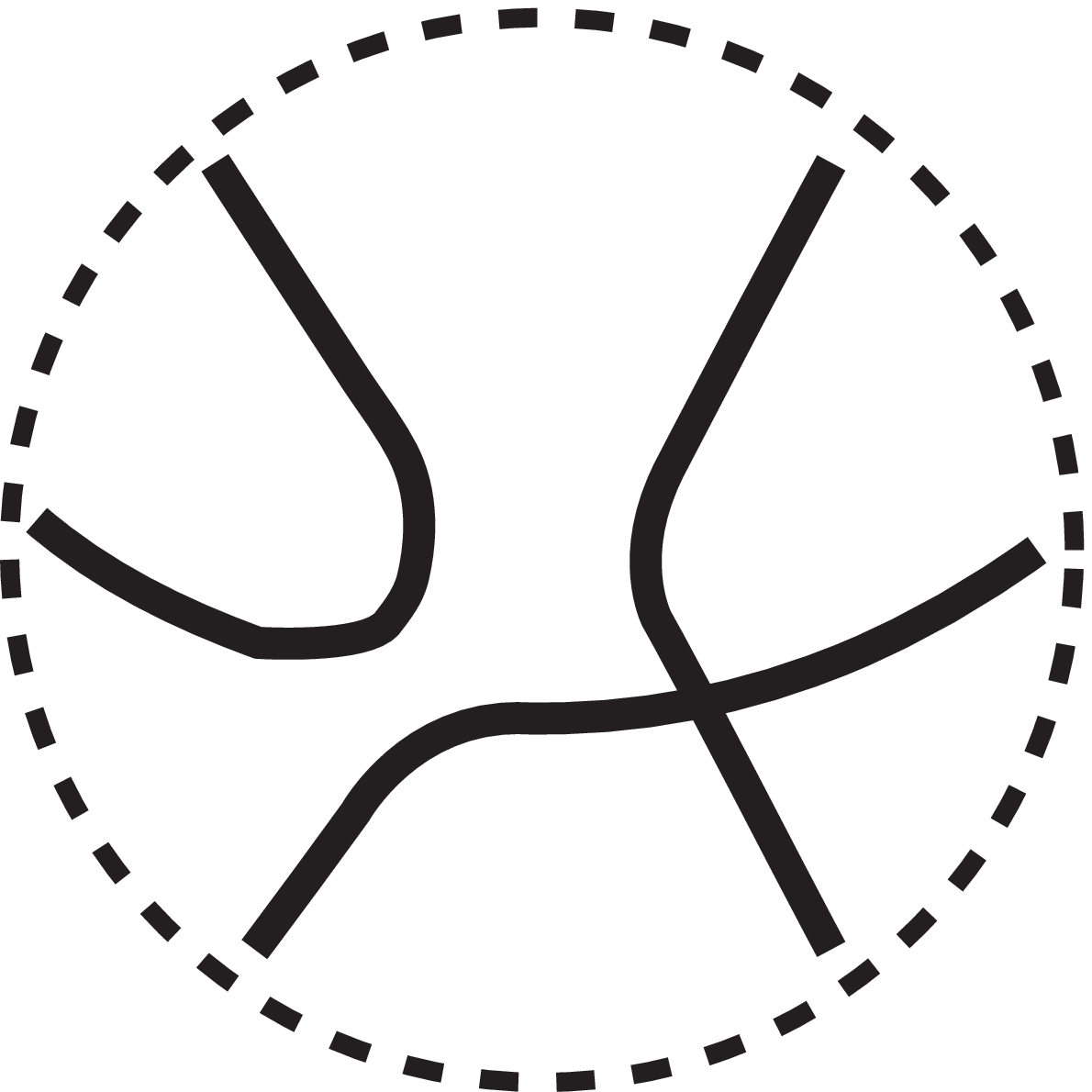}}}
\newcommand{\thirdfinsecothopr}{\raisebox{-0.35\height}{\includegraphics[width=0.7cm]{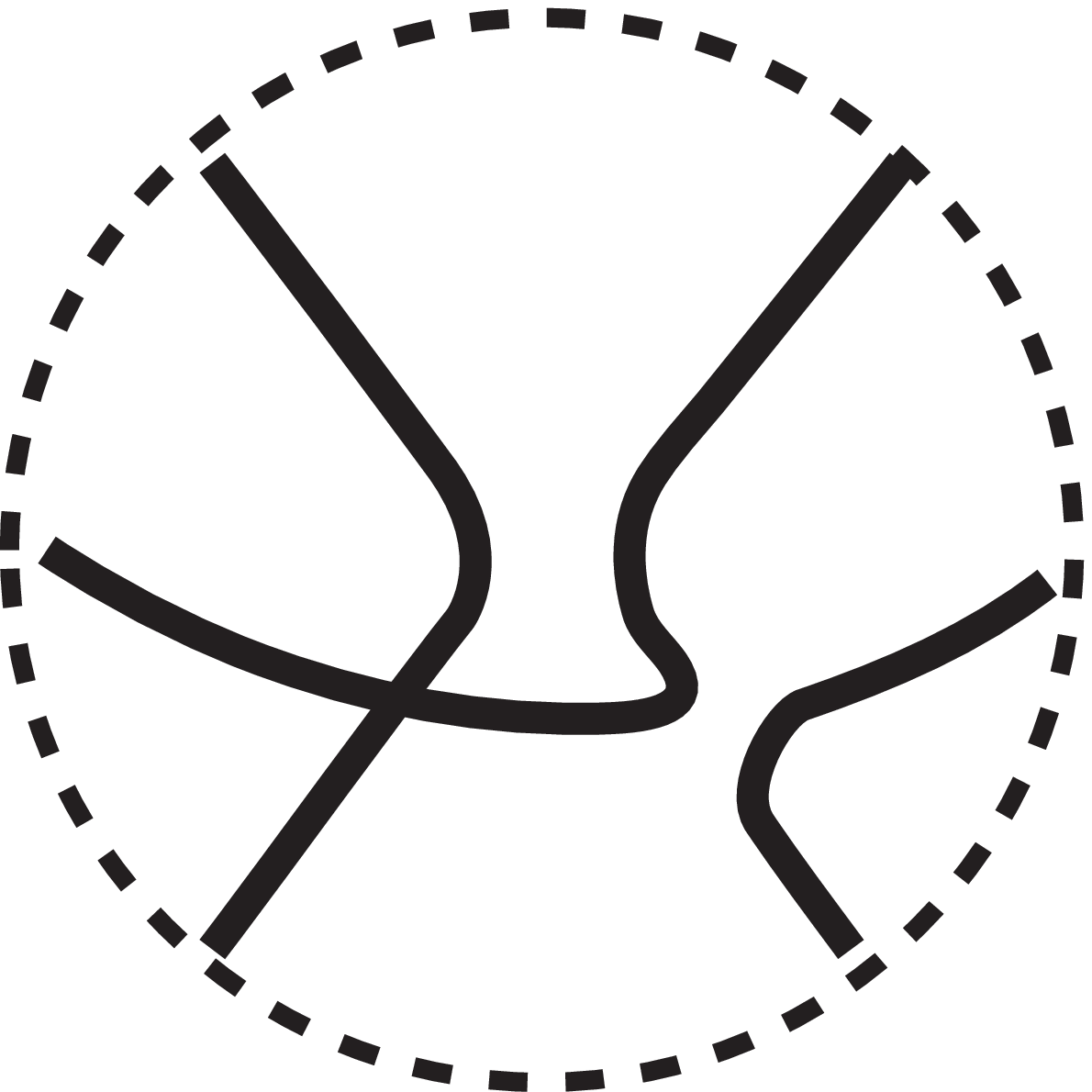}}}
\newcommand{\thirdfinsectthopr}{\raisebox{-0.35\height}{\includegraphics[width=0.7cm]{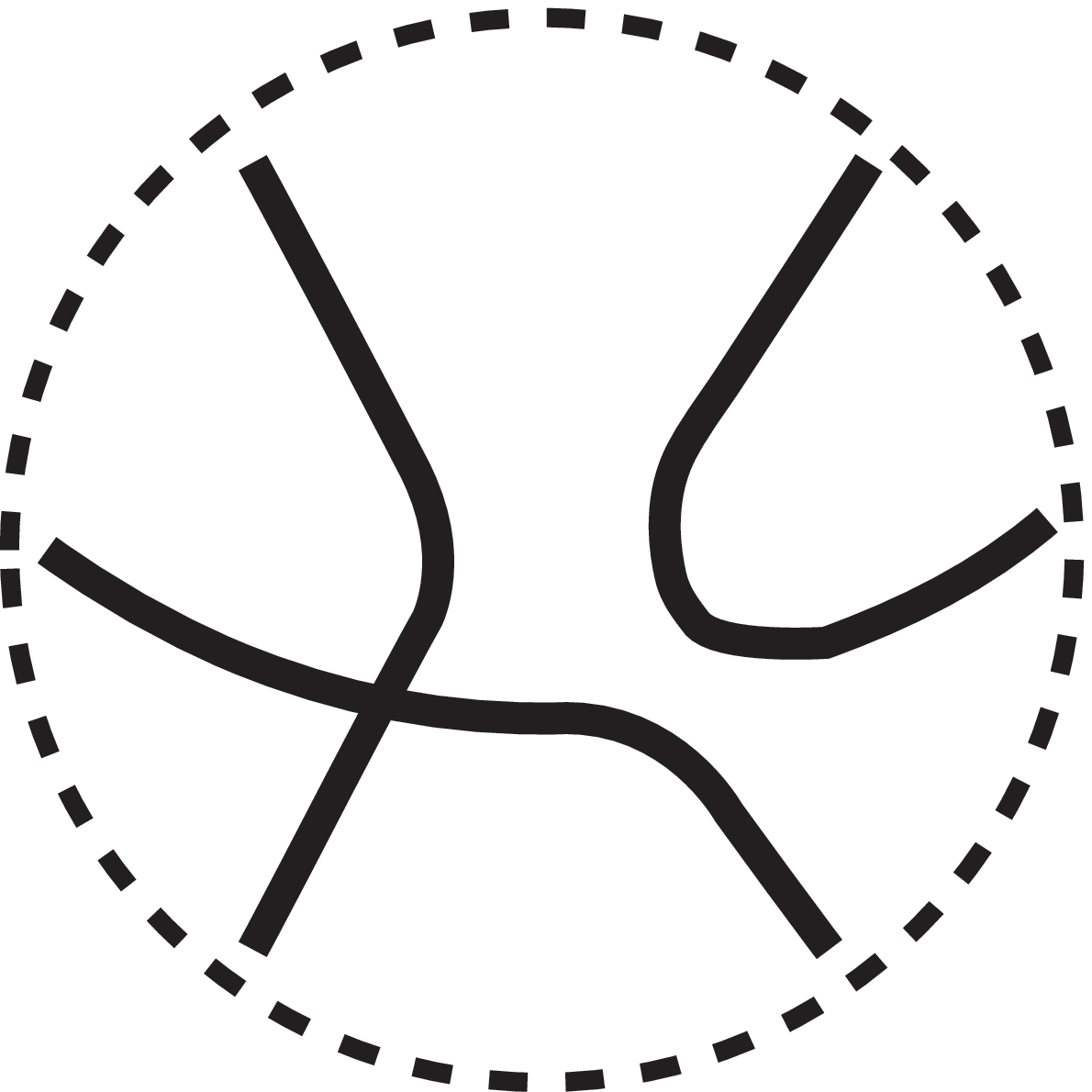}}}
\newcommand{\thirdfitsecnthtpr}{\raisebox{-0.35\height}{\includegraphics[width=0.7cm]{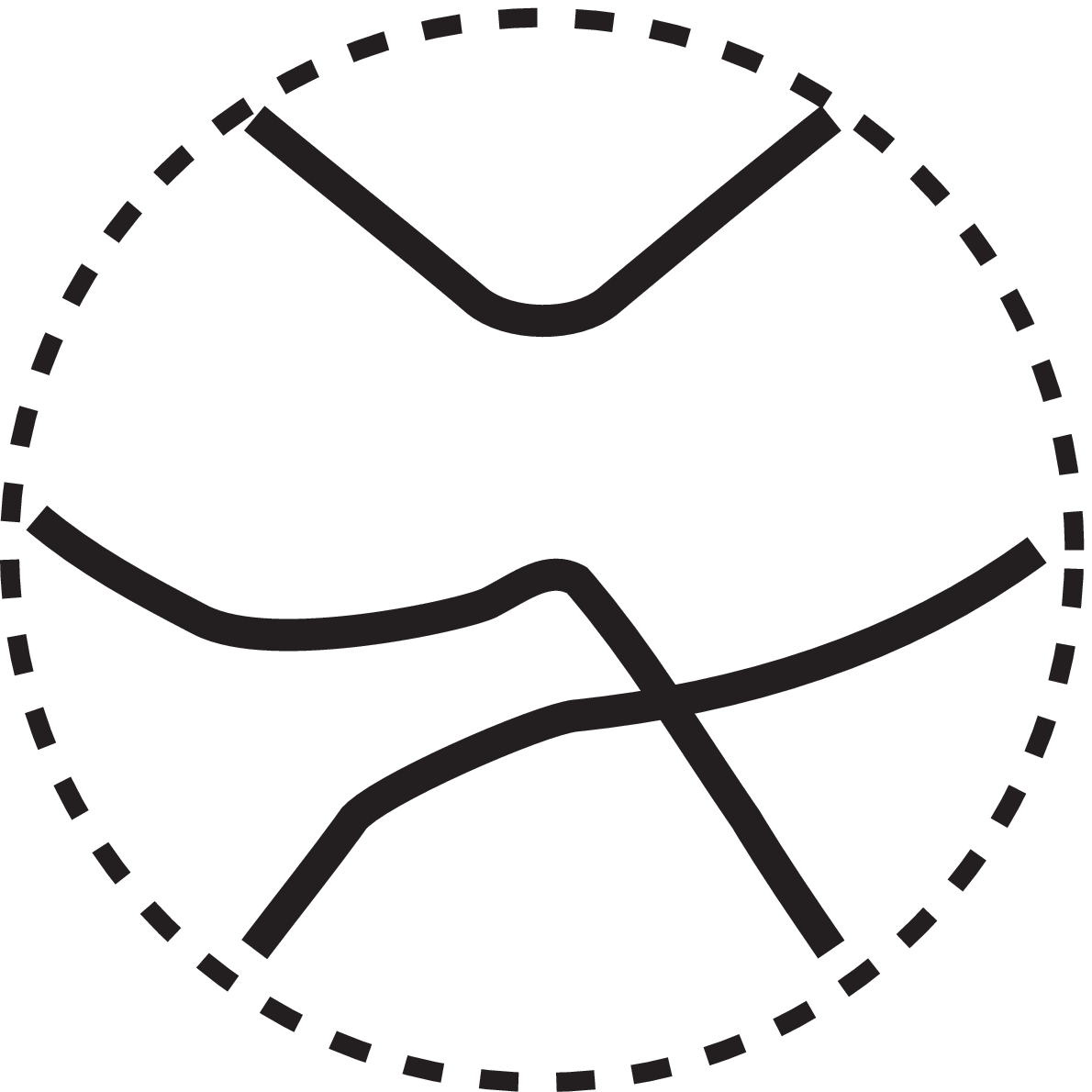}}}
\newcommand{\thirdfinsectthtpr}{\raisebox{-0.35\height}{\includegraphics[width=0.7cm]{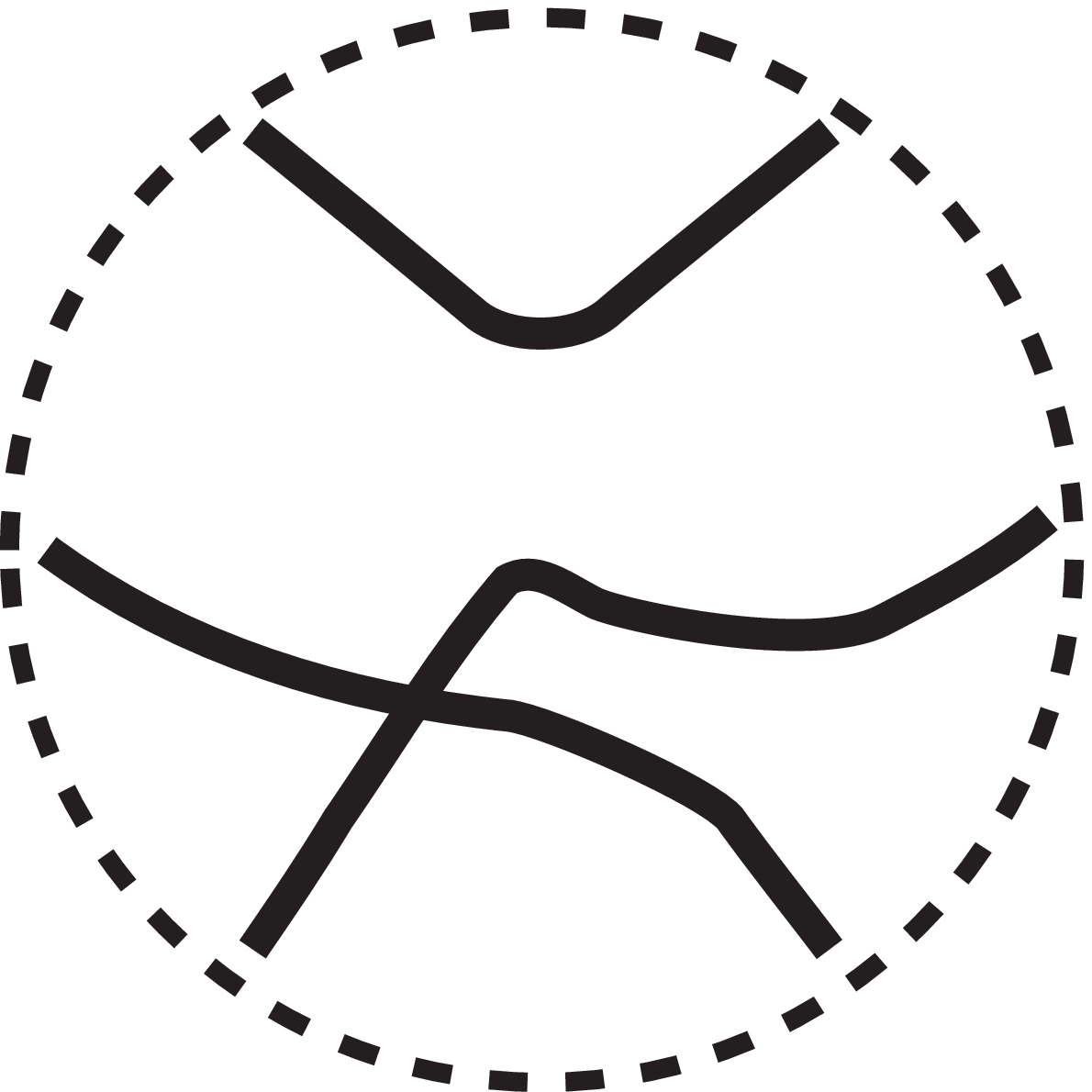}}}
\newcommand{\thirdfiosecnthtpr}{\raisebox{-0.35\height}{\includegraphics[width=0.7cm]{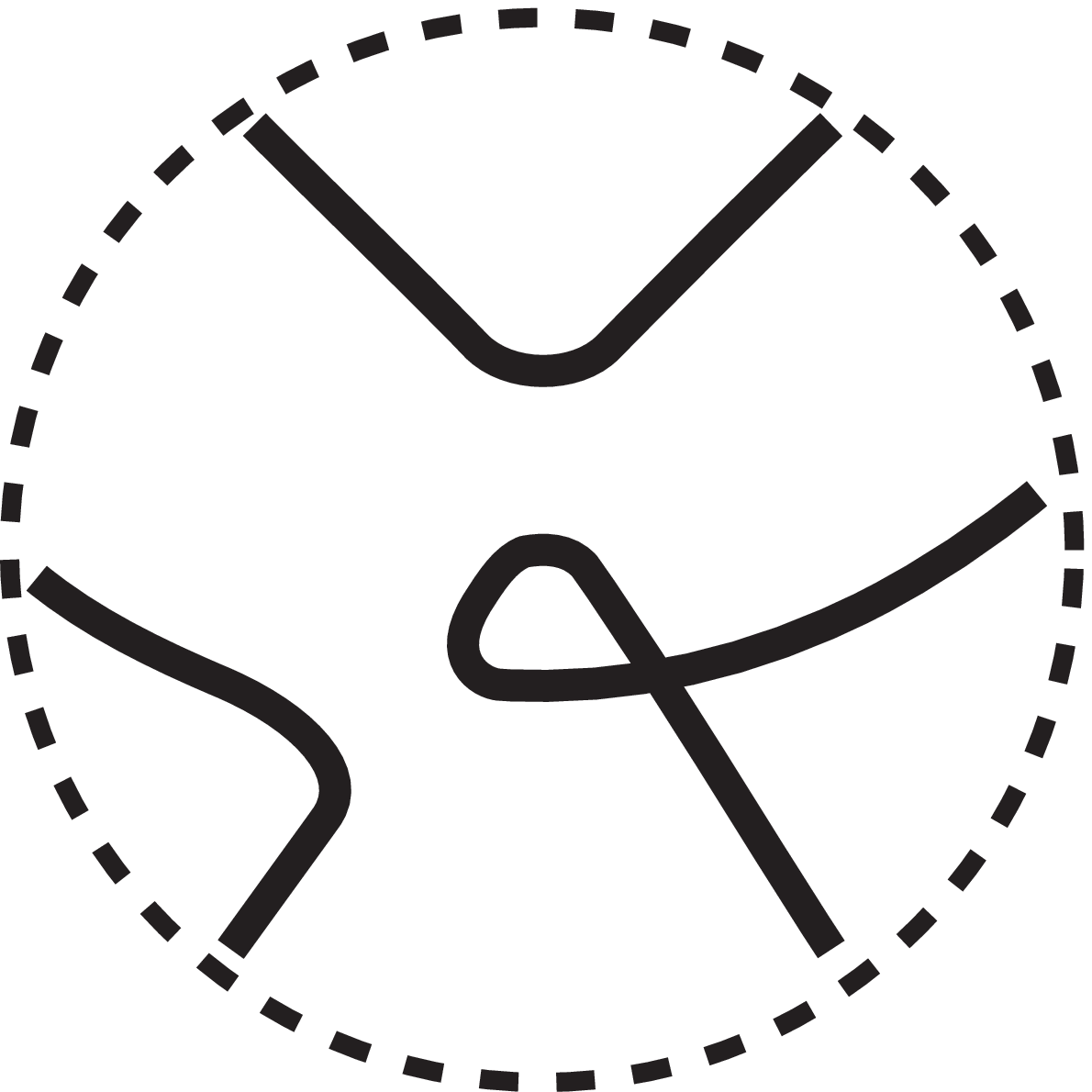}}}
\newcommand{\thirdfinsecothtpr}{\raisebox{-0.35\height}{\includegraphics[width=0.7cm]{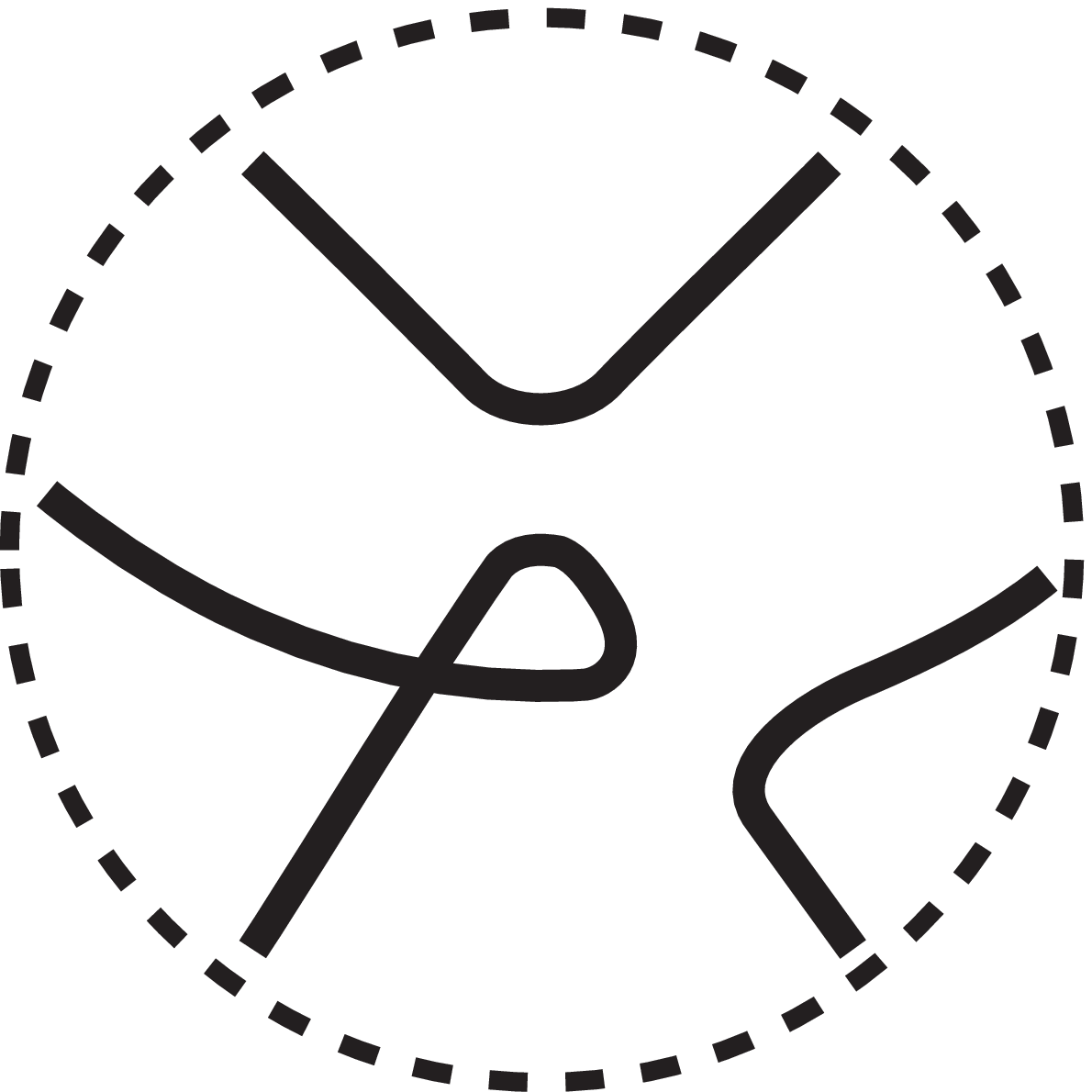}}}
\newcommand{\thirdfiosecothopr}{\raisebox{-0.35\height}{\includegraphics[width=0.7cm]{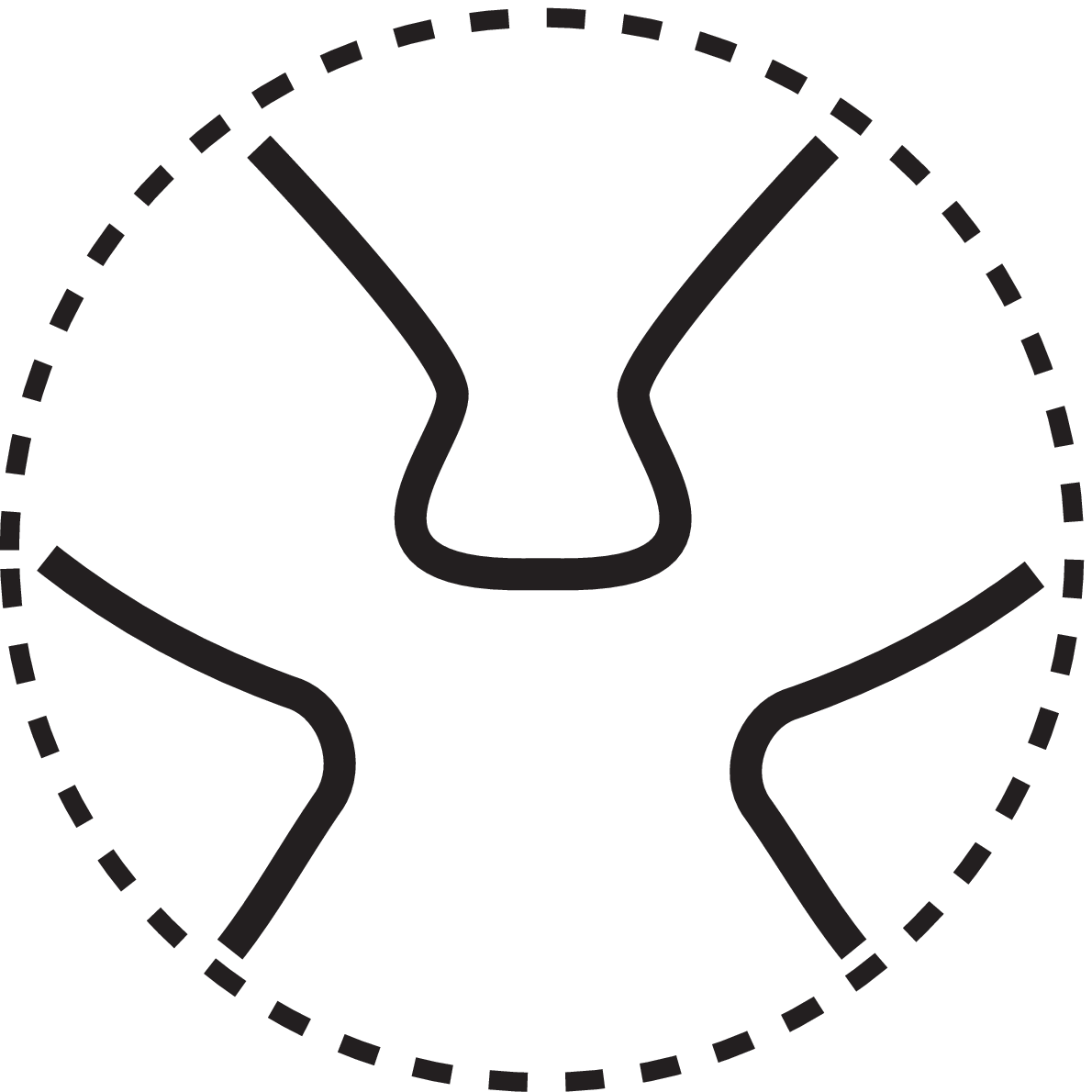}}}
\newcommand{\thirdfitsectthtpr}{\raisebox{-0.35\height}{\includegraphics[width=0.7cm]{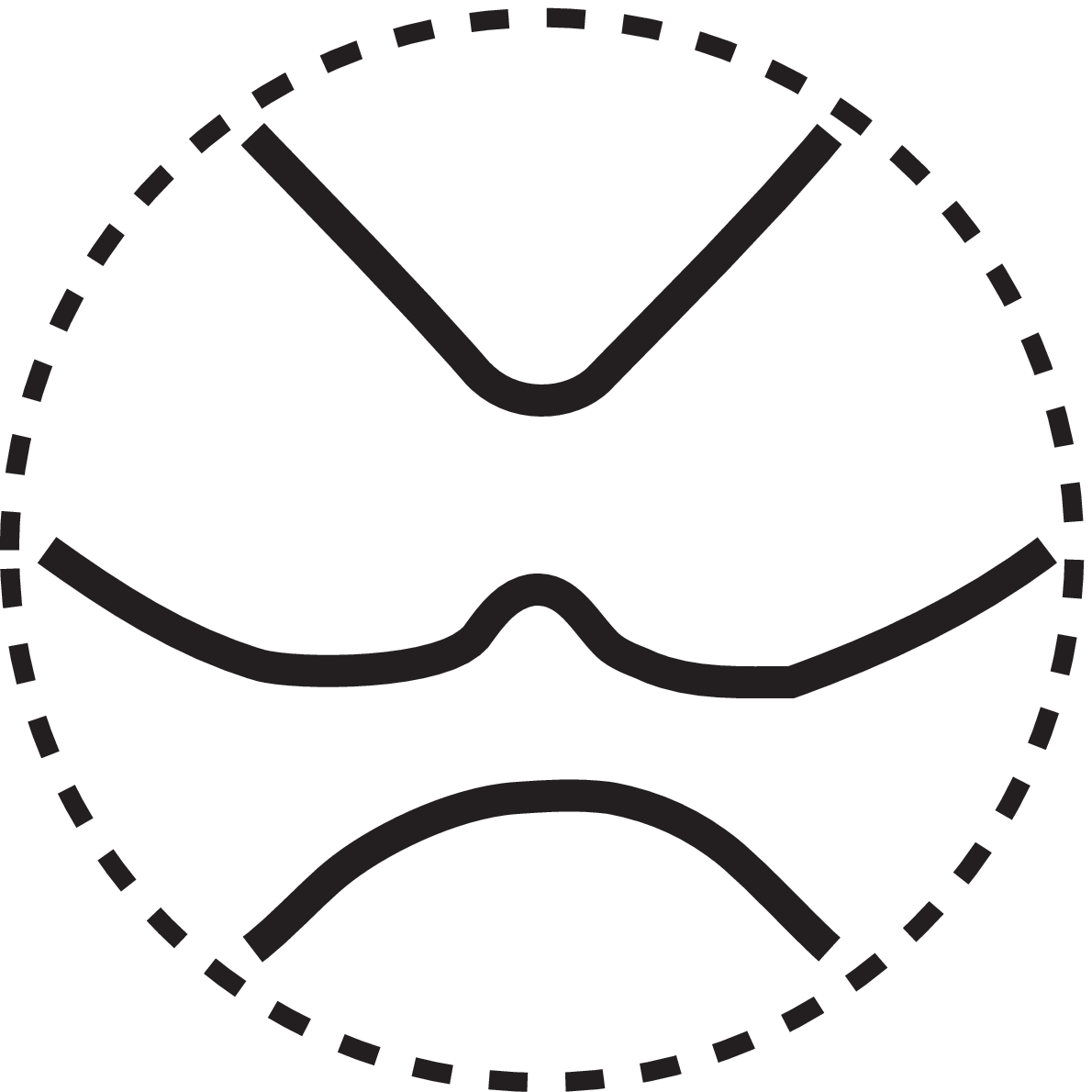}}}
\newcommand{\thirdfiosecothtpr}{\raisebox{-0.35\height}{\includegraphics[width=0.7cm]{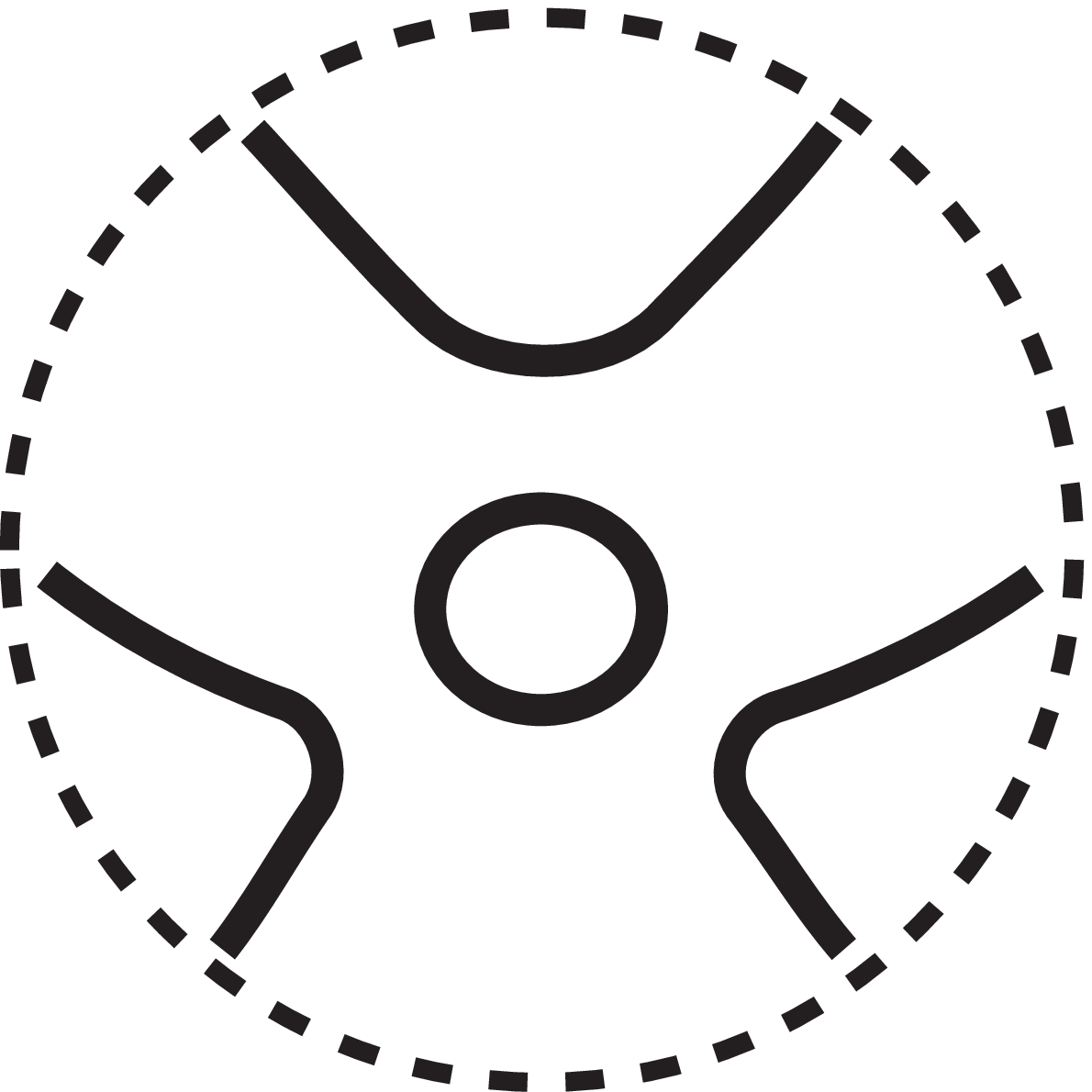}}}
\newcommand{\thirdfitsectthopr}{\raisebox{-0.35\height}{\includegraphics[width=0.7cm]{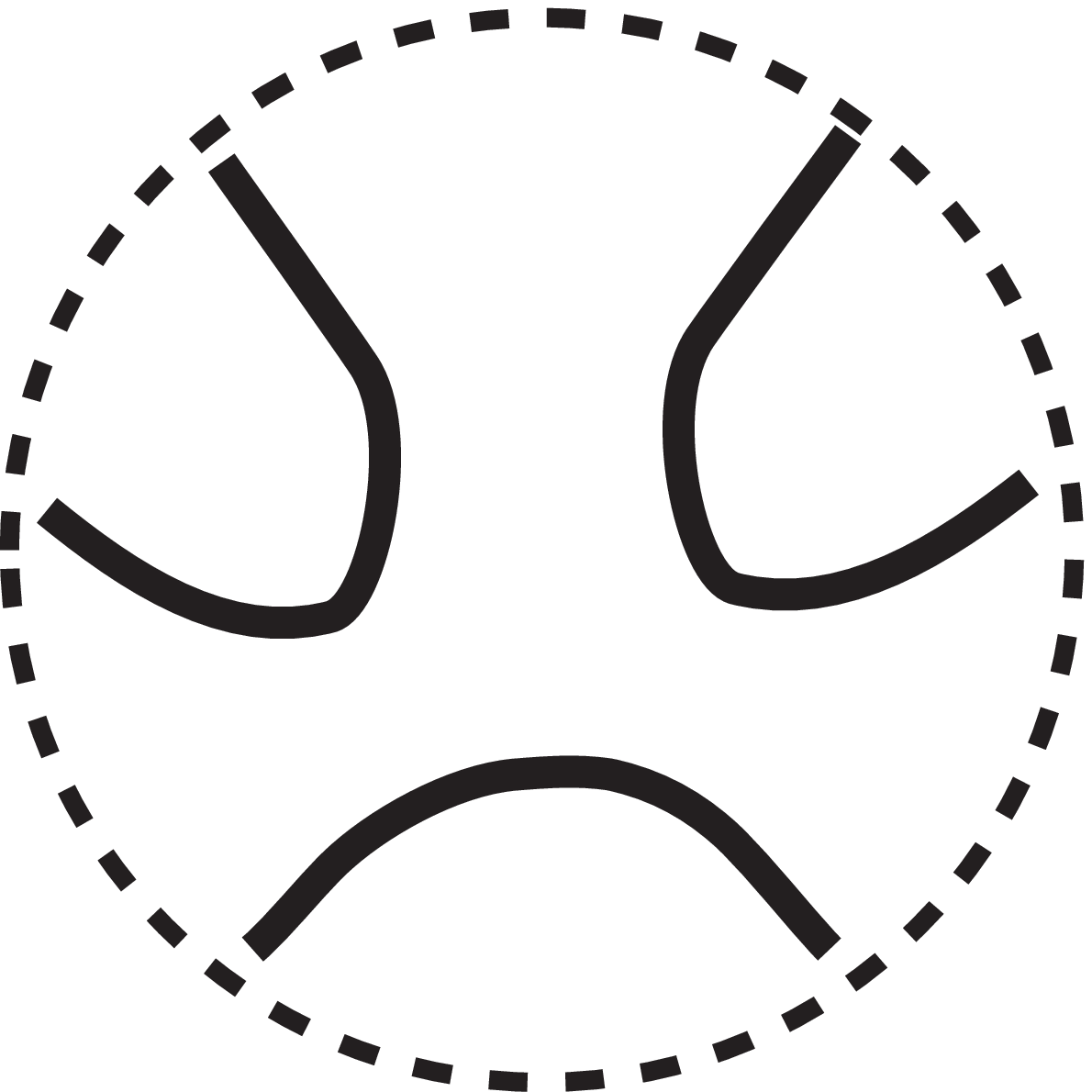}}}
\newcommand{\thirdfitsecothopr}{\raisebox{-0.35\height}{\includegraphics[width=0.7cm]{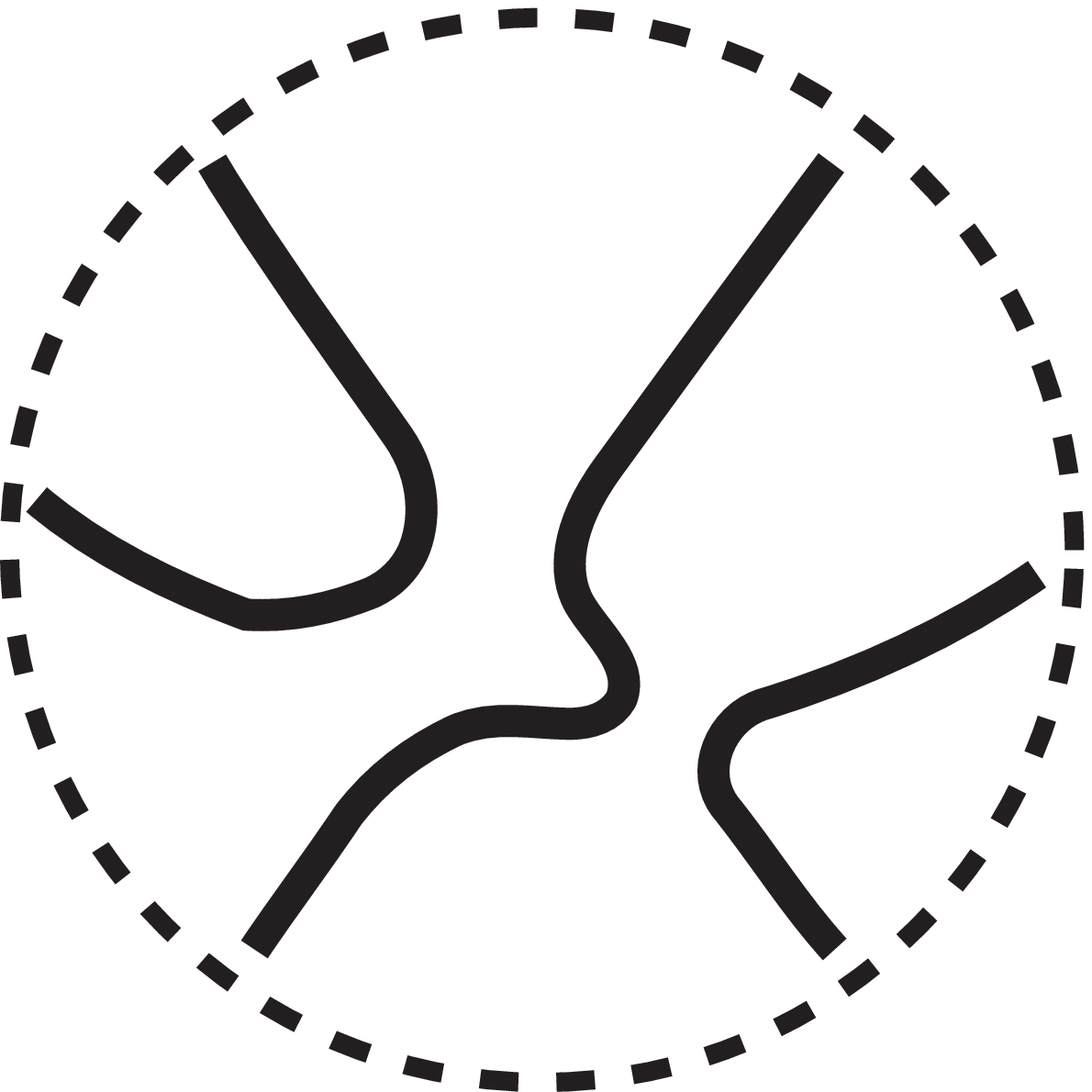}}}

\date{}

\title{Virtual Knot Invariants Arising From Parities}

\author{Denis Petrovich Ilyutko\\
{\em\footnotesize Department of Mechanics and Mathematics,}\\
{\em\footnotesize Moscow State University, Russia}\\
{\em\footnotesize dilyutko@yahoo.com}\\ \\
Vassily Olegovich Manturov\\
{\em\footnotesize Faculty of Science,}\\
{\em\footnotesize People's Friendship University of Russia, Russia}\\
{\em\footnotesize vomanturov@yandex.ru}\\ \\
Igor Mikhailovich Nikonov\\
{\em\footnotesize Department of Mechanics and Mathematics,}\\
{\em\footnotesize Moscow State University, Russia}\\
{\em\footnotesize nikonov@mech.math.msu.su}}

\begin{document}

\renewcommand{\baselinestretch}{0.75}
 \maketitle
\renewcommand{\baselinestretch}{1}

\abstract{In~\cite {FrKn,Sbornik} it was shown that in some knot
theories the crucial role is played by {\em parity}, i.e.\ a
function on crossings valued in $\{0,1\}$ and behaving nicely with
respect to Reidemeister moves. Any parity allows one to construct
functorial mappings from knots to knots, to refine many invariants
and to prove minimality theorems for knots. In the present paper, we
generalise the notion of parity and construct parities with
coefficients from an abelian group rather than $\mathbb{Z}_2$ and
investigate them for different knot theories. For some knot theories
we show that there is the universal parity, i.e.\ such a parity that
any other parity factors through it. We realise that in the case of
flat knots all parities originate from homology groups of underlying
surfaces and, at the same time, allow one to ``localise'' the global
homological information about the ambient space at crossings.

We prove that there is only one non-trivial parity for free knots,
the Gaussian parity. At the end of the paper we analyse the
behaviour of some invariants constructed for some modifications of
parities.}

\newpage
\tableofcontents

\section{Introduction}

In~\cite{Sbornik}, the second named author introduced the notion of
{\em parity} into the study of different knot theories, especially
virtual knots: one distinguishes between two types of crossings,
{\em even} ones and {\em odd} in a way compatible with the
Reidemeister moves so that the parity allows one to refine many
invariants, and construct new invariants. In some sense, odd
crossings are responsible for non-triviality of link diagrams, and
one can prove many minimality and non-triviality theorems starting
with some parity. For every concrete parity, one gets explicit
counterparts of most of theorems proved in~\cite{Sbornik}.

One goal of the present paper is to generalise the notion of parity
and construct the parity with coefficients from an abelian group.
Another goal is to classify parities for different knot theories.

The paper is organised as follows. In the next section we recall the
definitions of different ``knot theories'' and the main
constructions which will be used within the paper.

In Section~\ref {sec:def_par} we introduce the notion of parity with
coefficients in an abelian group. In this section we give the main
examples of parities for different knot theories. We also give a
receipt how to construct parities from homology classes and indicate
how to construct characteristic homology classes from a knot itself;
these classes lead to concrete parities.

Section~\ref {sec:uni_par} is devoted to the universal parity. We
deduce some basic properties of parity from the parity axioms and
show that for some knot theories any parity can be obtained from one
parity, the universal parity.

We conclude the paper with some applications of parity. Firstly, we
construct a functorial map from knots to knots which allows us to
extend some invariants. Secondly, we extend the parity
bracket~\cite{FrKn} to the parity bracket for any parity valued in
$\{0,1\}$.

 \section*{Acknowledgments}
The authors are grateful to V.\,V.~Chernov, A.\,T.~Fomenko,
L.~Kauffman, O.\,V.~Manturov for their interest to this work. The
first author was partially supported by grants of RF President NSh
-- 3224.2010.1, RFBR 10-01-00748-a, RNP 2.1.1.3704, the Federal
Agency for Education NK-421P/108, Ministry of Education and Science
of the Russian Federation 02.740.11.5213 and 14.740.11.0794, the
second author was partially supported by grants of RF President NSh
-- 3224.2010.1, Ministry of Education and Science of the Russian
Federation 14.740.11.0794, the third author was partially supported
by grants of RF President NSh -- 3224.2010.1, RFBR 10-01-00748-a,
RNP 2.1.1.3704, Ministry of Education and Science of the Russian
Federation 02.740.11.5213 and 14.740.11.0794.

\section{Basic definitions}

\subsection {Framed 4-graphs and chord diagrams}

By a {\em graph} we always mean a finite graph; loops and multiple
edges are allowed.

Let $G$ be a graph with the set of vertices $V(G)$ and the set of
edges $E(G)$. We think of an edge as an equivalence class of the two
half-edges forming the edge. From now on, by a {\em 4-graph} we mean
the following generalisation of a four-valent graph: a
$1$-dimensional complex, with each connected component being
homeomorphic either to the circle (with no matter how many
$0$-cells) or to a four-valent graph; by a vertex we shall mean only
vertices of those components which are homeomorphic to four-valent
graphs, and by edges we mean either edges of
four-valent-graph-components or circular components; the latter will
be called {\em cyclic edges}.

We say that a 4-graph is {\em framed} if for every vertex of it, the
four emanating half-edges are split into two pairs. We call
half-edges from the same pair {\em opposite}. We shall also apply
the term opposite to edges containing opposite half-edges. By an
{\em isomorphism} of framed 4-graphs we assume a framing-preserving
homeomorphism. All framed 4-graphs are considered up to isomorphism.
Denote by $G_0$ the framed 4-graph homeomorphic to the circle. By a
{\em unicursal component} of a framed 4-graph we mean either its
connected component homeomorphic to the circle or an equivalence
class of its edges, where the equivalence is generated by the
relation of being opposite.

 \begin {definition}
By a {\em chord diagram} we mean a cubic graph consisting of one
selected Hamiltonian cycle (a cycle passing through all vertices of
the graph) and a set of {\em chords}. We call this cycle the {\em
core circle} of the chord diagram. A chord diagram is {\em oriented}
whenever its core circle is oriented. Edges belonging to the core
circle are called {\em arcs} of the chord diagram. One distinguishes
between {\em oriented} and {\em non-oriented} chord diagrams
depending on whether an orientation of the core circle is given or
not. A chord diagram is depicted on the plane as the Euclidean
circle with a collection of chords connecting end points of chords.
 \end {definition}

For a chord diagram $D$, the corresponding framed 4-graph $G(D)$
with a unique unicursal component is constructed as follows. If the
set of chords of $D$ is empty then the corresponding graph will be
$G_0$. Otherwise, the edges of the graph are in one-to-one
correspondence with the arcs of the chord diagram, and the vertices
are in one-to-one correspondence with chords of $D$. The arcs
incident to the same chord end correspond to the (half-)edges which
are formally opposite at the vertex corresponding to the chord.

The inverse procedure (of constructing a chord diagram from a framed
4-graph with one unicursal component) is evident. In this situation
every connected framed 4-graph can be considered as a topological
space obtained from the circle by identifying some pairs of points.
Thinking of the circle as the core circle of a chord diagram, where
the pairs of identified points will correspond to chords, one
obtains a chord diagram. The chord diagram obtained from a framed
4-graph with one unicursal component in this way is called a {\em
Gauss diagram}.

 \begin {definition}\label {def:link}
We say that two chords $a$ and $b$ of a chord diagram $D$ are {\em
linked} if the ends of the chord $b$ belong to two different
connected components of the complement to the ends of $a$ in the
core circle of $D$. Otherwise we say that chords are {\em unlinked}.

We say that two vertices of a framed 4-graph $G$ are {\em linked} if
the corresponding chords of its Gauss diagram are linked.
 \end {definition}

Define an operation on framed 4-graphs.

 \begin {definition}
By a {\em smoothing} of a framed 4-graph $G$ at a vertex $v$ we mean
any of the two framed 4-graphs obtained from $G$ by removing $v$ and
repasting the edges, see Fig.~\ref {smooth}. The rest of the graph
(together with all framings at vertices except $v$) remains
unchanged.

Note that we may consider further smoothings at {\em several}
vertices. Later on, by a {\em smoothing} we mean a sequence of
smoothings at several vertices.
 \end {definition}

 \begin{figure}
\centering\includegraphics[width=200pt]{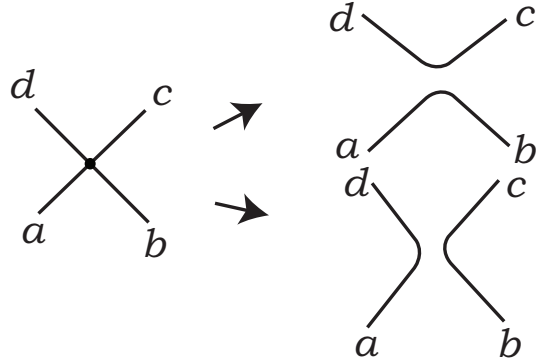} \caption{Two
smoothings of a vertex of a framed graph} \label{smooth}
 \end{figure}

 \subsection {Virtual knots, flat knots and free knots}

In this subsection we consider some knot theories. Let us give main
definitions.

A {\em virtual diagram} is a framed $4$-graph immersed in ${\mathbb
R}^2$ with a finite number of intersections of edges. Moreover, each
intersection is a transverse double point which we call a {\em
virtual crossing} and mark by a small circle, and each vertex of the
graph is endowed with the classical crossing structure (with a
choice for underpass and overpass specified). The vertices of the
graph with that additional structure are called {\em classical
crossings} or just {\em crossings}.

A {\em virtual link} is an equivalence class of virtual diagrams
modulo generalised Reidemeister moves. The latter consist of the
usual Reidemeister moves referring to classical crossings and the
{\em detour move} that replaces one arc containing only virtual
(self-)intersections by another arc of such sort in any other place
of the plane, see Fig.~\ref{detour}.

{\it When drawing framed graphs on the plane, we always assume that
the framing is induced from the plane. In figures depicting moves we
always take into consideration that each side of the move shows a
small area of the diagram homeomorphic to a disc.}

 \begin{figure}
\centering\includegraphics[width=300pt]{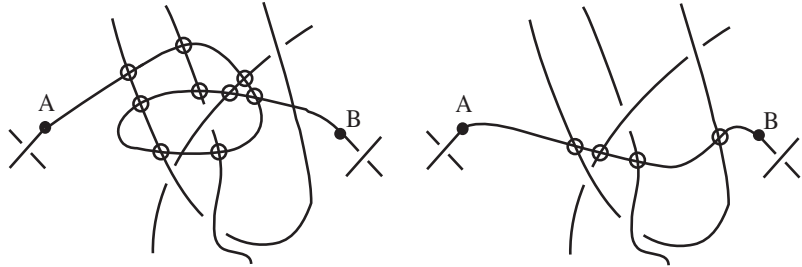} \caption{The
detour move} \label{detour}
 \end{figure}

 \begin {rk}
If we consider embeddings of framed 4-graphs with the classical
crossing structure at each vertex and the usual Reidemeister moves
on them, then we get classical diagrams and classical links.
 \end {rk}

Let us consider an immersion of a framed 4-graph in ${\mathbb R}^2$
and flatten the classical crossings in the Reidemeister moves and
the detour move to double points, i.e.\ we just disregard
over/undercrossing information. We can then define an equivalence
relation on diagrams without overcrossing and undercrossing
structure specified using these flattened Reidemeister moves and
detour move. As a result we get a new object --- a {\em flat knot}.
It is easy to see that flat knots are equivalence classes of virtual
knots modulo transformation swapping over/undercrossing structure.

J.\,S.~Carter, S.~Kamada and M.~Saito showed that we can consider
virtual knots as equivalence classes of embedded framed 4-graphs on
compact oriented surfaces~\cite {CKS}, where two knots are
equivalent if there exists a finite sequence of stabilisations and
Reidemeister moves transforming one knot to the other. The same is
true for flat knots.

Let $K$ be a virtual diagram, and let $S$ be a closed oriented
2-surface. We call the pair $P=(S,K)$ {\em a canonical link surface
diagram} (CLSD) if there exists an embedding of the underlying
framed 4-graph of $K$ into $S$ such that the complement to the image
of this embedding is a disjoint union of 2-cells. Denote by
$\widetilde{S}$ a neighbourhood of the embedding of $K$ in $S$. For
a CLSD, $P=(S,K)$, if there exists an orientation preserving
embedding $f\colon \widetilde{S}\to M$ into a closed oriented
surface $M$, we call $f(K)$ {\em a diagram realisation} of $K$ in
$M$. Two CLSD's $P=(S,K)$ and $P'=(S',K')$ are related by {\em an
abstract Reidemeister move} if there is a closed oriented surface
$M$ and diagram realisations of $K$ and $K'$ in $M$ which are
related by a Reidemeister move in $M$. Two CLSD's are {\em
equivalent} if they are related by a finite sequence of abstract
Reidemeister moves. Following N.~Kamada and S.~Kamada~\cite {KK} one
can construct a bijection
 $$
\psi\colon\{\hbox{virtual link diagrams}\}\to\{\hbox{CLSD's}\}.
 $$
The idea of this map is illustrated in Fig.~\ref {lsd}. Having a
virtual link diagram $K$, we take all classical crossings of it and
associate with a neighbourhood of a crossing two crossing bands ---
a `piece of $2$-surface', and with a virtual crossing we associate a
pair of skew bands (when drawing on the plane it does not matter
which band is over and which one is under). If we connect these
crossings and bands by (non-overtwisted) bands going along edges, we
get a $2$-surface with boundary. Gluing its boundary components by
discs, we get an orientable closed 2-surface. We call $\psi(L)$ a
{\em CLSD associated with a virtual diagram $L$}.

 \begin{figure}
\centering\includegraphics[width=300pt]{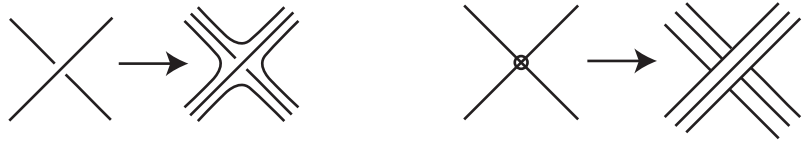} \caption{The local
structure} \label{lsd}
 \end{figure}

We have defined virtual knots and flat knots by using their diagrams
which are obtained by immersions of framed 4-graphs in the plane.
Let us now consider abstract framed 4-graphs and define the
equivalence relation between two graphs using moves analogous to the
Reidemeister moves. Recall that in figures depicting moves on
diagrams we draw only the changing parts; the stable part will be
omitted. In the case of one unicursal component a move can be
represented on a Gauss diagram; it changes the diagram on some set
of arcs; we shall not draw those chords away from the Reidemeister
move being performed; the arcs having no ends of chords taking part
in the move, will be depicted by dotted lines.

 \begin {definition}
{\em The first Reidemeister move} is an addition/removal of a loop,
see Fig.~\ref {1mR}.

{\em The second Reidemeister move} is an addition/removal of a bigon
formed by a pair of edges which are adjacent (not opposite) in each
of the two vertices, see Fig.~\ref {2mR}.

{\em The third Reidemeister move} is shown in Fig.~\ref {3mR}.
 \end {definition}

 \begin{figure}
\centering\includegraphics[width=200pt]{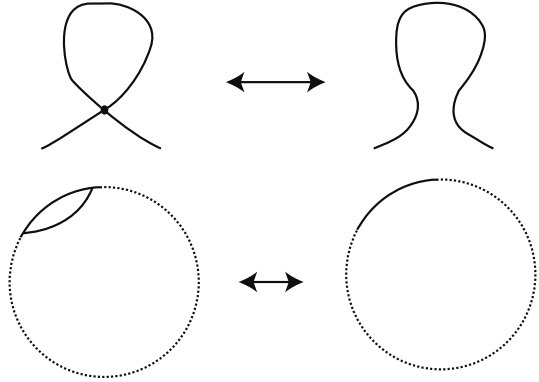} \caption{The first
Reidemeister move and its chord diagram version} \label{1mR}
 \end{figure}

 \begin{figure}
\centering\includegraphics[width=300pt]{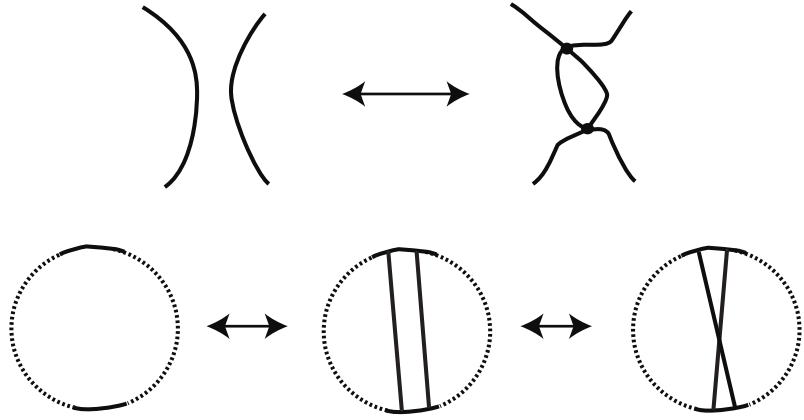} \caption{The
second Reidemeister move and its chord diagram version} \label{2mR}
 \end{figure}

 \begin{figure}
\centering\includegraphics[width=300pt]{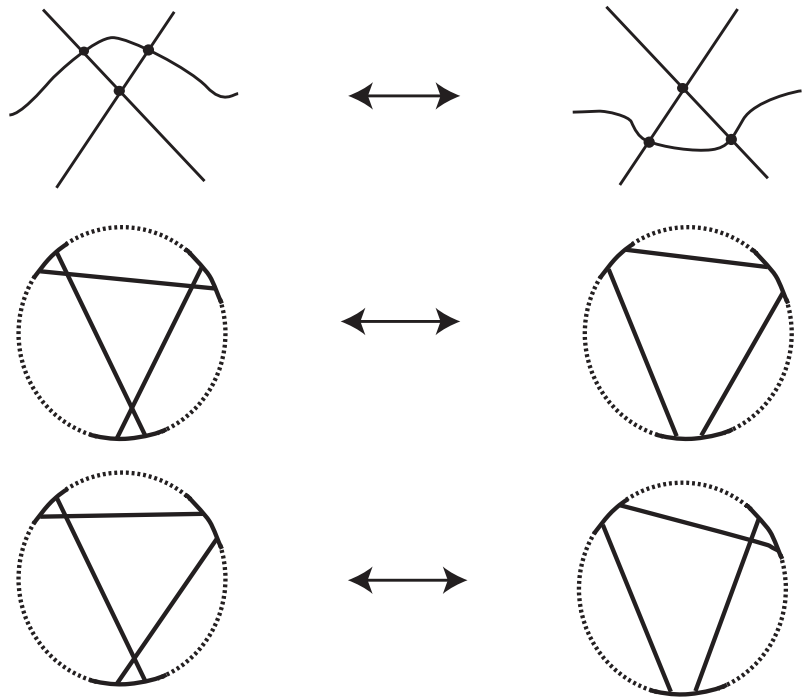} \caption{The third
Reidemeister move and its chord diagram version} \label{3mR}
 \end{figure}

 \begin {rk}
In the cases of the second Reidemeister move and third Reidemeister
move we have one picture for a framed 4-graph and several pictures
for chord diagrams. The number of the pictures for chord diagrams
depends on ways of joining the ends of edges for framed 4-graphs.
 \end {rk}

 \begin {definition}
A {\em free link} is an equivalence class of framed 4-graphs modulo
Reidemeister moves.
 \end {definition}

It is evident that the number of components of a framed 4-graph does
not change after applying a Reidemeister move, so, it makes sense to
talk about the {\em number of components} of a free link.

By a {\em free knot} we mean a free link with one unicursal
component. Free knots can be treated as equivalence classes of Gauss
diagrams by a finite sequence of Reidemeister moves.

The {\em free unknot} (resp., the  {\em free $n$-component unlink})
is the free knot (link) represented by $G_0$ (resp., by $n$ disjoint
copies of $G_0$).

The exact statement connecting virtual knots and free knots sounds
as follows:

 \begin {lemma}
A free knot is an equivalence class of virtual knots modulo two
transformations\/{\em:} classical crossing switches and
virtualisations.

A virtualisation is a local transformation shown in Fig.~\ref
{virt}.
 \end {lemma}

One may think of a virtualisation as way of changing the immersion
of a framed 4-graph in plane.

 \begin{figure}
\centering\includegraphics[width=200pt]{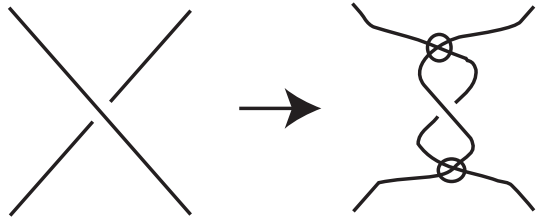} \caption{The
virtualisation move} \label{virt}
 \end{figure}

\section {The Definition of the parity}\label {sec:def_par}

\subsection {Category of knot diagrams}

Let $\mathcal{K}$ be a knot. We shall use the notion of `knot' in
one of the following situations:
 \begin{enumerate}
  \item a free knot;
  \item a homotopy class of curves immersed in a given surface;
  \item a flat knot;
  \item a virtual knot.
 \end{enumerate}

Let us define the category $\mathfrak{K}$ of diagrams of the knot
$\mathcal{K}$. The objects of $\mathfrak{K}$ are diagrams of
$\mathcal{K}$ and morphisms of the category $\mathfrak{K}$ are
(formal) compositions of {\em elementary morphisms}. By an
elementary morphism we mean
 \begin{itemize}
  \item
an isotopy of diagram;
  \item
a Reidemeister move.
 \end{itemize}

 \begin{definition}
A {\em partial bijection} of sets $X$ and $Y$ is a triple
$(\widetilde X,\widetilde Y,\phi)$, where $\widetilde X\subset X$,
$\widetilde Y\subset Y$ and $\phi\colon\widetilde X\to \widetilde Y$
is a bijection.
 \end{definition}

 \begin {rk}
Since the number of vertices of a diagram may change under
Reidemeister moves, there is no bijection between the sets of
vertices of two diagrams connected by a sequence of Reidemeister
moves. To construct any connection between two sets of vertices we
have introduced the notion of a partial bijection which means just
the bijection between the subsets of vertices corresponding to each
other in the two diagrams.
 \end {rk}

Let us denote by $\V$ the vertex functor on $\mathfrak{K}$, i.e.\ a
functor from $\mathfrak{K}$ to the category, objects of which are
finite sets and morphisms are partial bijections. For each diagram
$K$ we define $\V(K)$ to be the set of classical crossings of $K$,
i.e.\ the vertices of the underlying framed 4-graph. Any elementary
morphism $f\colon K\to K'$ naturally induces a partial bijection
$f_*\colon\V(K)\to\V(K')$.

\subsection{A parity}

Now we are going to define a parity with coefficients in an
arbitrary abelian group. In~\cite
{FrKn,FrKnLi,Sbornik,Paritytrieste} the parity with coefficients in
$\mathbb{Z}_2$ was defined. We extend that notion to the case with
an abelian group. Note that one can define a parity with a
non-abelian group, see, for example,~\cite {FunMap}.

Let $A$ be an abelian group.

 \begin{definition}\label {def:parity}
A {\em parity $p$ on diagrams of a knot $\mathcal{K}$ with
coefficients in $A$} is a family of maps $p_K\colon\V(K)\to A$,
$K\in\ob(\mathfrak{K})$, such that for any elementary morphism
$f\colon K\to K'$ the following holds:
  \begin{enumerate}
   \item
$p_{K'}(f_*(v))=p_K(v)$ provided that $v\in\V(K)$ and there exists
$f_*(v)\in\V(K')$;
   \item
$p_K(v_1)+p_K(v_2)=0$ if $f$ is a decreasing second Reidemeister
move and $v_1,\,v_2$ are the disappearing crossings;
   \item
$p_K(v_1)+p_K(v_2)+p_K(v_3)=0$ if $f$ is a third Reidemeister move
and $v_1,\,v_2,\,v_3$ are the crossings participating in this move.
  \end{enumerate}
 \end{definition}

 \begin {rk}
Note that each knot can have its own group $A$, and, therefore,
different knots generally have different parities.
 \end {rk}

 \begin {lemma}\label {lem:par_fir}
Let $p$ be any parity and $K$ be a diagram. Then $p_K(v)=0$ if $f$
is a decreasing first Reidemeister move applied to $K$ and $v$ is
the disappearing crossing of $K$.
 \end {lemma}

 \begin {proof}
Let us apply the second Reidemeister move $g$ to the diagram $K$ as
is shown in Fig.~\ref {par_fir}. We have
 $$
p_{K'}(v_1)+p_{K'}(v_2)=0,\qquad
p_{K'}(g_*(v))+p_{K'}(v_1)+p_{K'}(v_2)=p_K(v)=0.
 $$
 \end {proof}

 \begin{figure}
\centering\includegraphics[width=200pt]{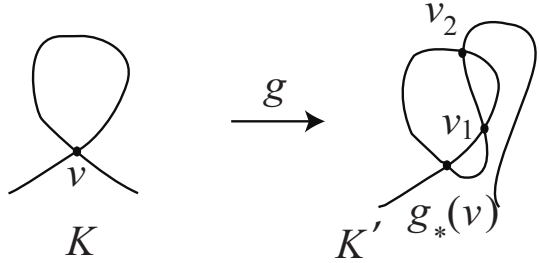}
\caption{Reduction of the first Reidemeister move to the second and
third Reidemeister moves} \label{par_fir}
 \end{figure}

Let us consider some examples of parities for some knot theories.

\subsubsection {Gaussian parity for free, flat and virtual knots}

Let $A=\mathbb{Z}_2$ and $K$ be a virtual (flat) knot diagram
(resp., a framed 4-graph with one unicursal component).

Define the map $gp_K\colon\V(K)\to\mathbb{Z}_2$ by putting
$gp_K(v)=0$ if the number of vertices linked with $v$ is even ({\em
an even crossing}), and $gp_K(v)=1$ otherwise ({\em an odd
crossing}).

 \begin {lemma}\cite{Sbornik}\label {lem:gpar}
The map $gp$ is a parity for free, flat and virtual knots.
 \end {lemma}

 \begin {definition}
The parity $gp$ is called the {\em Gaussian parity}.
 \end {definition}

\subsection {Parity and homology}

A natural source of parities comes from one-dimensional
$\mathbb{Z}_{2}$-(co)homology classes of the underlying surface of a
(virtual) knot. We shall see that if we consider curves in a given
closed 2-surface then (modulo some restrictions) these homology
classes will lead to well-defined parities for knots on such
surfaces (the same works for virtual knots in the thickening of this
surface). The inverse statement is also true: if we take a given
parity on a given surface, then it will lead to a certain
$\mathbb{Z}_{2}$-homology class of the surface.

So, when we have a knot and a fixed surface associated with it, this
gives us a universal receipt of constructing parities and leads us
to the universal parity, see ahead.

However, when passing to virtual knots by means of the
stabilisation, this causes the following trouble: the surface is not
fixed any more and there is no canonical coordinate system on this
surface. Thus, for example, if we work on a concrete torus, we may
fix a coordinate system on it and take the parity corresponding to
the `meridian'. However, when we stabilise and destabilise, we may
destroy the coordinate system on the surface, so it will be
impossible to recover the initial (co)homology class.

To this end, we introduce the notion of a characteristic class for
underlying surfaces corresponding to virtual knots (see rigorous
definition ahead). This is a class which does not depend on anything
except a given virtual knot and behaves nicely on surfaces coming
from diagrams, in particular, under stabilisations/destabilisations.

We give some concrete examples of constructing characteristic
classes.

As we shall see later, this approach does not always help: for the
flat knot diagram (in Fig.~\ref {dec}) on the surface of genus $2$
(the surface is represented as a decagon with opposite sides
identified) is so symmetric, that every characteristic class of it
is trivial (see Example~\ref {sec:def_par}.\ref {exa:dec}), though
when we restrict ourselves to this concrete surface of genus $2$,
there will be non-trivial parities which have non-zero values on the
crossings of the flat knot diagram.

 \begin{figure}
\centering\includegraphics[width=150pt]{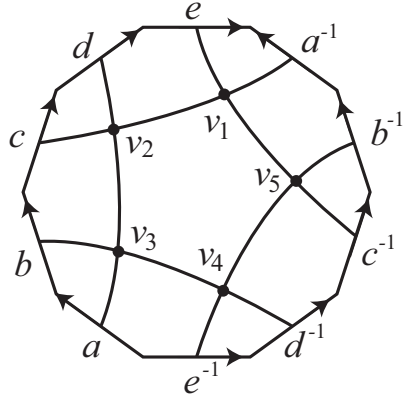} \caption{A knot
in a surface of genus two} \label{dec}
 \end{figure}

To overcome this difficulty, we enlarge the notion of parity.
Instead of a parity valued in $\mathbb{Z}_{2}$, we introduce the
universal parity valued in some linear space over $\mathbb{Z}_{2}$
which is closely related to knot diagrams (the
$\mathbb{Z}_{2}$-homology group of the underlying space with a fixed
basis) and see that all previously known $\mathbb{Z}_{2}$-valued
parities factor through this universal parity.

This parity allows one to work with examples where characteristic
classes and their corresponding parities fail.

First of all we describe a connection between a parity and the
homologies of a surface.

\subsubsection{Homological parity for homotopy
classes of curves generically immersed in a surface}

Let $S$ be a connected closed surface. We consider a free homotopy
class $\mathcal{K}$ of curves generically immersed in $S$.

Let $A=H_1(S,\Z_2)/[\mathcal{K}]$, where $[\cdot]$ denotes a
homological class.

Let $K$ be a framed 4-graph embedded in $S$ representing a curve
from $\mathcal{K}$. For each vertex $v$ we have two halves of the
graph, $K_{v,1}$ and $K_{v,2}$, obtained by smoothing at this
vertex, see Fig.~\ref{smo}.

 \begin{figure}
\centering\includegraphics[width=250pt]{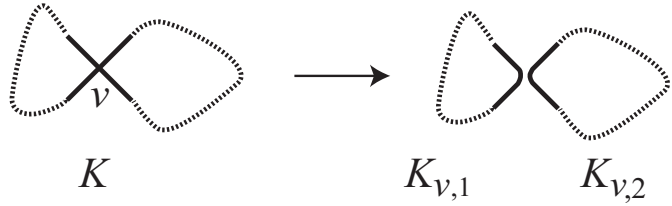} \caption{The graphs
$K_{v,1}$ and $K_{v,2}$} \label{smo}
 \end{figure}

Define the map $hp_K\colon\V(K)\to A$ by putting
$hp_K(v)=[K_{v,1}]$.

 \begin {lemma}\cite{Sbornik}\label {lem:hpar}
The map $hp$ is a parity for homotopy classes of curves generically
immersed in $S$.
 \end {lemma}

 \begin {proof}
From the definition of $A$ it follows that $hp$ does not depend on
the choice of a half for a vertex.

Let $f\colon K\to K'$ be an elementary morphism.

1) Since Reidemeister moves are performed in a small area of $S$
homeomorphic to a disc, we have $hp_{K'}(f_*(v))=hp_K(v)$ provided
that $v\in\V(K)$ and there exists $f_*(v)\in\V(K')$.

2) Let $f$ be a decreasing second Reidemeister move, and let
$v_1,\,v_2$ be the disappearing crossings. Denote by $K_{v_1,1}$ and
$K_{v_2,1}$ the two halves corresponding to the vertices $v_1$ and
$v_2$, see Fig.~\ref {reid2par}.

 \begin{figure}
\centering\includegraphics[width=250pt]{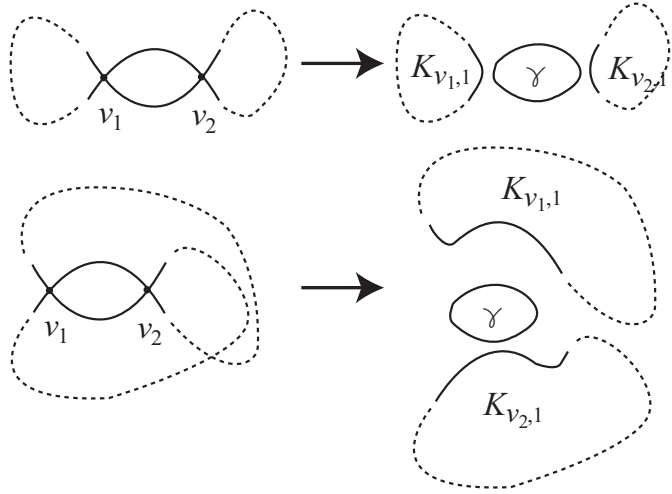} \caption{The
second Reidemeister move} \label{reid2par}
 \end{figure}

We have
 $$
hp_K(v_1)+hp_K(v_2)=[K_{v_1,1}]+[K_{v_2,1}]=[K_{v_1,1}]+[K_{v_2,1}]+[\gamma]=
[K]=0.
 $$

3) Let $f$ be a third Reidemeister move, and let $v_1,\,v_2,\,v_3$
be the crossings participating in this move. Denote by $K_{v_1,1}$,
$K_{v_2,1}$ and $K_{v_3,1}$ the three halves corresponding to $v_1$,
$v_2$ and $v_3$ respectively, see Fig.~\ref {reid3par} (we consider
only one case depicted in Fig.~\ref {reid3par}, all other versions
of the third Reidemeister move can be treated in the same way).

 \begin{figure}
\centering\includegraphics[width=300pt]{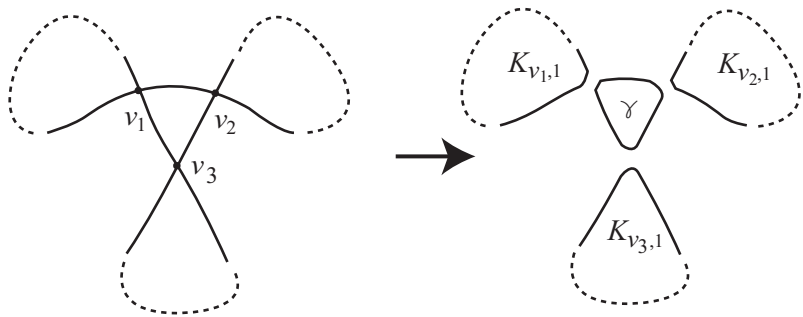}
\caption{The third Reidemeister move} \label{reid3par}
 \end{figure}

We have
 $$
hp_K(v_1)+hp_K(v_2)+hp_K(v_3)=[K_{v_1,1}]+[K_{v_2,1}]+[K_{v_3,1}]
 $$
 $$
=[K_{v_1,1}]+[K_{v_2,1}]+[K_{v_3,1}]+[\gamma]= [K]=0.
 $$
 \end {proof}

\subsubsection{Characteristic classes for framed 4-graphs}

Our next task is to understand the topological nature of parity. As
we shall see, when we deal with curves on a fixed surface, all
possible parities for such curves are closely connected with
(co)homology classes with coefficients in $\mathbb{Z}_{2}$.

However, when we deal with virtual knots or knots in an abstract
thickened surface, then there is no canonical choice of the
coordinate system on the surface, so we can not say what is a
`cohomology class dual to the longitude' or a `cohomology class dual
to the meridian'. Moreover, cohomology classes have to be chosen in
a way compatible with stabilisations.

There is a partial remedy which deals with so-called characteristic
classes. Roughly speaking, a characteristic class is a class on the
surface corresponding to a knot diagram which can be recovered from
the diagram itself. This will be discussed in~\ref {subsub:char_vk}.

Consider a framed $4$-graph $K$ with one unicursal component. The
homology group $H_{1}(K,\Z_{2})$ is generated by halves
corresponding to vertices. If the set of framed $4$-graphs
(possibly, with some further decorations at vertices) is endowed
with a {\em parity}, then we can construct the following cohomology
class $h$: for each of the halves $K_{v,1},\,K_{v,2}$ we set
$h(K_{v,1})=h(K_{v,2})=p_K(v)$, where $p_K(v)$ is the parity of the
vertex $v$. Taking into account that every two halves for each
vertex sum up to give the cycle generated by the whole graph, we
have defined a ``characteristic'' cohomology class $h$ from
$H_{1}(K,\Z_{2})$.

Collecting the properties of this cohomology class we see that

 \begin{enumerate}
  \item
For every framed $4$-graph $K$ we have $h(K)=0$.
  \item
Let $K'$ be obtained from $K$ by a second Reidemeister move
increasing the number of crossings by two. Then for every basis
$\{\alpha_{i}\}$ of $H_{1}(K,\Z_{2})$ there exists a basis in
$H_{1}(K',\Z_{2})$ consisting of one ``bigon'' $\gamma$, the
elements $\alpha'_{i}$ naturally corresponding to $\alpha_{i}$ and
one additional element $\delta$, see Fig.~\ref{r2r3}, left.

\begin{figure}
\centering\includegraphics[width=300pt]{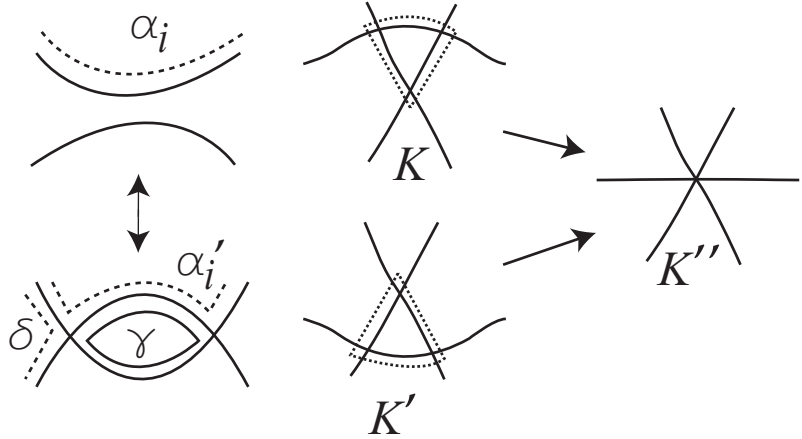} \caption{The
cohomology condition for Reidemeister moves} \label{r2r3}
\end{figure}

Then the following holds: $h(\alpha_{i})=h(\alpha'_{i})$,
$h(\gamma)=0$.

  \item
Let $K'$ be obtained from $K$ by a third Reidemeister move. Then
there exists a graph $K''$ with one vertex of valency  $6$ and the
other vertices of valency $4$ which is obtained from either of $K$
or $K'$ by contracting the ``small'' triangle to the point. This
generates the mappings $i\colon H_{1}(K,\Z_{2})\to
H_{1}(K'',\Z_{2})$ and $i'\colon H_{1}(K',\Z_{2})\to
H_{1}(K'',\Z_{2})$, see Fig.~\ref{r2r3}, right.

We require the following to hold: the cocycle $h$ is equal to zero
for small triangles, besides that if for $a\in
H_{1}(K,\Z_{2}),\,a'\in H_{1}(K',\Z_{2})$ we have $i(a)=i'(a')$,
then $h(a)=h(a')$.
 \end{enumerate}

Note that in 2 no restriction on $h(\delta)$ is imposed.

Thus, every parity for free knots generates some $\Z_{2}$-cohomology
class for all framed $4$-graphs with one unicursal component, and
this class behaves nicely under Reidemeister moves.

The converse is true as well. Assume we are given a certain
``universal'' $\Z_{2}$-cohomology class for all framed 4-graphs
satisfying the conditions 1--3 described above (later we shall
describe the exact definition of the universality). Then it
originates from some {\em parity}. Indeed, it is sufficient to
define the parity of every vertex to be the parity of the
corresponding half. The choice of a particular half does not matter,
since the value of the cohomology class on the whole graph is zero.
One can easily check that parity axioms follow.

This point of view allows one to find parities for those knots lying
in $\Z_{2}$-homologically nontrivial manifolds. For more details,
see~\cite{MM1}.

\subsubsection{Characteristic parities for virtual knots}\label {subsub:char_vk}

Let $K$ be a virtual knot diagram, and let $P=(S,K)$ be the CLSD
associated with the diagram $K$. A {\em checkerboard colouring} of
$S$ with respect to $K$ is a colouring of all the components of
$S\setminus K'$, where $K'$ is the image of the embedding of $K$, by
two colours, say black and white, such that two components of
$S\setminus K'$ being adjacent by an edge of $K'$ have always
distinct colours.

We say that a virtual diagram {\em admits a checkerboard colouring}
or it is {\em checkerboard colourable} if the associated CLSD admits
a checkerboard colouring.

 \begin {theorem}[\cite{IM1}]
If two two virtual diagrams admitting a checkerboard colouring are
equivalent in the category of virtual knots, then they are
equivalent in the category of virtual knots admitting a checkerboard
colouring.
 \end {theorem}

We consider the category of virtual knots admitting a checkerboard
colouring.

 \begin {definition}
A {\em characteristic class} of a knot $\mathcal{K}=\{K\}$ is a
homology class of the surface $S$ associated with a diagram $K$ such
that this class does depend only on $\mathcal{K}$ and behaves nicely
under Reidemeister moves.
 \end {definition}

Consider the group $H_1(S,\Z_2)$ and any element $[\gamma]\in
H_1(S,\Z_2)$. We know that $[K']=0$.

Define the map $\chi_{K,\gamma}\colon\V(K)\to\mathbb{Z}_2$ by
putting $\chi_{K,\gamma}(v)$ to be equal to the intersection number
of $\gamma$ and $K'_{v,1}$, where $K'_{v,1}$ is a half of $K'$
corresponding to~$v$.

Our aim is to construct a homology class of $\gamma$, which does
only depend on a virtual knot generated by $K$, and defines a parity
on the virtual knot.

Consider the following cases.

1) Let $\gamma_a$ be the sum of halves over all classical crossings
(for each classical crossing we take only one half).

2) Let $\mathcal{L}$ be an arbitrary non-trivial free link with two
linked components. At each vertex of $K$ we can consider a smoothing
giving the link diagram with two components. We say that a classical
crossing $v$ of $K$ {\em leads to} $\mathcal{L}$ if after a
smoothing of it and considering the result just as a framing 4-graph
we get a diagram of $\mathcal{L}$. Let us define
 $$
\gamma_{\mathcal{L}}(K)=\sum\limits_{v}K'_{v,1},
 $$
where the sum is taken over all classical crossings giving a diagram
of $\mathcal{L}$.

 \begin {theorem}
The maps $\chi_{K,\gamma_a}$ and $\chi_{K,\gamma_{\mathcal{L}}}$ are
parities for virtual knots with coefficients in $\Z_2$.
 \end {theorem}

 \begin {proof}
We consider only the map $\chi_{K,\gamma_{\mathcal{L}}}$.

Let $f\colon K_1\to K_2$ be an elementary morphism of two knot
diagrams. Consider two CLSD's $P_1=(S_1,K_1)$ and $P_2=(S_2,K_2)$
associated with $K_1$ and $K_2$, respectively. It is sufficient to
consider two cases:

1) If $S_1$ and $S_2$ have the same genus, then the virtue of the
claim follows from Lemma~\ref {lem:hpar}.

2) If the genus of $S_2$ is smaller than the genus of $S_1$ by $1$,
then $f$ is a decreasing second Reidemeister move, see Fig.~\ref
{handle}.

 \begin{figure}
\centering\includegraphics[width=250pt]{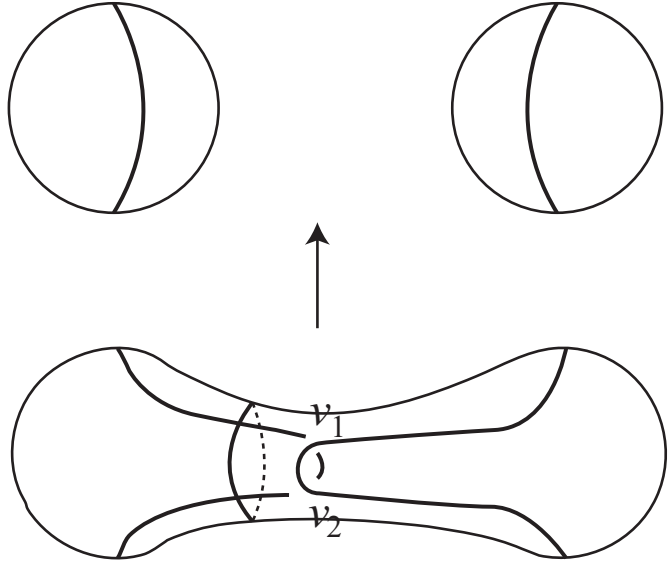} \caption{A
second Reidemeister move adds a handle} \label{handle}
 \end{figure}

As $\mathcal{L}$ is a free link then the classical crossings $v_1$
and $v_2$ participating in the move either simultaneously give the
free link $\mathcal{L}$ or do not give it.

Denote by $K'_i$ the image of $K_i$ in $S_i$. As any half of any
classical crossing of $K'_1$ intersects any half of a classical
crossing distinct from $v_1$ and $v_2$ either at $0$ or precisely
two of $v_1$, $v_2$ and we can pick halves $K'_{v_1,i}$ and
$K'_{v_2,j}$ in such a way that they are homotopic as curves on
$S_1$, we get
 $$
\chi_{K_1,\gamma_{\mathcal{L}}}(v_1)+\chi_{K_1,\gamma_{\mathcal{L}}}(v_2)=0,
 $$
and
$\chi_{K_2,\gamma_{\mathcal{L}}}(f_*(v))=\chi_{K_1,\gamma_{\mathcal{L}}}(v)$
provided that $v\in\V(K_1)$ and there exists $f_*(v)\in\V(K_2)$.
 \end {proof}

 \begin {exa}\label {exa:dec}
Consider the knot diagram $K$ depicted in Fig.~\ref {dec}. It is not
difficult to show that we have the non-trivial map
$hp_K\colon\V(K)\to H_1(S,\Z_2)/[\mathcal{K}]$. The image of this
map is the subgroup of $H_1(S,\Z_2)/[\mathcal{K}]$ generated by $5$
elements $a_i=hp_K(v_i)$ with the relations $a_1+a_2+a_3+a_4+a_5=0$,
cf.~\cite {FunMap}.

But if we want to construct a characteristic parity with the methods
described above we shall fail. $K$ is so symmetric that all five
crossings have the same parity, say $p$. Since we have pentagon, we
get $5p=0$ and, then, $p=0$.
 \end {exa}

Let $K$ be an oriented knot diagram. At each classical vertex we
have one smoothing respecting the orientation on $K$. We can
construct parity $\chi_{K,\mathcal{L}}$ with an oriented free link
$\mathcal{L}$ having two unicursal components by taking the sum only
over classical crossings whose smoothings give $\mathcal{L}$.

Let $\mathcal{L}$ be a non-invertible free link with two unicursal
components~\cite {FrNonInv}, see, for example, Fig.~\ref {noninv}.
If a vertex of an oriented knot leads to $\mathcal{L}$, then this
vertex does most probably not lead to $\mathcal{\overline{L}}$,
where $\mathcal{\overline{L}}$ is the free link obtained from
$\mathcal{L}$ by reversion of the orientation. It means that a
parity does feel an orientation on diagrams.

\begin{figure}
\centering\includegraphics[width=200pt]{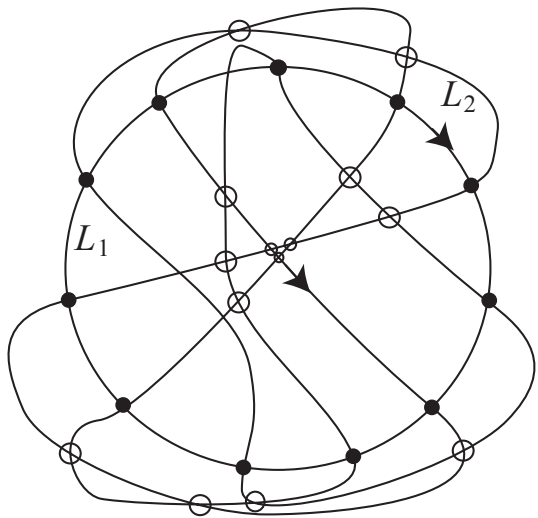} \caption{A
non-invertible free link} \label{noninv}
\end{figure}

\section{The universal parity}\label {sec:uni_par}

In Section~\ref {sec:def_par}, we have given a receipt how to
construct parities from homology classes and indicated how to
construct characteristic homology classes from the knot itself;
these classes lead to concrete parities. However, when we apply such
characteristic classes to the knot in Fig.~\ref {dec}, we see that
all corresponding parities vanish. Nevertheless, the corresponding
flat knot lies in a surface $S_{2}$ of genus $2$ and is not
contractible. So, there are some homology classes (which are
presumably not characteristic) which yield some parity for some
coordinate system of $S_{2}$ which is non-trivial on some vertices
of the knot. The idea of the present section is to construct the
universal parity, cf.~\cite {FunMap}, valued in a certain group
related to the knot rather than the group $\mathbb{Z}_{2}$. This
parity will be universal in the sense that any concrete parity on a
given surface factors through the universal one.

 \begin{definition}
A parity $p_u$ with coefficients in $A_u$ is called a {\em universal
parity} if for any parity $p$ with coefficients in $A$ there exists
a unique homomorphism of group $\rho\colon A_u\to A$ such that
$p_K=\rho\circ (p_u)_K$ for any diagram $K$.
 \end{definition}

Let us describe a construction of the universal parity in general
case.

Let $K$ be a knot diagram. Denote by $1_{K,v}$ the generator of the
direct summand in the group $\bigoplus_{K}\bigoplus_{v\in\V(K)}\Z$
corresponding to the vertex $v$ of $K$.

Let $A_u$ be the group
 $$
A_u=\left(\bigoplus_{K}\bigoplus_{v\in\V(K)}\Z\right) / \mathcal R,
 $$
where $\mathcal R$ is the set of relations of four types:
 \begin{enumerate}
   \item
$1_{K',f_*(v)}=1_{K,v}$ if $v\in\V(K)$ and there exists
$f_*(v)\in\V(K')$;
   \item
$1_{K,v_1}+1_{K,v_2}=0$ if $f$ is a decreasing second Reidemeister
move and $v_1,\,v_2$ are the disappearing crossings;
   \item
$1_{K,v_1}+1_{K,v_2}+1_{K,v_3}=0$ if $f$ is a third Reidemeister
move and $v_1,\,v_2,\,v_3$ are the crossings participating in this
move.
  \end{enumerate}

The map $(p_u)_K$ for each diagram $K$ is defined by the formula
$(p_u)_K(v)=1_{K,v},\ v\in\V(K)$.

If $p$ is a parity with coefficients in a group $A$, one defines the
map $\rho\colon A_u\to A$ in the following way:
 $$
\rho\left(\sum_{K,\ v\in\V(K)} \lambda_{K,v}1_{K,v}\right) =
\sum_{K,\ v\in\V(K)} \lambda_{K,v}p_K(v),\quad \lambda_{K,v}\in\Z.
 $$

The examples below present explicit description of the universal
parity.

\subsection{Free knots}

In the present subsection we show that in the case of the free knot
theory there exists only one non-trivial parity, the Gaussian
parity.

 \begin{theorem}\label {thm:uniq_gau}
Let $\mathcal{K}$ be a free knot. Then the Gaussian parity
{\em(}with coefficients in $\mathbb{Z}_2${\em )} on diagrams of
$\mathcal{K}$ is the universal parity.
 \end{theorem}

 \begin {rk}
Theorem~\ref {thm:uniq_gau} means that for each free knot and for
each parity on it either all vertices are even or they have the
Gaussian parity.
 \end {rk}

This theorem will follow from Lemmas~\ref {lem:sum_par},~\ref
{lem:gaev_to_ev},~\ref {lem:ev_od_ae}.

We consider free knots as Gauss diagrams with an ordered collection
of distinct chords $\{a_1,\dots,a_n\}$. Let us choose a point
distinct from ends of chords on the core circle of a chord diagram.
When going around the circle from the chosen point counter-clockwise
order we will meet each chord end. Denoting each end of a chord by
the same letter as the chord we will get a word, where each letter
corresponds to a chord and occurs precisely twice.

 \begin {definition}
Let $D$ be a chord diagram. We will say that an ordered collection
of chords with numbers $i_1,\dots,i_k$ of $D$ forms a {\em polygon},
if a word, corresponding to $D$, contains the following sequences of
distinct letters $b_{2p-1}b_{2p}$, where
$b_{2p-1},b_{2p}\in\{a_{i_{\sigma(p)}},a_{i_{\sigma(p-1)}}\}$,
$p=1,\dots,k$, for some permutation $\sigma\in S_k$.

The pairs $(b_{2p-1},b_{2p})$ of letters $b_{2p-1},\,b_{2p}$ from
the definition of a polygon are said to be {\em sides} of polygon.
 \end {definition}

 \begin {exa}
Consider the chord diagrams depicted in Fig.~\ref {exmnogo}. The
chords denoted by $a_2,\,a_4,\,a_5,\,a_6,\,a_8$ form a convex
pentagon (left) and a non-convex pentagon (right).

In Fig.~\ref {pol_kn} we depict a hexagon for a knot diagram. The
knot diagram does not intersect the interior of the hexagon.
 \end {exa}

 \begin{figure}
\centering\includegraphics[width=250pt]{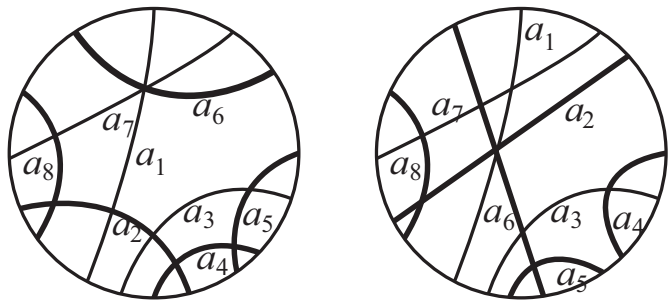}
\caption{Pentagons} \label{exmnogo}
 \end{figure}

 \begin{figure}
\centering\includegraphics[width=150pt]{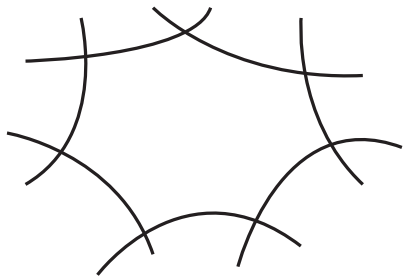} \caption{A
hexagon} \label{pol_kn}
 \end{figure}

 \begin {lemma}\label {lem:sum_par}
For every parity and any chord diagram the sum of the parities of
chords forming a polygon is equal to $0$. \label{thy}
 \end {lemma}

 \begin {rk}
The claim of Lemma~\ref {lem:sum_par} can be taken as a definition
of a parity, see~\cite {FunMap}.
 \end {rk}

 \begin {proof}
Let $p$ be an arbitrary parity on chord diagrams of the free knot
$\mathcal{K}$, and let $D$ be a chord diagram representing
$\mathcal{K}$. Let us prove the claim of the lemma by induction over
the number of sides of a polygon.

{\it The induction base}. The virtue of the claim for a loop, bigon,
triangle follows from Lemma~\ref {lem:par_fir} and Definition~\ref
{sec:def_par}.\ref {def:parity}, respectively.

{\it The induction step}. Assume that the claim is true for
$(k-1)$-gons. Let us consider an arbitrary $k$-gon
$a_{i_1}a_{i_2}\ldots a_{i_k}$.

Let us apply the second Reidemeister move to the chord diagram $D$
by adding two chords $b$ and $c$, see Fig.~\ref {mnogo1} (in
Fig.~\ref {mnogo1} we have depicted the three possibilities of
applying the second Reidemeister move depending on the ends of
chords $a_{i_1},\,a_{i_2},\,a_{i_3},\,a_{i_k}$).

 \begin{figure}
\centering\includegraphics[width=250pt]{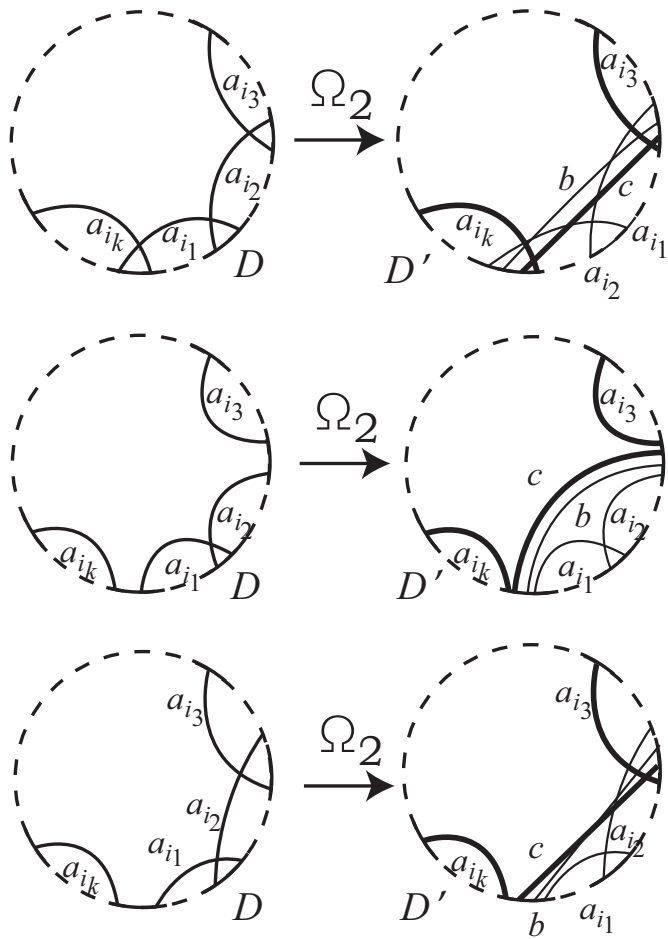} \caption{The
second Reidemeister move} \label{mnogo1}
 \end{figure}

As a result we shall obtain the new chord diagram $D'$ and the
$(k-1)$-gon $c\,a_{i_3}a_{i_4}\ldots a_{i_k}$ and the triangle
$b\,a_{i_1}a_{i_2}$. By the induction hypothesis, we have
  $$
p_{D'}(c)+\sum\limits_{j=3}^kp_{D'}(a_{i_j})=0,\quad
p_{D'}(b)+p_{D'}(a_{i_1})+p(a_{i_2})=0,\quad p_{D'}(b)+p_{D'}(c)=0.
 $$
Therefore,
 $$
\sum\limits_{j=1}^kp_{D'}(a_{i_j})=\sum\limits_{j=1}^kp_{D}(a_{i_j})=0.
 $$
 \end {proof}

 \begin {rk}
If we work with knot diagrams, then the corresponding picture for
Lemma~\ref {lem:sum_par} looks like as is shown in Fig.~\ref
{reid2_mno}.
 \end {rk}

 \begin{figure}
\centering\includegraphics[width=250pt]{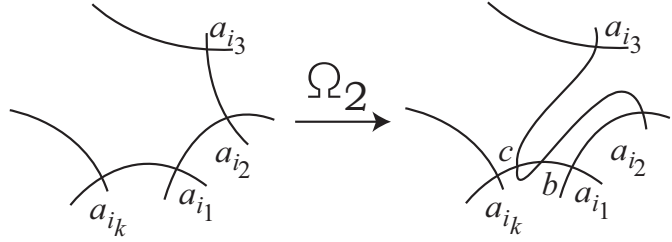} \caption{The
second Reidemeister move} \label{reid2_mno}
 \end{figure}

Let us pass from the free knot theory to the flat knot theory and
the virtual knot theory. Since bigons and triangles participating in
Reidemeister moves can be spanned by discs we get the following

 \begin {crl}\label {st:hom_par}
For every parity and any flat {\em(}virtual\/{\em)} knot diagram the
sum of the parities of crossings forming a polygon, which is spanned
by a disc in the underlying surface, is equal to $0$.
 \end {crl}

By using virtualisation moves we can transform any polygon to a
polygon which is spanned by a disc in the underlying surface. As a
result we get the following

 \begin {crl}
If we consider the theory of pseudo-knots, i.e.\ the theory of
virtual knots modulo the virtualisation move, then Lemma~\ref
{lem:sum_par} remains true in this theory too, that is the existence
of the writhe number gives us no additional information.
 \end {crl}

 \begin {lemma}\label {lem:gaev_to_ev}
For a free knot {\em(}pseudo-knot\/{\em)} with a diagram $K$ and an
arbitrary parity $p$ we have $p_K(a)=0$ if $gp_K(a)=0$.
 \end {lemma}

 \begin {proof}
Let $p$ be a parity, and let $a$ be a chord of a chord diagram $D$
with $gp_D(a)=0$. Let us consider the two halves of the core circle
of $D$, which are obtained by removing the chord $a$. Since
$gp_D(a)=0$ each half-circle corresponding to $a$ contains an even
number of ends of chords. Let us apply the induction over the number
of ends of chords.

{\it The induction base}\/: If the number of ends on any half-circle
is equal to $0$, then $p_D(a)=0$ by using the property of the first
Reidemeister move.

{\it The induction step}\/: Assume that for any chord $d$ of $D$
with $gp_D(d)=0$ such that a half-circle contains less than $n=2k$
ends of chords, we have $p_D(d)=0$. Let us consider a chord $a$ such
that one of its half-circles, $K_{a,1}$, contains exactly $n$ ends
of chords and the other one, $K_{a,2}$, contains more than or equal
to $n$ ends.

Let us orient $D$ in counterclockwise manner and consider the
following two cases.

1) The first two ends in $K_{a,1}$ belong to two distinct chords
$a_1,\,a_2$, see Fig.~\ref {g_p1}. Apply the second increasing
Reidemeister move by adding a pair of chords $b,\,b'$ in such a way
that the half-circle corresponding to $b'$ would contain the set of
ends lying in $K_{a,1}$ minus the first ends of $a_1,\,a_2$, see
Fig.~\ref {g_p2} (above). Let us show that $p_{D'}(a)+p_{D'}(b)=0$
in the new chord diagram $D'$. Let us add the pair of chords
$c,\,c'$ to form the triangle $a_1a_2c$, see Fig.~\ref {g_p2}
(below). Then $p_{D''}(a_1)+p_{D''}(a_2)+p_{D''}(c)=0$ in $D''$.
Moreover, we have the pentagon $aa_1ca_2b$ and, therefore, the
following equality holds (Lemma~\ref {lem:sum_par})
 $$
p_{D''}(a)+p_{D''}(a_1)+p_{D''}(c)+p_{D''}(a_2)+p_{D''}(b)=0.
 $$
We get $p_{D''}(a)+p_{D''}(b)=0$ and $p_{D'}(a)+p_{D'}(b)=0$. In the
half-circle corresponding to $b'$ the number of ends is less than
the number of ends in the half-circle corresponding to $a$. By the
induction hypothesis, we get $p_{D'}(b)=p_{D'}(b')=0$, and
$p_D(a)=0$.

  \begin{figure}
\centering\includegraphics[width=230pt]{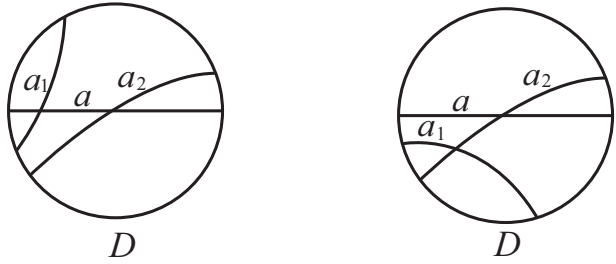} \caption{The
Gaussian parity zero} \label{g_p1}
 \end{figure}

  \begin{figure}
\centering\includegraphics[width=230pt]{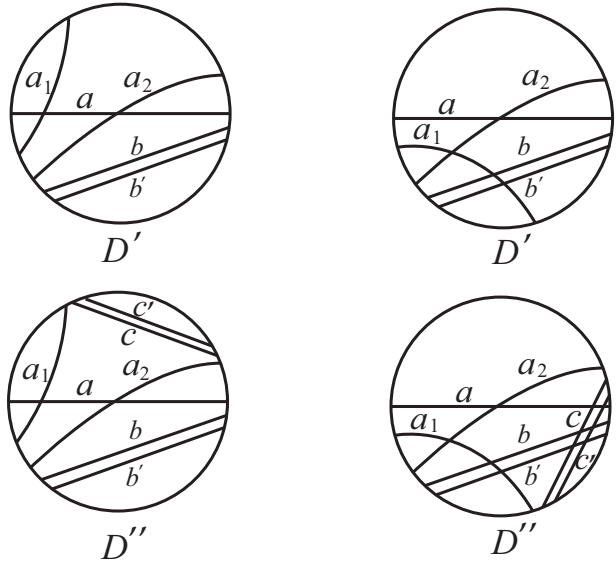} \caption{The
Gaussian parity zero} \label{g_p2}
 \end{figure}

2) If the first two ends belong to the same chord $c$, then
$p_D(c)=0$ (the first Reidemeister move) and $c$ forms the triangle
in $D'$ with the chords $a$ and $b$. Therefore,
$p_{D'}(a)+p_{D'}(b)+p_{D'}(c)=0$. By the induction hypothesis, we
get $p_{D'}(b)=p_{D'}(b')=0$ and $p_D(a)=p_{D'}(b)=0$.
 \end {proof}

 \begin {lemma}\label {lem:ev_od_ae}
Let $p$ be an arbitrary parity  {\em(}with coefficients from a group
$A${\em)} on diagrams of the free knot represented by a chord
diagram $D$. Then for any two chords $a,\,b$ such that
$gp_D(a)=gp_D(b)=1$ we have $p_D(a)=p_D(b)=x\in A$ and $2x=0$.
  \end {lemma}

 \begin {proof}
Let $c_1,\dots,c_k$ be ends of chords lying between the nearest ends
of $a$ and $b$.

Apply $k$ times the second Reidemeister moves as it is shown in
Fig.~\ref {ng_p1} (in the center). Let us show that
$p_{D'}(d_l)=(-1)^lx$, where $x=p_{D'}(a)$. Apply the second
Reidemeister move by adding two chords $f,\,f'$ to form the triangle
$ad_1f$. We have
 $$
gp_{D''}(a)=gp_{D''}(d_1)=1\Longrightarrow
gp_{D''}(f)=0\Longrightarrow p_{D''}(f)=0
 $$
 $$
\Longrightarrow p_{D'}(d_1)=p_{D''}(d_1)=-x.
 $$
By the induction we can prove that $p_{D'}(d_l)=(-1)^lx$ and
$p_D(b)=(-1)^{k+1}x$.

Let us apply the third Reidemeister move to the triangle $ad_1f$.
The parity $p$ and the Gaussian parity of the chord $a$ do not
change but the parity of the number of ends of chords between $a$
and $b$ changes. Applying the previous trick we get
$p_{D}(b)=(-1)^kx$, i.e.\ $2x=0$.

  \begin{figure}
\centering\includegraphics[width=350pt]{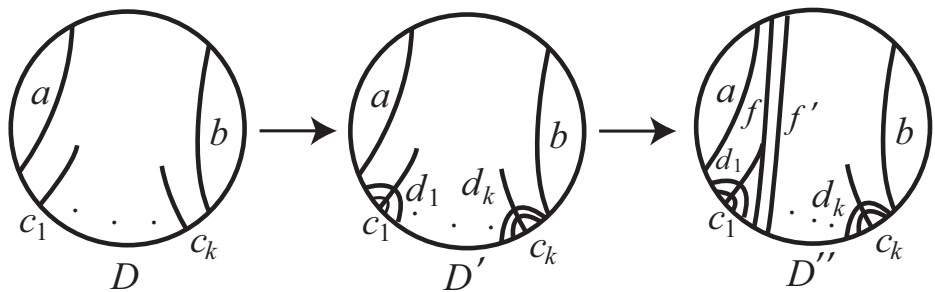} \caption{The
Gaussian parity one} \label{ng_p1}
 \end{figure}
 \end {proof}

By using Lemmas~\ref {lem:gaev_to_ev},~\ref {lem:ev_od_ae} for any
parity $p$ (with coefficients from a group $A$) on diagrams of the
free knot having a diagram $K$ we can construct the homomorphism
$\rho\colon A\to\Z_2$ by taking $\rho(x)=1$, where $p_K(a)=x$ and
$gp_K(a)=1$. This concludes the proof of Theorem~\ref
{thm:uniq_gau}.

 \begin {rk}
Let $p$ be a parity on a free knot $\mathcal{K}$. It is not possible
that there exist two diagrams $K_1$ and $K_2$ of $\mathcal{K}$, both
having chords being odd in the Gaussian parity such that $p$ is
trivial on $K_1$, and $p$ is the Gaussian parity on $K_2$. It
follows from the fact that there is a sequence of Reidemeister moves
transforming $K_1$ to $K_2$ such that any diagram in this sequence
has chords being odd in the Gaussian parity.
 \end {rk}

Before passing to classical knots, we should point out the
following. It is known that classical knot and link theories embed
in virtual knot and link theories~\cite {GPV,KaV}. This means that
if two classical knot (link) diagrams are virtually equivalent then
they are isotopic (classically equivalent).

Nevertheless, the parity axiomatic applied to the classical knot
theory as a part of the  virtual knot theory and to the classical
knot theory as it is, should be treated differently.

Namely, from the above we get the following

 \begin {theorem}
Any parity on virtual knots {\em(}one-component knots, not
links\/{\em)} is trivial on any classical knots.
 \end {theorem}

By itself, it does not guarantee that there is no non-trivial parity
on classical knots: possibly, there might be some which does not
extend to virtual knots? Indeed, for the classical knot theory as it
is we are restricted only to those diagrams having classical
crossings, and some ``additional'' crossing used to prove the above
lemmas can make the diagram classical.

However, the following theorem holds as well.

 \begin {theorem}
For classical knot theory there exists a unique parity --- the
trivial parity.
 \end {theorem}

The proof is indeed a slight modification of Theorem~\ref
{thm:uniq_gau}, which is based on Lemmas~\ref {lem:gaev_to_ev},~\ref
{lem:ev_od_ae}. We just use classical knot diagrams on the plane and
bear in mind Corollary~\ref {st:hom_par}.

\subsection{Homotopy classes of curves generically immersed in a surface}

In the previous subsection we have the situation when all polygons
``are spanned'' by discs on the plane. Now we are interested in
those polygons which are spanned by discs in a surface. As a result
we deal with the homology of the surface.

 \begin{theorem}\label{thm:flat_knots_surf}
Let $\mathcal{K}$ be a homotopy class of curves generically immersed
in a surface $S$. Then the homological parity {\em(}with
coefficients in $H_1(S,\mathbb{Z}_2)/[\mathcal{K}]${\em)} is the
universal parity on curves of $\mathcal{K}$.
 \end{theorem}

 \begin{proof}
We start the  proof of the theorem with the following general
lemmas.

 \begin{lemma}\label{lem:par_mod2}
Let $p$ be a parity, $K$ be a curve on $S$ and $a\in\V(K)$.
Then $2p_K(a)=0$.
 \end{lemma}

 \begin{proof}
By applying the second and third Reidemeister moves we get curves
$K_1$ and $K_2$ (see Fig.~\ref{parity_mod2}). We have the equality
$p_{K_1}(a)+p_{K_1}(b)=0$. Then $p_{K_2}(a)+p_{K_2}(b)=0$. We also
have $p_{K_2}(a)+p_{K_2}(c)+p_{K_2}(d)=0$ and
$p_{K_2}(b)+p_{K_2}(c)+p_{K_2}(d)=0$. Hence, $p_{K_2}(a)=p_{K_2}(b)$
and $2p_{K_2}(a)=0$. Then $2p_{K_1}(a)=0$ and $2p_{K}(a)=0$.

 \begin{figure}
\centering\includegraphics[width=350pt]{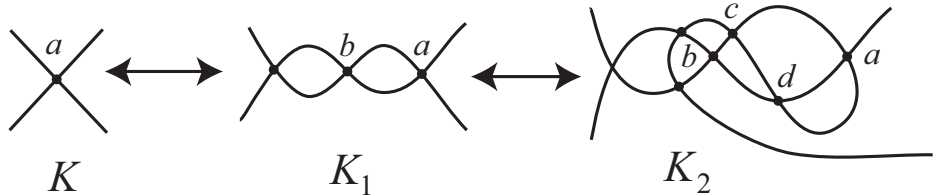} \caption{The
second and third Reidemeister moves} \label{parity_mod2}
 \end{figure}
 \end{proof}

 \begin{lemma}[cf.~\cite {FunMap}]\label{lem:hom_par1}
Let $K$ be a framed 4-graph with one unicursal component. Consider
$K$ as a $1$-dimensional cell complex. Then
$H_1(K,\mathbb{Z}_2)/[K]\cong \bigoplus_{v\in\V(K)}{\mathbb Z}_2$.
 \end{lemma}

 \begin{proof}
Let $C$ be the chord diagram corresponding to $K$. Then $C$ and $K$
are homotopy equivalent as topological spaces. Let $C'$ (resp.,
$K'$) is the topological space obtained by gluing to $C$ (resp.,
$K$) a $2$-disc along the core circle of $C$. Then $C'$ and $K'$ are
homotopy equivalent too and $H_1(C',\Z_2)\cong
H_1(K',\Z_2)=H_1(K,\Z_2)/[K]$. On the other hand, $C'$ is homotopy
equivalent to the bouquet of circles corresponding to the cords of
the diagram $C$, i.e.\ the crossings of $K$. Hence,
$H_1(C',\Z_2)\cong\bigoplus_{v\in\V(K)}\mathbb Z_2$.
 \end{proof}

The isomorphism of the lemma identifies the generator of the group
$\mathbb Z_2$ corresponding to a vertex $v\in\V(K)$ with the
homology class $[K_{v,1}]=[K_{v,2}]\in H_1(K,\mathbb{Z}_2)/[K]$.

 \begin{lemma}\label{lem:hom_par2}
Let $\omega$ be a closed path on the curve $K$ with rotation points
$v_1,\,v_2,\dots, v_k$. Then
$[\omega]=\sum\limits_{i=1}^k[K_{v_i}]\in H_1(K,\mathbb{Z}_2)/[K]$.
 \end{lemma}

 \begin{proof}
By attaching a half $K_{v_i,j}$ for each vertex $v_i$ to the path
$\omega$ we get a closed path without rotation points, i.e.\ a
multiple of $K$. Thus,
 $$
[\omega]+\sum_{i=1}^k[K_{v_i}]=m[K]=0.
 $$
 \end{proof}

Let us return now to the proof of Theorem~\ref
{thm:flat_knots_surf}.

Let $p$ be a parity with coefficients in a group $A$ on curves of a
homotopy class $\mathcal K$ on a closed $2$-surface $S$.

Let $K$ be a curve from $\mathcal{K}$ on the surface $S$. Assume
that $K$ splits the surface into a union of $2$-cells. Arguing as
above in Lemma~\ref{lem:sum_par}, we obtain the following

 \begin{lemma}\label{lem:surf_par}
Let $e$ be a cell in $S\setminus K$ with vertices $v_1,\dots,v_k$
{\em(}not necessarily distinct{\em)}. Then $\sum\limits_{i=1}^k
p_K(v_i)=0$.\hfill $\Box$
 \end{lemma}

Let us show that the map $\rho_K\colon H_1(S,\mathbb{Z}_2)/[\mathcal K]\to
A$ given by the formula $\rho([K_{v,1}])=p_K(v)$, $v\in\V(K)$, is
well defined.

The group $H_1(S,\mathbb{Z}_2)/[K]$ is the first homology group of
the topological space $S'$ obtained from $S$ by gluing a disc along
$K$. $S'$ can also be considered as the result of gluing cells $e\in
S\setminus K$ to the space $K'$ of Lemma~\ref{lem:hom_par1}. Hence,
\begin{multline*}
H_1(S,\mathbb{Z}_2)/[K]=\left(H_1(K',\mathbb{Z}_2)/[K]\right)/([\partial
e],\ e\in S\setminus
K)\\
=\bigoplus_{v\in\V(K)}\Z_2[K_{v,1}]\Big/\left(\sum_{v\in
e\cap\V(K)}[K_{v,1}]=0,\ e\in S\setminus K\right) \\
=\bigoplus_{v\in\V(K)}\Z1_{K,v}\Big/\left(2\cdot1_{K,v}=0,\
v\in\V(K);\ \sum_{v\in e\cap\V(K)}1_{K,v}=0,\ e\in S\setminus
K\right).
\end{multline*}
The second equality follows from
Lemmas~\ref{lem:hom_par1},~\ref{lem:hom_par2}.

On the other hand, due to Lemmas~\ref{lem:par_mod2}
and~\ref{lem:surf_par} we have identities $2p_K(v)=0$, $v\in\V(K)$,
and $\sum\limits_{v\in e\cap\V(K)}p_K(v)=0$, $e\in S\setminus K$,
which imply that the map $\rho$ is well defined epimorphism of
groups.

Let $f\colon K\to K'$ be an elementary morphism (an isotopy or a
Reidemeister move) and the diagram $K'$ splits the surface into
cells. Then for any vertex $v'\in\V(K')$ such that $v'=f_*(v)$ for
some $v\in\V(K)$ we have $[K_{v,1}]=[K'_{v',1}]$ and
$p_K(v)=p_{K'}(v')$. Since the elements $[K'_{v',1}]$ for such
vertices $v'$ generate the group $H_1(S,\mathbb{Z}_2)/[\mathcal K]$
the maps $\rho_K$ and $\rho_{K'}$ coincide. Hence, the map
$\rho=\rho_K$ does not depend on a choice of the diagram $K$ and
$p_K=\rho\circ hp_K$ for any diagram which splits the surface into
cells.

If $S\setminus K$ is not a union of cells, then we can apply second
Reidemeister moves several times and obtain a diagram $K'$ splitting
the surface into cells. By properties of the parities $hp$ and $p$
we have $[K_{v,1}]=[K'_{f_*(v),1}]$ and $p_K(v)=p_{K'}(f_*(v))$ for
any $v\in\V(K)$. Therefore $p_K(v)=p_{K'}(f_*(v))=\rho\circ
hp_{K'}(f_*(v))=\rho\circ hp_K(v)$.

Thus, $p_K=\rho\circ hp_K$ for any diagram $K$, so the homological
parity $hp$ is universal.
\end{proof}

The homological parity remains universal if we pass from the
category of homotopy classes of curves on a given surface $S$ to the
category of knots on $S$ (to be precise, knots in the thickened
surface). The following lemma shows that in some sense parity does
not feel the over- and undercrossing structure.

 \begin{lemma}\label {lem:alt_tr}
Let $p$ be a parity on the category of knots on a surface
$S$, and let $K$ be a diagram of a knot on $S$. If vertices
$a,\,b\in\V(K)$ form a bigon in $S$ then $p_K(a)+p_K(b)=0$. If
vertices $v_1,\,v_2,\,v_3\in\V(K)$ form a triangle in $S$ then
$p_K(a)+p_K(b)+p_K(c)=0$.
 \end{lemma}

 \begin{proof}
We prove the lemma for a triangle, the proof for a bigon is
analogous. Let the vertices $a,\,b,\,c\in\V(K)$ form a triangle. If
one can apply the third Reidemeister move to the triangle, the
identity $p_K(a)+p_K(b)+p_K(c)=0$ follows from definition of parity.
Otherwise the vertices constitute an alternating triangle. By
applying three second and one third Reidemeister moves we get the
diagram $K'$ (see Fig.~\ref{alt_triang}), where the following
equalities hold:
 \begin{gather*}
p_{K'}(b)+p_{K'}(c)+p_{K'}(d)=0,\\
p_{K'}(e)+p_{K'}(f)+p_{K'}(g)=0,\\
p_{K'}(a)+p_{K'}(f)+p_{K'}(g)=0,\\
p_{K'}(e)+p_{K'}(d)=0.
\end{gather*}
Then we have $p_{K'}(a)=p_{K'}(e)=p_{K'}(d)=p_{K'}(b)+p_{K'}(c)$ (we
do not need signs because Lemma~\ref{lem:par_mod2} remains true in
the category of knots). Therefore, $p_K(a)+p_K(b)+p_K(c)=0$.
\end{proof}

 \begin{figure}
\centering\includegraphics[width=350pt]{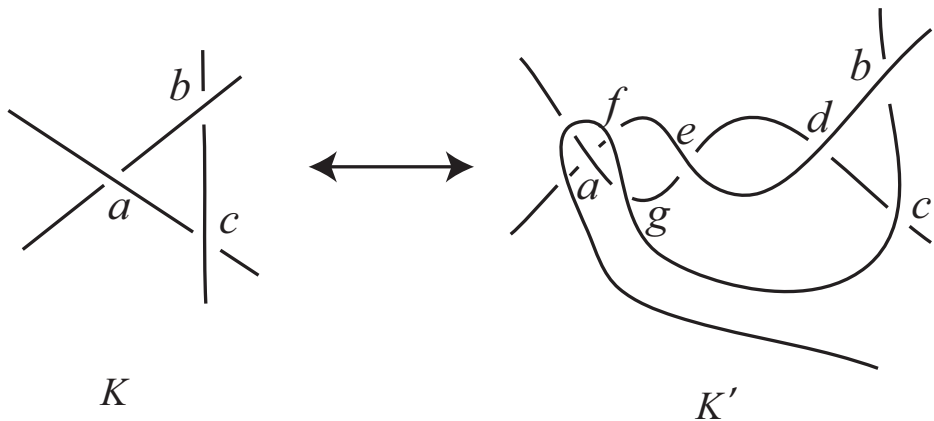} \caption{An
alternating triangle} \label{alt_triang}
 \end{figure}

The claim above ensures that Lemma~\ref{lem:surf_par} holds in the
current situation too. Hence, one can repeat the proof of
Theorem~\ref{thm:flat_knots_surf} and get the following result.

 \begin{theorem}\label{thm:knots_surf}
Let $\mathcal{K}$ be a knot on a surface $S$. Then the homological
parity {\em(}with coefficients in
$H_1(S,\mathbb{Z}_2)/[\mathcal{K}]${\em)} is the universal parity on
diagrams of $\mathcal{K}$.
 \end{theorem}

 \begin{crl}
Any parity on classical knots is trivial.
 \end{crl}

 \begin{proof}
Any classical knot $\mathcal K$ is represented by diagrams on $S^2$.
But $H_1(S^2,\Z_2)=0$, so the universal parity group as well as any
parity is trivial.
 \end{proof}

 \section{Applications of parity}

Let us briefly summarize some theorems from~\cite{Sbornik}
reformulating them for parities with coefficients from an abelian
group.

 \subsection{The functorial mapping $f$}

Let $\mathcal K$ be a virtual, flat or free knot and $\mathfrak K$
be the corresponding category of its diagrams.

Let us consider any family of maps $\widetilde
p_K\colon\V(K)\to\Z_2$, $K\in\ob(\mathfrak K)$, that possesses all
the properties of Definition~\ref {sec:def_par}.\ref{def:parity}
except for the property 3. Instead of it we impose the condition: if
$v_1,\,v_2,\,v_3$ are crossings participating in a third
Reidemeister move then the number of vertices $v$ among
$v_1,\,v_2,\,v_3$ such that $p_K(v)=1$ is not equal to $1$. We call
such a family a {\em pseudoparity $\widetilde p$ of $\mathcal K$
with coefficients in $\Z_2$}.

The following statement follows directly from the definition.

 \begin{lemma}\label{lem:pseudopar}
If $p$ is a parity {\em(}with coefficients in a group $A${\em)},
then the formula
 $$
\widetilde p_K(v)=\left\{\begin{array}{cl}1,& p_K(v)\ne 0,\\
0,& p_K(v)=0\end{array}\right.
 $$
defines a pseudoparity on $\mathcal K$.
 \end{lemma}

Let $\widetilde p$ be a pseudoparity on a knot $\mathcal K$ and $K$
be a diagram of $\mathcal K$. We call a classical crossing $v$ of
$K$ an {\em odd crossing} if $\widetilde p_K(v)=1$ and an {\em even
crossing} if $\tilde p_K(v)=0$. Let $f_{\widetilde p}(K)$ be the
diagram obtained from $K$ by making all odd crossings virtual. In
other words, we remove all odd chords of the corresponding chord
diagram.

 \begin{theorem}
The map $f_{\widetilde p}$ defines a functor from the category of
diagrams of a virtual {\em(}resp., flat, free\/{\em)} knot $\mathcal
K$ with the pseudoparity $\widetilde p$ to the category of diagrams
of the virtual {\em(}resp., flat, free\/{\em)} knot $\mathcal
K'=f_{\widetilde p}(\mathcal K)$.
 \end{theorem}

 \begin{proof}
The map $f_{\widetilde p}$ determines how a functor should act on
objects of the category $\mathfrak K$. We need to show that for any
elementary morphism $h\colon K_1\to K_2$ between two diagrams of
$\mathcal K$ there exists an elementary morphism $f_{\widetilde
p}(h)$ connecting the diagrams $f_{\widetilde p}(K_1)$ and
$f_{\widetilde p}(K_1)$.

If $h$ is an isotopy, then the diagrams $f_{\widetilde p}(K_1)$ and
$f_{\widetilde p}(K_1)$ are isotopic and we can take this isotopy
for $f_{\widetilde p}(h)$. If $h$ is a detour move, the diagrams
$f_{\widetilde p}(K_1)$ and $f_{\widetilde p}(K_1)$ are also related
by a detour move.

If $h$ is a first Reidemeister move and the vertex $v$ of the move
is even, then the diagrams $f_{\widetilde p}(K_1)$ and
$f_{\widetilde p}(K_1)$ differ by a first Reidemeister move. If $v$
is odd, the diagrams are connected by a detour move.

If $h$ is a second Reidemeister move and the vertices $v_1,\,v_2$ of
the move are even, then $f_{\widetilde p}(h)$ is a second
Reidemeister move. If the vertices are odd, then we can connect the
diagrams with a detour move.

If $h$ is a third Reidemeister move then depending on the
(pseudo)parity of the vertices of the move, we can take for the map
$f_{\widetilde p}(h)$ either a third Reidemeister move (if all the
vertices of the move are even) or a detour move (if there are odd
vertices).
 \end{proof}

 \begin {rk}
The mapping ``deleting'' all odd classical crossings is a mapping
into itself, i.e.\ we do not go out from the category. If we had had
a non-trivial parity in the category of classical knots, then we
could have gone out from the category to the category of virtual
knots.
 \end {rk}

 \begin{crl}
For any pseudoparity $\widetilde p$ on $\mathcal K$ the isotopy
class of the diagram $f_{\widetilde p}(K)$ does not depend on the
choice of a diagram $K$ of the knot $\mathcal K$. In other words,
the knot $f_{\widetilde p}(\mathcal K)$ is correctly defined.
 \end{crl}

In the case of the trivial pseudoparity $\widetilde p$ (i.e.\
$\widetilde p_K(v)=0$ for any $v\in\V(K)$) we have $f_{\widetilde
p}(\mathcal K)=\mathcal K$.

As an example showing the power of the notion of parity we present
the following theorem.

 \begin{theorem}[\cite {FrKn}]
Let $K$ be a framed 4-graph with one unicursal component such that
all vertices of $K$ are odd and no decreasing second Reidemeister
move can be applied to $K$. Then $K$ is a minimal diagram of the
corresponding free knot in the following strong sense\/{\em:} for
any diagram $K'$ equivalent to $K$ there is a smoothing of $K'$
isomorphic to the graph $K$.\label{mnm}
 \end{theorem}

 \subsection{The Parity Bracket}

A particular case of the parity bracket firstly appeared in~\cite
{FrKn}. That bracket was constructed for the Gaussian parity and
played a significant role in proving minimality theorems. Also the
bracket was generalised for the case of graph-links, see~\cite
{IM2}, and allowed the authors to prove the existence of
non-realisable graph-links, for more details see~\cite {IM2}.

In this subsection we consider the parity bracket for any parity
valued in $\mathbb{Z}_2$. This bracket is a generalisation of the
bracket from~\cite {FrKn}.

Let ${\mathfrak{G}}$ be the set of all equivalence classes of framed
graphs with one unicursal component modulo second Reidemeister
moves. Consider the linear space $\Z_2\mathfrak{G}$.

Let $\mathcal K$ be a virtual (resp., flat, free) knot, $p$ be a
parity on diagrams of $\mathcal K$ with coefficients from the group
$\Z_2$, and $K$ be a diagram of $\mathcal K$ with
$\V(K)=\{v_1,\dots,v_n\}$. For each element $s\in\{0,1\}^n$ we
define $K_s$ to be equal to the sum of all graphs obtained from $K$
by a smoothing at each vertex $v_i$ if $s_i=1$. If $|s|=l$, $K_s$
contains $2^l$ summands. Define $q_{K,s}(v_i)=p_K(v_i)$ if $s_i=0$,
and $q_{K,s}(v_i)=1-p_K(v_i)$ if $s_i=1$.

Consider the following sum (the {\em parity bracket})
 $$
[K]=\sum\limits_{s\in\{0,1\}^n}\prod\limits_{i=1}^nq_{K,s}(v_i)
K_s\in\Z_2\mathfrak{G},
 $$
where only those summands with one unicursal component are taken
into account.

 \begin {theorem}
If $K$ and $K'$ represent the same knot then the following equality
holds in $\Z_2\mathfrak{G}${\em:} $[K]=[K']$.
 \end {theorem}

 \begin {proof}
Let us check the invariance $[K]\in\Z_2\mathfrak{G}$ under the three
Reidemeister moves.

1) Let $K'$ differ from $K$ by a first Reidemeister move, and
$\V(K')=\{v_1,v_2,\dots,v_{n+1}\}$, $\V(K)=\{v_1,v_2,\dots,v_n\}$.
We have $p_{K'}(v_{n+1})=0$ and
 $$
\left[K'\right]=\left[\firstn\right]=\sum\limits_{s\in\{0,1\}^{n+1}}\prod\limits_{i=1}^{n+1}q_{K',s}(v_i)
K'_s
 $$
 $$
=\sum\limits_{s\in\{0,1\}^n}\prod\limits_{i=1}^nq_{K',s}(v_i)
\left(p_{K'}(v_{n+1})\firstn+(1-p_{K'}(v_{n+1}))\left(\firstfio+\firstfit\right)\right)
 $$
 $$
=\sum\limits_{s\in\{0,1\}^n}\prod\limits_{i=1}^nq_{K',s}(v_i)\firstfio=[K].
 $$

2) Let $K'$ be obtained from $K$ by a second Reidemeister move
adding two vertices, where
$\V(K')=\{v_1,v_2,\dots,v_{n+1},v_{n+2}\}$ and
$\V(K)=\{v_1,v_2,\dots,v_n\}$. We have
$p_{K'}(v_{n+1})+p_{K'}(v_{n+2})=0$, i.e.\
$p_{K'}(v_{n+1})=p_{K'}(v_{n+2})=0$ or
$p_{K'}(v_{n+1})=p_{K'}(v_{n+2})=1$,  and
 $$
[K']=\left[\sectwon\right]=\sum\limits_{s\in\{0,1\}^{n+2}}
\prod\limits_{i=1}^{n+2}q_{K',s}(v_i)
K'_s
 $$
 $$
=\sum\limits_{s\in\{0,1\}^n}\prod\limits_{i=1}^nq_{K',s}(v_i)
\left(p_{K'}(v_{n+1})p_{K'}(v_{n+2})\sectwon\right.
 $$
 $$
+p_{K'}(v_{n+1})(1-p_{K'}(v_{n+2}))\left(\secfino+\secfint\right)
+(1-p_{K'}(v_{n+1}))p_{K'}(v_{n+2})\left(\secsecno+\secsecnt\right)
 $$
 $$
+\left.(1-p_{K'}(v_{n+1}))(1-p_{K'}(v_{n+2}))
\left(\secfioseco+\secfiosect+\secfitseco+\secfitsect\right)\right)
 $$
 $$
=\sum\limits_{s\in\{0,1\}^n}\prod\limits_{i=1}^nq_{K',s}(v_i)\secfioseco=[K].
 $$

 \begin{figure}
\centering\includegraphics[width=350pt]{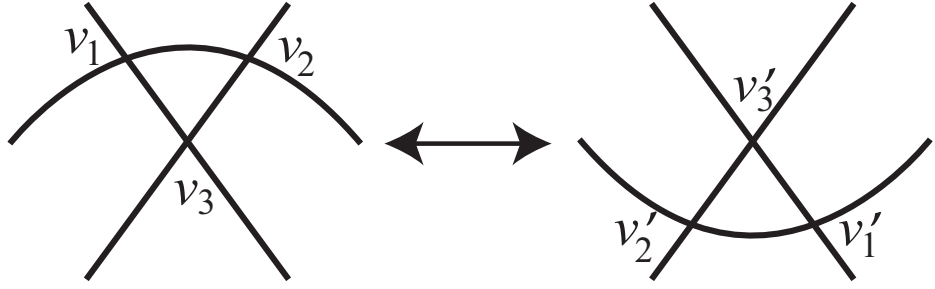} \caption{A third
Reidemeister move} \label{thr}
 \end{figure}

3) Let $K'$ be obtained from $K$ by a third Reidemeister move
applied to vertices $v_1,\,v_2,\,v_3$ in $K$. Denote by
$v'_1,\,v'_2,\,v'_3\in\V(K')$ the vertices corresponding to
$v_1,\,v_2,\,v_3$, see Fig.~\ref {thr} (here
$\V(K')=\{v_1,v_2,\dots,v_n\}$ and
$\V(K)=\{v'_1,v'_2,\dots,v'_n\}$). We have
$p_{K}(v_1)+p_{K}(v_2)+p_{K}(v_3)=0$,
$p_{K'}(v'_1)+p_{K'}(v'_2)+p_{K'}(v'_3)=0$, and
 $$
[K]=\left[\thirdn\right]=\sum\limits_{s\in\{0,1\}^{n}}
\prod\limits_{i=1}^{n}q_{K,s}(v_i)K_s
 $$
 $$
=\sum\limits_{s\in\{0,1\}^{n-3}}
\prod\limits_{i=4}^{n}q_{K,s}(v_i)\left(\underbrace{p_K(v_1)p_K(v_2)p_K(v_3)}_{=0}\thirdn\right.
 $$
 $$
+p_K(v_1)p_K(v_2)(1-p_K(v_3))\left(\thirdfinsecntho+\thirdfinsecntht\right)
 $$
 $$
+(1-p_K(v_1))p_K(v_2)p_K(v_3)\left(\thirdfiosecnthn+\thirdfitsecnthn\right)
 $$
 $$
+p_K(v_1)(1-p_K(v_2))p_K(v_3)\left(\thirdfinsecothn+\thirdfinsectthn\right)
 $$
 $$
+\underbrace{(1-p_K(v_1))(1-p_K(v_2))p_K(v_3)}_{=0}\left(\thirdfiosecothn+\thirdfiosectthn+
\thirdfitsecothn+\thirdfitsectthn\right)
 $$
 $$
+\underbrace{(1-p_K(v_1))p_K(v_2)(1-p_K(v_3))}_{=0}\left(\thirdfiosecntho+\thirdfiosecntht+
\thirdfitsecntho+\thirdfitsecntht\right)
 $$
 $$
+\underbrace{p_K(v_1)(1-p_K(v_2))(1-p_K(v_3))}_{=0}\left(\thirdfinsecotho+\thirdfinsecotht+
\thirdfinsecttho+\thirdfinsecttht\right)
 $$
 $$
+(1-p_K(v_1))(1-p_K(v_2))(1-p_K(v_3))\left(\thirdfiosecotho+\thirdfiosecttho+\thirdfitsecotho
\right.
 $$
 $$
\left.+\thirdfitsecttho+\thirdfitsecttht
+\underbrace{\thirdfiosecotht}_{=0}+\underbrace{\thirdfiosecttht+\thirdfitsecotht}_{=0}\right)
 $$
 $$
=+p_K(v_1)p_K(v_2)(1-p_K(v_3))\left(\thirdfinsecntho+\thirdfinsecntht\right)
 $$
 $$
+(1-p_K(v_1))p_K(v_2)p_K(v_3)\left(\thirdfiosecnthn+\thirdfitsecnthn\right)
 $$
 $$
+p_K(v_1)(1-p_K(v_2))p_K(v_3)\left(\thirdfinsecothn+\thirdfinsectthn\right)
 $$
 $$
+(1-p_K(v_1))(1-p_K(v_2))(1-p_K(v_3))\left(\thirdfiosecotho+\thirdfiosecttho+\thirdfitsecotho
+\thirdfitsecttho+\thirdfitsecttht\right),
 $$
 $$
[K']=\left[\thirdnpr\right]=\sum\limits_{s\in\{0,1\}^{n}}
\prod\limits_{i=1}^{n}q_{K',s}(v'_i)K'_s
 $$
 $$
=\sum\limits_{s\in\{0,1\}^{n-3}}
\prod\limits_{i=4}^{n}q_{K',s}(v'_i)
\left(\underbrace{p_{K'}(v'_1)p_{K'}(v'_2)p_{K'}(v'_3)}_{=0}\thirdnpr\right.
 $$
 $$
+p_{K'}(v'_1)p_{K'}(v_2)(1-p_{K'}(v'_3))\left(\thirdfinsecnthopr+\thirdfinsecnthtpr\right)
 $$
 $$
+(1-p_{K'}(v'_1))p_{K'}(v'_2)p_{K'}(v'_3)\left(\thirdfinsecothnpr+\thirdfinsectthnpr\right)
 $$
 $$
+p_{K'}(v'_1)(1-p_{K'}(v'_2))p_{K'}(v'_3)\left(\thirdfiosecnthnpr+\thirdfitsecnthnpr\right)
 $$
 $$
+\underbrace{(1-p_{K'}(v'_1))(1-p_{K'}(v'_2))p_{K'}(v'_3)}_{=0}
\left(\thirdfiosecothnpr+\thirdfiosectthnpr+
\thirdfitsecothnpr+\thirdfitsectthnpr\right)
 $$
 $$
+\underbrace{(1-p_{K'}(v'_1))p_{K'}(v'_2)(1-p_{K'}(v'_3))}_{=0}\left(\thirdfinsecothopr+\thirdfinsecothtpr+
\thirdfinsectthopr+\thirdfinsectthtpr\right)
 $$
 $$
+\underbrace{p_{K'}(v'_1)(1-p_{K'}(v'_2))(1-p_{K'}(v'_3))}_{=0}\left(\thirdfiosecnthopr+\thirdfiosecnthtpr+
\thirdfitsecnthopr+\thirdfitsecnthtpr\right)
 $$
 $$
+(1-p_{K'}(v'_1))(1-p_{K'}(v'_2))(1-p_{K'}(v'_3))\left(\thirdfitsectthopr+
\thirdfitsecothopr+\thirdfiosectthopr\right.
 $$
 $$
\left.+\thirdfiosecothopr+\thirdfitsectthtpr+\underbrace{\thirdfiosecothtpr}_{=0}
+\underbrace{\thirdfiosectthtpr+\thirdfitsecothtpr}_{=0}\right)
 $$
 $$
=p_{K'}(v'_1)p_{K'}(v_2)(1-p_{K'}(v'_3))\left(\thirdfinsecnthopr+\thirdfinsecnthtpr\right)
 $$
 $$
=(1-p_{K'}(v'_1))p_{K'}(v'_2)p_{K'}(v'_3)\left(\thirdfinsecothnpr+\thirdfinsectthnpr\right)
 $$
 $$
+p_{K'}(v'_1)(1-p_{K'}(v'_2))p_{K'}(v'_3)\left(\thirdfiosecnthnpr+\thirdfitsecnthnpr\right)
 $$
 $$
+(1-p_{K'}(v'_1))(1-p_{K'}(v'_2))(1-p_{K'}(v'_3))\left(\thirdfitsectthopr+
\thirdfitsecothopr+\thirdfiosectthopr
+\thirdfiosecothopr+\thirdfitsectthtpr\right).
 $$

As we consider $\Z_2\mathfrak{G}$ (i.e.\ up to second Reidemeister
moves), we have
 $$
\thirdfinsecntho=\thirdfinsecnthopr,\quad\thirdfinsecntht=\thirdfinsecnthtpr,\quad
\thirdfiosecnthn=\thirdfinsecothnpr,\quad\thirdfitsecnthn=\thirdfinsectthnpr,
 $$
 $$
\thirdfinsecothn=\thirdfiosecnthnpr,\quad\thirdfinsectthn=\thirdfitsecnthnpr,\quad
\thirdfiosecotho=\thirdfitsectthopr,\quad\thirdfiosecttho=\thirdfitsecothopr,
 $$
 $$
\thirdfitsecotho=\thirdfiosectthopr,\quad
\thirdfitsecttho=\thirdfiosecothopr,\quad\thirdfitsecttht=\thirdfitsectthtpr.
 $$
Therefore, $[K]=[K']$.
 \end {proof}

\bibliographystyle{amsplain}

\begin{thebibliography}{10}

 \bibitem{Af}
D.\,M.~Afanasiev, Refining virtual knot invariants by means of
parity // {\em Matem.\ Sb.} {\bf 201}:6 (2010), pp.\ 3--18.

 \bibitem{Ca}
J.\,S.~Carter, Closed Curves that never extend to proper maps of
disks // {\em Proc.\ Amer.\ Math.\ Soc.} {\bf 113}:3 (1991), pp.\
879--888.

 \bibitem {CKS}
J.\,S.~Carter, S.~Kamada, and M.~Saito, Stable equivalence of knots
on surfaces and virtual knot cobordisms // {\it J.\ Knot Theory
Ramifications} {\bf 11}:3 (2002), pp.\ 311-–322.

 \bibitem{Gib}
A.~Gibson, Homotopy Invariants of Gauss words,
arXiv:math.GT$\slash$0902.0062.

 \bibitem{GPV}
M.~Goussarov, M.~Polyak, and O.~Viro, Finite type invariants of
classical and virtual knots, {\it Topology} {\bf 39} (2000), pp.\
1045--1068.

 \bibitem{IM1}
D.\,P.~Ilyutko, V.\,O.~Manturov, Introduction to graph-link theory,
{\it Journal of Knot Theory and Its Ramifications} {\bf 18}:6
(2009), pp.\ 791–-823.

 \bibitem{IM3}
D.\,P.~Ilyutko, V.\,O.~Manturov, Graph-links, {\it Doklady
Mathematics} {\bf 80}:2 (2009), pp.\ 739--742 (Original Russian Text
in {\it Doklady Akademii Nauk} {\bf 428}:5 (2009), pp.\ 591-–594).

 \bibitem{IM}
D.\,P.~Ilyutko, V.\,O.~Manturov, Cobordisms of Free Knots // {\it
Doklady Mathematics} {\bf 80}:3 (2009), pp.\ 1--3 (Original Russian
Text in {\it Doklady Akademii Nauk} {\bf 429}:4 (2009), pp.\
439-–441).

 \bibitem{IM2}
D.\,P.~Ilyutko, V.\,O.~Manturov, Graph-links,
arXiv:math.GT$\slash$1001.0384.

 \bibitem{KK}
N.~Kamada and S.~Kamada, Abstract link diagrams and virtual knots //
{\em Journal of Knot Theory and Its Ramifications} {\bf 9}:1 (2000),
pp.\ 93--109.

 \bibitem{KaV}
L.\,H.~Kauffman,  Virtual knot theory, {\em European Journal of
Combinatorics} {\bf 20}:7 (1999), pp.\ 663--690.

 \bibitem{FrKn}
V.\,O.~Manturov, On Free Knots, arXiv:math.GT$\slash$0901.2214 v2.

 \bibitem{FrKnLi}
V.\,O.~Manturov, On Free Knots and Links,
arXiv:math.GT$\slash$0902.0127.

 \bibitem{FrNonInv}
V.\,O.~Manturov, Free Knos are Not Invertible,
arXiv:math.GT$\slash$0909.2230.

 \bibitem{Sbornik}
V.\,O.~Manturov, Parity in Knot Theory // {\it Matem.\ Sb.}, {\bf
201}:5 (2010),  pp.\ 65--110.

 \bibitem{Paritytrieste}
V.\,O.~Manturov, Free Knots and Parity,
arXiv:math.GT$\slash$09125348v.1, to appear in: {\em Proceedings of
the Advanced Summer School on Knot Theory, Trieste}, Series of Knots
and Everything, World Scientific.

 \bibitem{Fintype}
V.\,O.~Manturov, Free Knots, Groups, and Finite-Type Invariants,
arXiv:math.GT$\slash$1004.4325.

 \bibitem{MM1}
V.\,O.~Manturov, O.\,V.~Manturov, Free Knots and Groups // {\it J.\
Knot Theory Ramifications} {\bf 19}:2 (2010), pp.\ 181--186,
arXiv:math.GT$\slash$0912.2694.

 \bibitem{MM2}
V.\,O.~Manturov, O.\,V.~Manturov, Free Knots and Groups // {\it
Doklady Mathematics} {\bf 82}:2 (2010), pp.\ 697--700 (Original
Russian Text in {\it Doklady Akademii Nauk} {\bf 434}:1 (2010), pp.\
25-–28).

 \bibitem{FunMap}
V.\,O.~Manturov, A Fuctorial Map from Knots inThickened Surfaces to
Classical Knots and Generalisations of Parity,
arXiv:math.GT$\slash$1011.4640.

 \bibitem{Flats}
V.\,G.~Turaev, Virtual open strings and their cobordisms, preprint
(2004), arXiv:math.GT$\slash$0311185v5.

 \bibitem{Tu}
V.\,G.~Turaev, Cobordisms of Words, arXiv:math.CO$\slash$0511513v2.

 \bibitem{Tur}
V.\,G.~Turaev, Cobordisms of Knots on Surfaces,
arXiv:math.GT$\slash$0703055v1.


\end{thebibliography}

\end{document}